\documentclass[11pt]{amsart}

\let\oldtocsection=\tocsection
 
\let\oldtocsubsection=\tocsubsection

\renewcommand{\tocsection}[2]{\hspace{0em}\oldtocsection{#1}{#2}}
\renewcommand{\tocsubsection}[2]{\hspace{1em}\oldtocsubsection{#1}{#2}}

\calclayout

\usepackage{enumitem}

\newcommand{\subscript}[2]{$#1#2$}

\setenumerate{nolistsep}
\setitemize{nolistsep}
\usepackage[square, numbers, comma, sort&compress]{natbib}
\usepackage[final]{pdfpages}
\usepackage{graphicx}
\usepackage{esint}
\usepackage{latexsym}
\usepackage{hyperref}
\usepackage{cleveref}
\usepackage{autonum}
\usepackage{mathtools}

\usepackage{amsmath}
\usepackage{amssymb}
\usepackage{mwe}

\usepackage{tikz}
\usetikzlibrary{graphs, angles, quotes,patterns,hobby}
\usepackage{pgfplots}
\pgfplotsset{compat=1.6}

\DeclareMathOperator{\dist}{dist}

\DeclareMathOperator{\supp}{supp}

\DeclareMathOperator{\pv}{p.v.}

\DeclareMathOperator{\E}{\mathcal{E}}
\DeclareMathOperator{\M}{\mathcal{M}}
\DeclareMathOperator{\Hcal}{\mathcal{H}}

\DeclareMathOperator{\diver}{div}
\DeclareMathOperator{\diam}{diam}
\DeclareMathOperator{\Div}{div}
\DeclareMathOperator{\Lip}{Lip}
\DeclareMathOperator{\D}{\mathcal{D}}
\DeclareMathOperator{\ext}{ext}
\DeclareMathOperator{\db}{db}
\DeclareMathOperator{\inter}{int}
\DeclareMathOperator{\n1}{\nabla_1}
\DeclareMathOperator{\Rn1}{\mathbb{R}^{n+1}}
\DeclareMathOperator{\R}{\mathbb{R}}

\DeclareMathOperator{\LD}{\mathsf{LD}}
\DeclareMathOperator{\Stop}{\mathsf{Stop}}
\DeclareMathOperator{\Bad}{\mathsf{Bad}}
\DeclareMathOperator{\ve}{\varepsilon}

\DeclarePairedDelimiter{\norm}{\lVert}{\rVert}

\renewcommand{\epsilon}{\varepsilon}
\renewcommand{\tilde}{\widetilde}

\theoremstyle{plain}
\newtheorem{theor}{Theorem}[section] 

\theoremstyle{plain}

\theoremstyle{plain}
\newtheorem*{theor*}{Theorem}

\theoremstyle{plain}

\theoremstyle{remark}
\newtheorem{rem}{Remark}[section] 

\theoremstyle{definition}
\newtheorem{defin}{Definition}[section] 

\theoremstyle{plain}
\newtheorem{coroll}{Corollary}[section] 

\theoremstyle{remark}

\theoremstyle{plain}
\newtheorem{lemm}{Lemma}[section] 

\numberwithin{equation}{section}

\usepackage{hyperref}
\hypersetup{
    colorlinks=true,
    linkcolor=black,
    filecolor=black,      
    urlcolor=blue,
    citecolor=red,
}

\begin{document}


\title[Single layer potentials and rectifiability for general measures]{Gradient of the single layer potential and quantitative rectifiability for general Radon measures}

\author{Carmelo Puliatti}
\address{BGSMath and Departament de Matem\`atiques, Universitat Aut\`onoma
de Barcelona, 08193, Bellaterra, Barcelona, Catalonia}
\email{puliatti@mat.uab.cat}
\thanks{
The author was partially supported by 2017-SGR-0395 (Catalonia), MTM-2016-77635-P (MICINN, Spain) and MDM-2014-0445  (MINECO through the Mar\'ia de Maeztu Programme for Units of Excellence in R\&D, Spain).
}

\begin{abstract}
We identify a set of sufficient local conditions under which a significant portion of a Radon measure $\mu$ on $\Rn1$ with compact support can be covered by an $n$-uniformly rectifiable set at the level of a ball $B\subset \Rn1$ such that $\mu(B)\approx r(B)^n$. This result involves a flatness condition, formulated in terms of the so-called $\beta_1$-number of $B$, and  the $L^2(\mu|_B)$-boundedness, as well as a control on the mean oscillation on the ball, of the operator
\begin{equation}
T_\mu f(x)=\int \nabla_x\mathcal{E}(x,y)f(y)\,d\mu(y).
\end{equation}
Here $\mathcal{E}(\cdot,\cdot)$ is the fundamental solution for a uniformly elliptic operator in divergence form associated with an $(n+1)\times(n+1)$ matrix with H\"older continuous coefficients. This generalizes a work by Girela-Sarri\'on and Tolsa for the $n$-Riesz transform. The motivation for our result stems from a two-phase problem for the elliptic harmonic measure.
\end{abstract}
 \maketitle
 \tableofcontents
\section{Introduction}
In the recent work \cite{PPT}, Laura Prat, Xavier Tolsa and the author dealt with the connection between rectifiability and the boundedness of the gradient of the single layer potential. This operator plays a central role in the study of partial differential equations.  Our goal is to investigate the nature of the gradient of the single layer potential
for certain elliptic operators and apply the results to the study of elliptic measure.

An elliptic equivalent of the so-called David-Semmes problem in codimension 1 was considered in \cite{PPT}, under the assumption of H\"older continuity of the coefficients of the uniformly elliptic matrix defining a differential operator in divergence form. The case of the codimension 1 Riesz transform was studied in the deep works of Mattila, Melnikov and Verdera in the plane and by Nazarov, Tolsa and Volberg for higher dimensions (see \cite{MMV} and \cite{NTV_acta}). We remark that the David-Semmes problem for higher codimensions is still unsolved.

In the same spirit of \cite{PPT}, the aim of the present article is to establish an elliptic equivalent of a quantitative rectifiability theorem that Girela-Sarri\'on and Tolsa proved for the Riesz transfom in \cite{GT}.

Let $\mu$ be a Radon measure on $\Rn1$. Its associated $n$-dimensional Riesz transform is
\begin{equation}
\mathcal{R}^n_\mu f(x)=\int\frac{x-y}{|x-y|^{n+1}}f(y)\,d\mu(y), \,\,\qquad f\in L^1_{ \rm loc}(\mu),
\end{equation}
whenever the integral makes sense.
Given $x\in\mathbb{R}^{n+1}$ and $r>0,$ we denote by $B(x,r)$ the open ball of center $x$ and radius $r$.
A Radon measure $\mu$ has growth of degree $n$ if there exists a constant $C>0$ such that
\begin{equation}\label{n_growth}
\mu\big(B(x,r)\big)\leq C r^n\, \qquad \text{ for all }x\in\Rn1, \,\,r>0.
\end{equation}
We call $\mu$ $n$-Ahlfors-David regular (also abbreviated by $n$-AD-regular or just AD-regular) if there exists some constant $C>0$, also referred to as an AD-regularity constant, such that
\begin{equation}
C^{-1}r^n\leq \mu\big(B(x,r)\big)\leq Cr^n \,\qquad \text{ for all }x\in\supp\mu, \,0<r<\diam(\supp \mu).
\end{equation}
A set $E\subset \Rn1$ is said $n$-AD-regular if $\mathcal H^n|_E$ is a $n$-AD-regular measure, $\mathcal H^n$ denoting the $n$-dimensional Hausdorff measure in $\Rn1$.  Note that the support of an $n$-AD-regular measure is $n$-AD-regular.
	
A set $E\subset \Rn1$ is called $n$-rectifiable if there exists a countable family of Lipschitz functions $f_j\colon \mathbb{R}^n\to \Rn1$ such that
\begin{equation}
\mathcal H^n\Big(E\setminus \bigcup_jf_j(\mathbb{R}^n) \Big)=0.
\end{equation}
A measure $\mu$ is rectifiable if it vanishes outside a rectifiable set $E$ and, moreover, it is absolutely continuous with respect to $\mathcal H^n|_E$.

David and Semmes introduced the quantitative version of the notion of rectifiability, which is important because of its relations with singular integrals.
A set $E$ is called $n$-uniformly rectifiable (or just uniformly rectifiable) if it is $n$-AD regular and there exist $\theta,M>0$ such that for all $x\in E$ and all $r>0$ there is a Lipschitz mapping $g$ from the ball $B_n(0,r)\subset\mathbb R^n$ to $\Rn1$ with $\Lip(g)\leq M$ such that
\begin{equation}
\mathcal H^n\big(E\cap B(x,r)\cap g(B_n(0,r))\big)\geq \theta r^n.
\end{equation}
We say that a measure $\mu$ is $n$-uniformly rectifiable if it is $n$-AD-regular and it vanishes out of a $n$-uniformly rectifiable set.

Many characterizations of uniformly rectifiable measures are present in the literature. In particular, if the measure is $n$-AD-regular, then it is $n$-uniformly rectifiable if and only if its associated $n$-Riesz transform is bounded on $L^2$ (see \cite{DS}, \cite{MMV} and \cite{NTV_acta}).

This fact plays a crucial role in the study of the geometric properties of harmonic measure. In particular, it was used in \cite{seven_authors} to prove that the mutual absolute continuity of the the harmonic measure for an open set $\Omega\subset\Rn1$ with respect to surface measure $\mathcal H^n$ in a subset of $\partial\Omega$ implies the $n$-rectifiability of that subset. This answered a problem raised by Bishop (see \cite{bishop}).

The analogous result for elliptic measure has been proved in \cite{PPT}, following the ideas of \cite{seven_authors},  as an application of the characterization of uniform rectifiability via the boundedness of the gradient of single layer potential.

Another question proposed by Bishop asks whether, given two disjoint domains $\Omega_1,\Omega_2\subset\Rn1$, mutual absolute continuity of their respective harmonic measures implies absolute continuity with respect to surface
measure  in $\partial\Omega_1\cap\partial\Omega_2$ and rectifiability.

This is a so-called \textit{two phase problem} for harmonic measure and was eventually solved in its full generality in \cite{AMTV}. This work relies on three main tools: a blow-up argument for harmonic measure (see also \cite{KPT} and \cite{TV}), a monotonicity formula (\cite{acf}) and a quantitative rectifiability criterion (see \cite{GT}).

In particular, we point out that the theorem by Girela-Sarri\'on and Tolsa served to overcome some intrinsic technical  issue in the formulation of the problem and it can be interpreted as an adapted version of previous results by David and L\'eger,  which were formulated in terms of the so-called Menger curvature of a measure (see \cite{david} and \cite{leger}). Their theorem is of fundamental importance also in other two-phase problems examined in \cite{AMT2} and the very recent work \cite{PT}.
The goal of the present paper is to encounter an analogue criterion in the context of elliptic PDE's in divergence form with H\"older coefficients.

Let $A=(a_{ij})_{1\leq i,j \leq n+1}$ be an $(n+1)\times (n+1)$ matrix whose entries $a_{ij}\colon\R^{n+1} \to \R$  are measurable functions in $L^\infty(\R^{n+1})$. Assume also that there exists $\Lambda>0$ such that
\begin{align}\label{eqelliptic1}
&\Lambda^{-1}|\xi|^2\leq \langle A(x) \xi,\xi\rangle,\quad \mbox{ for all $\xi \in\R^{n+1}$ and a.e. $x\in\R^{n+1}$,}\\
&\langle A(x) \xi,\eta \rangle  \leq\Lambda |\xi| |\eta|, \quad \mbox{ for all $\xi, \eta \in\R^{n+1}$ and a.e. $x\in\R^{n+1}$.} \label{eqelliptic2}
\end{align}
We consider the elliptic equation
\begin{equation}\label{eq:ellipticeq}
L_A u(x)\coloneqq   -\mathrm{div}\left(A(\cdot) \nabla u (\cdot) \right)(x)=0,
\end{equation}
which should be understood in the distributional sense.
We say that a function $u \in W^{1,2}_{\rm loc}(\Omega)$ is a {\it solution} of \eqref{eq:ellipticeq}, or {\it $L_A$-harmonic}, in an open set $\Omega \subset \R^{n+1}$ if
$$
\int A \nabla u \cdot \nabla \varphi = 0, \quad \mbox{ for all $\varphi \in C_c^\infty(\Omega)$.}
$$

We denote by $\mathcal{E}_A(x,y)$, or just by $\mathcal{E}(x,y)$ when the matrix $A$ is clear from the context, the {\it fundamental solution} for $L_A$ in $\R^{n+1}$, so that $L_x \mathcal{E}_A(x,y) = \delta_y$ in the distributional sense, where $\delta_y$ is the Dirac mass at the point $y \in \R^{n+1}$. For a construction of the fundamental solution under the assumptions \eqref{eqelliptic1} and \eqref{eqelliptic2} on the matrix $A$ we refer to \cite{HK}.
Given a measure $\mu$, the function $f(x)=\int \mathcal{E}_A(x,y)\,d\mu(y)$  is usually known as the {\it single layer potential} of $\mu$.
We define
\begin{equation}\label{eq:Kdef}
K(x,y) = \nabla_1 \mathcal{E}_A(x,y),
\end{equation}
the subscript $1$ indicating that we take the gradient with respect to the first variable, and we consider \eqref{eq:Kdef} as the kernel of the singular integral operator
\begin{equation}\label{eq:Tmudef}
T\mu(x) = \int K(x,y) \,d\mu(y),
\end{equation}
for $x$ away from $\mathrm{supp}(\mu)$.   Observe that $T\mu$ is the gradient of the single layer potential of $\mu$.

Given a function $f\in L^1_{loc}(\mu)$,
we set also
\begin{equation}\label{eq:Tfdef}
T_\mu f(x) = T(f\,\mu)(x) = \int K(x,y) f(y)\,d\mu(y),
\end{equation}
and, for $\ve>0$, we consider the $\ve$-truncated version
$$
T_{\ve}\mu(x) = \int_{|x-y|>\ve} K(x,y) \,d\mu(y).
$$
We also write $T_{\mu,\ve} f(x) = T_{\ve} (f\mu)(x)$. We say that the operator $T_\mu$ is bounded on $L^2(\mu)$ if the operators $T_{\mu,\ve}$ are bounded on $L^2(\mu)$ uniformly on $\ve>0$.

In the specific case when $A$ is the identity matrix, $-L_A=\Delta$ and $T$ is the $n$-dimensional Riesz transform up to a dimensional constant factor.
We say that the matrix $A$ is H\"older continuous with exponent $\alpha\in (0,1)$ (or briefly $C^\alpha$ continuous), if there exists $C_h>0$ such that
\begin{equation}\label{eq:Holdercont}
|a_{ij}(x)-a_{ij}(y)| \leq C_h |x-y|^\alpha\quad \mbox{ for all $x,y \in \mathbb{R}^{n+1}$ and $1 \leq i,j \leq n+1$.}
\end{equation}
Under this assumption on the coefficients, the kernel $K(\cdot,\cdot)$ turns out to be locally of Calder\'on-Zygmund type (see Lemma \ref{lemcz} for more details). However we remark that, contrarily to what happens in the case of the kernel of the Riesz transform, in general $K(\cdot,\cdot)$ is neither homogeneous nor antisymmetric (not even locally).

For our applications, it is useful to determine whether $T_{\mu,\varepsilon}f$ converges pointwise $\mu$-almost everywhere for $\varepsilon\to 0$. In case it does, we denote the limit as
\begin{equation}
\pv T_\mu f(x)=\lim_{\varepsilon\to 0}T_{\mu,\varepsilon}f(x)
\end{equation}
and we call it the \textit{principal value} of the integral $T_\mu f(x)$. One can prove the existence of the principal values for general Radon measures with compact support under the additional assumption of $L^2(\mu)$-boundedness of $T_\mu$. In particular, our first result is the following.
\begin{theor}\label{theorem_pv_layer_pot}
Let $\mu$ be a Radon measure on $\Rn1$ with compact support and with growth of degree $n$, i.e. suppose that there is $C>0$ such that
\begin{equation}
\mu(B(x,r))\leq C r^n \,\,\,\,\,\text{ for all } \,\,x\in\Rn1.
\end{equation}
Let $A$ be a matrix that satisfies \eqref{eqelliptic1}, \eqref{eqelliptic2} and \eqref{eq:Holdercont} and assume, moreover, that the gradient of the single layer potential $T_\mu$ associated with $L_A$ is bounded on $L^2(\mu)$. Then:
\begin{enumerate}
\item for $1\leq p<\infty$ and all $f\in L^p(\mu)$, $\pv T_\mu f(x)$ exists for $\mu$-a.e. $x\in\Rn1$;
\item for all $\nu\in M(\Rn1)$, $\pv T\nu(x)$ exists for $\mu$-a.e. $x\in\Rn1$.
\end{enumerate}
\end{theor}
If $A\equiv Id,$ Theorem \ref{theorem_pv_layer_pot} reduces to its analogous for the Riesz transform (see for example \cite[Chapter 8]{tolsa_book}).
In light of this result, in the rest of the paper we will often denote the principal value operator simply as $T\nu$ with a slight abuse of notation.

Given a ball $B=B(x,r)\subset \Rn1,$ we denote by $r(B)$ its radius and, for $a>0,$ by $a B$ its dilation $B(x,a r).$
Multiple notions of density come into play in this paper. For a ball $B$, we denote
\begin{equation}
\Theta_\mu(B)=\frac{\mu(B)}{r(B)^n}
\end{equation}
and, for $\gamma>0$, its smoothened version
\begin{equation}\label{p_density_balls}
P_{\mu, \gamma}(B)\coloneqq  \sum_{j\geq 0} 2^{-j\gamma}\Theta_{\mu}(2^j B).
\end{equation}
We remark that if $\gamma_1\leq\gamma_2,$ then
\begin{equation}
P_{\mu,\gamma_2}(B)=\sum_{j\geq 0}2^{-j\gamma_2}\Theta_\mu(2^j B )\leq \sum_{j\geq 0}2^{-j\gamma_1}\Theta_\mu(2^jB)=P_{\mu,\gamma_1}(B).
\end{equation}
Another notion of density that we need is the pointwise one. In particular, we  denote the upper and lower $n$-densities of $\mu$ at $x$ respectively as
\begin{equation}
\Theta^*_\mu(x)\coloneqq  \limsup_{r\to 0}\frac{\mu\big(B(x,r)\big)}{(2r)^n} \,\, \text{ and } \,\,\Theta_{*,\mu}(x)\coloneqq  \liminf_{r\to 0}\frac{\mu\big(B(x,r)\big)}{(2r)^n}.
\end{equation}
A way to quantify the flatness of a measure at the level of a ball $B$ is in terms of the $\beta$-coefficients. For an $n$-plane $L$ we denote
\begin{equation}
\beta^L_{\mu,1}(B)=\frac{1}{r(B)^n}\int_B\frac{\dist(x,L)}{r(B)}d\mu(x)\qquad\text{ and }\qquad \beta_{\mu,1}(B)=\inf_L \beta^L_{\mu,1}(B),
\end{equation}
the infimum being taken over all hyperplanes in $\Rn1$.
Using a standard notation, given $E\subset\Rn1$ with $\mu(E)>0$ and $f\in L^1_{loc}(\mu)$ we write
\begin{equation}
m_{\mu,E}(f)=\frac{1}{\mu(E)}\int_E f d\mu
\end{equation}
for the mean of $f$ with respect to the measure $\mu$ on the set $E$.
The main result of the paper is the following.
\begin{theor}\label{main_theorem}
Let $n>1$, let $\mu$ be a Radon measure on $\Rn1$ with compact support and consider an open ball $B\subset\Rn1$. Let $C_0, C_1>0$ and let $A$ be a matrix satisfying \eqref{eqelliptic1}, \eqref{eqelliptic2} and \eqref{eq:Holdercont}. Denote by $T_\mu$ the gradient of the single layer potential associated with $L_A$ and $\mu$. Suppose that $\mu$ and $B$ are such that, for some positive $\lambda, \delta$ and $\epsilon$ and some $\tilde{\alpha}\in (0,1),$ the following properties hold
\begin{enumerate}
\item $r(B)\leq\lambda$.
\item $C_0^{-1}r(B)^n\leq \mu(B)\leq C_0 r(B)^n.$
\item $P_{\mu,\tilde{\alpha}}(B)\leq C_0$ and for all $x\in B$ and $0<r\leq r(B)$ we have $\mu\big(B(x,r)\big)\leq C_0 r^n.$
\item $T_{\mu|_B}$ is bounded on $L^2\big(\mu|_B\big)$ with $\|T_{\mu|_B}\|_{L^2(\mu|_B)\to L^2(\mu|_B)}\leq C_1$ and $T\big(\chi_{2B}\mu\big)\in L^2\big(\mu|_B\big)$.
\item $\beta_{\mu,1}(B)\leq \delta.$
\item We have
\begin{equation}
\int_B\big|T\mu(x)-m_{\mu,B}(T\mu)\big|^2d\mu(x)\leq\epsilon\mu(B).
\end{equation}
\end{enumerate}
There exists a choice of $\lambda,\delta$ and $\epsilon$ small enough and a proper choice of $\tilde{\alpha}=\tilde \alpha(\alpha,n)$, all possibly depending on $C_0$ and $C_1$, such that if $\mu$ satisfies (1)$-\cdots-$(6), there exists a $n$-uniformly rectifiable set $\Gamma$ that covers a big portion of the support of $\mu$ inside $B$. That is to say, there exists $\tau>0$ such that
\begin{equation}
\mu(B\cap\Gamma)\geq \tau\mu(B).
\end{equation}
\end{theor}

Notice that Theorem \ref{main_theorem} immediately implies that a big piece of $\mu|_B$ is mutually absolutely continuous with a big piece of $\mathcal H^n|_\Gamma$. This is a relevant feature in light of possible applications, in particular to elliptic measure.

{Our proof of the theorem shows that a good choice for $\tilde{\alpha}$ is $\tilde{\alpha}=\alpha/2^{n+1}$.}
It is not clear whether  Theorem \ref{main_theorem} holds with a condition on $P_{\mu,\alpha}(B),$ that would be a more natural homogeneity to assume.
We remark that the integral in the left hand side of the assumption (6) makes sense because of the existence of principal values ensured by Theorem \ref{theorem_pv_layer_pot} and the hypothesis $P_{\mu,\alpha}(B)<+\infty.$ For a sketch of the argument we refer to the end of Section \ref{section_pv}.

The main conceptual difference with respect to the analogous theorem for the Riesz transform in \cite{GT} is that we need to require the ball $B$ to be small enough.
The locality of our result reflects the non-scale invariant character of the H\"older regularity assumption for the coefficients of the matrix $A$. This issue is evident also in \cite{PPT} and it is not clear how to overcome this difficulty without making further assumptions on the matrix.

Another difference is that we could not formulate the theorem in terms of $P_{\mu, 1}$. The proofs of the rectifiability results for the harmonic measure in \cite{AMT} and \cite{AMTV} actually rely on the fact that the theorem of Girela-Sarri\'on and Tolsa holds for $\tilde\alpha=1$. However, a slight variation on their arguments allows to overcome this technical obstacle. We close
the introduction by presenting an application of Theorem \ref{main_theorem}, which is, in fact, its main motivation.

Before stating it, recall that if $\Omega$ is a Wiener regular set, the elliptic measure $\omega^p_{L_A}$ with pole at $p$ associated with the elliptic operator $L_A$ is the probability measure supported on $\partial\Omega$ such that, for $f\in C_0(\partial\Omega)$, 
\begin{equation}
\int f d \omega^p_{L_A}=\tilde{f}(p),
\end{equation}
where $\tilde{f}$ denotes the $L_A$-harmonic extension of $f$. A large literature is available on the subject. For example, we refer to \cite{HKM} and \cite{Ke} for its definition and basic properties.

\begin{theor}\label{theorem_two_phase_elliptic}
Let $n\geq 2$ and let $A$ be an elliptic matrix satisfying \eqref{eqelliptic1}, \eqref{eqelliptic2} and \eqref{eq:Holdercont}. Let $\Omega_1,\Omega_2\subset\Rn1$ be two Wiener-regular domains and, for $p_i\in\Omega_i$, $i\in\{1,2\},$ let $\omega^{p_i}_{L_A,i}$ be the respective elliptic measures in $\Omega_i$ associated with $L_A$ and with pole $p_i$. Suppose that $E$ is a Borel set such that $\omega^{p_1}_{L_A,1}|_E\ll\omega^{p_2}_{L_A,2}|_E\ll\omega^{p_1}_{L_A,1}|_E.$
Then there exists an $n$-rectifiable set $F\subset E$ with $\omega^{p_1}_{L_A,1}(E\setminus F)=0$ such that $\omega^{p_1}_{L_A,1}|_F$ and $\omega^{p_2}_{L_A,2}|_F$ are mutually absolutely continuous with respect to $\mathcal H^n|_F.$
\end{theor}
We remark that the generalization of the blow-up methods for the harmonic measure to our elliptic context is contained in the work \cite{AM2}. Also, the proof of Theorem \ref{theorem_two_phase_elliptic} follows closely the path of the work \cite{AMTV}. However, some variations are needed so that we decided to sketch the proof at the end of the paper, where we also provide precise references for the reader's convenience.

We finally remark that recently several studies have appeared concerning the connection between the geometry of a domain and the properties of its associated elliptic measure, among which we list \cite{ABHM}, \cite{AGMT}, \cite{HKMP}, \cite{HMiT}, \cite{HMT} and \cite{KKiPT}.

\subsection*{The structure of the work} \textit{Section 2} is devoted to settle our notation and to make an overview of the results in PDE's relevant for our work. In particular, we need some estimate for the gradient of the fundamental solution coming from homogenization theory. 

In \textit{Section 3} we prove Theorem \ref{theorem_pv_layer_pot}.

\textit{Section 4} contains the statement of the Main Lemma that we use to prove Theorem \ref{main_theorem}. The biggest advantage of the formulation of this lemma with respect to the one of the main theorem is that the flatness condition on the $\beta_1$-number is replaced by an hypothesis on the $\alpha$-numbers. The latter are more powerful when trying to transfer the flatness estimates to the integrals.

In \textit{Section 5} we discuss an equivalent formulation of the Main Lemma in terms of an auxiliary elliptic operator which shares more symmetries than $L_A$. This is a novelty of the elliptic case, this issue not being present in the work of Girela-Sarri\'on and Tolsa.

The \textit{Sections 6, 7, 8 and 9} follow the path of the original work for the Riesz transform, with some minor variations. They are necessary for expository reasons; indeed, they present the core of the contradiction argument for the proof of the Main Lemma and the construction of a periodic auxiliary measure.

\textit{Section 10} consists of the proof of two crucial results: the existence of the limit of proper smooth truncates of the potential of bounded periodic functions and a localization estimate for the potential close to a cube. We emphasize that these proofs rely on the periodicity of the modification of the elliptic matrix.

In \textit{Section 11} we complete the proof of the Main Lemma via a variational technique. We highlight that one of the most delicate point consists in finding an appropriate variant of a maximum principle in an infinite strip in our elliptic setting. Our argument heavily exploits the additional symmetries provided by the modified matrix.

In the final \textit{Section 12}, we present the application of the main rectifiability theorem to the study of elliptic measure,  sketching the proof of Theorem \ref{theorem_two_phase_elliptic}.

\subsection*{Acknowledgments} This work is part of my PhD thesis, which was supervised by Xavier Tolsa and Joan Verdera. I want to thank them for their guidance and their help.
I am also grateful to Mihalis Mourgoglou for useful discussions.

\section{Preliminaries and notation}\label{section_preliminaries}
It is useful to write $a\lesssim b$ to denote that there is a constant $C>0$ such that $a\leq C b.$ To make the dependence of the constant on a parameter $t$ explicit, we will write $a\lesssim_t b.$ Also, we say that $b\gtrsim a$ if $a\lesssim b$ and $a\approx b$ if both $a\lesssim b$ and $b\lesssim a.$

{All the cubes, unless specified, will be considered with their sides parallel to the coordinate axes.} Given a cube $Q$, we denote its side length as $\ell(Q)$ and, for $a>0,$ we understand $a Q$ as the cube with side length $a \ell(Q)$ and sharing the center with $Q$.

We say that a cube $Q$ has $t$-thin boundary if
\begin{equation}
\mu\big\{x\in 2Q:\dist(x,\partial Q)\leq \lambda \ell(Q)\big\}\leq t\lambda \mu(2Q)
\end{equation}
for every $\lambda>0.$
Analogously to \eqref{p_density_balls}, we define
\begin{equation}
P_{\mu,\gamma}(Q)=\sum_{j\geq 0}2^{-j\gamma}\Theta_\mu(2^jQ)=\sum_{j\geq 0}2^{-j\gamma}\frac{\mu(2^jQ)}{\ell(2^j Q)^n}.
\end{equation}

Given a measure $\mu$ and a measurable set $E$, we denote as $\mu|_E$ the restriction of $\mu$ to $E$ and, for $\phi\colon\Rn1\to\Rn1$, we use the notation $\phi\sharp\mu (E)\coloneqq \mu(\phi^{-1}(E))$.
An important tool in the study of rectifiabilty is the so-called $\alpha$-number introduced by Tolsa in   \cite{tolsa_proc_lms}. Let us fix a cube $Q\subset \mathbb{R}^{n+1}$ and consider two Radon measures $\mu$ and $\nu$ on $\mathbb{R}^{n+1}$. A natural way to define a distance between $\mu$ and $\nu$ is to consider the supremum
\begin{equation}
d_Q(\mu,\nu)\coloneqq  \sup_f\int f d(\mu-\nu),
\end{equation}
where $f\in\Lip(\mathbb{R}^{n+1}),$ $\|f\|_{\Lip}\leq 1$ and $\supp f\subseteq Q$. The distance $d_Q$ offers a way of quantifying the ``flatness'' of a measure alternative to that via $\beta_1$-numbers. More precisely, if we consider a $n$-plane $L$ in $\mathbb{R}^{n+1},$ we can define
\begin{equation}
\alpha^L_\mu(Q)\coloneqq  \frac{1}{\ell(Q)^{n+1}}\inf_{c\geq 0}d_Q(\mu,c\mathcal{H}^n|_L).
\end{equation}

Given a matrix $A(\cdot)$, possibly with variable coefficients, we use the notation $A^T(\cdot)$ to indicate its transpose.
Also, we write $\mathcal{L}^{n+1}$ for the Lebesgue measure on $\Rn1$.

\vspace{2mm}
\noindent{\bf Partial Differential Equations.} For any uniformly elliptic matrix $A$ with H\"older continuous coefficients, one can show that
$K(x,y)=\nabla_1 \E(x,y)$ is locally a Calder\'on-Zygmund kernel.

\begin{lemm}\label{lemcz}
Let $A$ be an elliptic matrix with H\"older continuous coefficients satisfying \eqref{eqelliptic1}, \eqref{eqelliptic2} and \eqref{eq:Holdercont}. If $K(\cdot,\cdot)$ is given by \eqref{eq:Kdef}, then it is locally a Calder\'on-Zygmund kernel. That is, for any given $R>0$,
\begin{itemize}
\item[(a)] $|K(x,y)| \lesssim |x-y|^{-n}$ for all $x,y\in \R^{n+1}$ with $x \not=y$ and $|x-y|\leq R$.
\item[(b)] $|K(x,y)-K(x,y')| + |K(y,x) - K(y',x)| \lesssim |y-y'|^{\alpha} |x-y|^{-n-\alpha}$ for all $y,y'\in B(x,R)$ with $2|y-y'| \leq |x-y|$.
\item[(c)] $|K(x,y)| \lesssim |x-y|^{{(1-n)}/2}$  for all $x,y\in \R^{n+1}$ with $|x-y|\geq 1$.
\end{itemize}
All the implicit constants in (a), (b) and (c) depend on $\Lambda$ and  $\|A\|_\alpha$, while the ones in (a) and (b) depend also on $R$.
\end{lemm}

The statements above are rather standard. For more details, see Lemma 2.1 from \cite{CMT}.

Let $\omega_{n}$ denote the surface measure of the unit sphere of $\mathbb{R}^{n+1}$.
For any elliptic matrix $A_0$ with constant coefficients, we have an explicit expression for the fundamental solution of $L_{A_0}$, which we denote by $\Theta(x,y;A_0).$ More precisely, $\Theta(x,y;A_0)=\Theta(x-y;A_0)$ with
\begin{equation}\label{fund_sol_const_matrix}
\Theta(z;A_0)=\Theta(z;A_{0,s})=
\begin{cases}\displaystyle
\frac{-1}{(n-1)\omega_n\sqrt{\det A_{0,s}}}\frac{1}{(A_{0,s}^{-1}z\cdot z)^{(n-1)/2}} \;\;\;\text{ for }n\geq 2, \\\\
\displaystyle
\frac{1}{4\pi\sqrt{\det A_{0,s}}}\log\big(A_{0,s}^{-1}z\cdot z\big)\;\;\text{ for }n=1,
\end{cases}
\end{equation}
where $A_{0,s}$ is the symmetric part of $A_0$, that is, $A_{0,s}=\frac12(A_0+A_0^T)$.

The reason why only the symmetric part of $A_0$ enters \eqref{fund_sol_const_matrix} it that, using Schwarz's theorem to exchange the order of partial derivatives writing $A_0=\{a_{ij}\}_{i,j}$, for every appropriate function $u$ we have
\begin{equation}\label{a_eq_a_sym}
L_{A_0} u=-\sum_{i,j}\partial_i(a_{ij}\partial_j u)=-\frac{1}{2}\sum_{i,j}a_{ij}\partial_i\partial_j u -\frac{1}{2}\sum_{i,j}a_{ij}\partial_j\partial_i u =-	\sum_{i,j}\frac{a_{ij}+a_{ji}}{2}\partial_i\partial_j u =L_{A_{0,s}}u.
\end{equation}
These formal considerations can be made rigorous by standard arguments.

Differentiating \eqref{fund_sol_const_matrix} we have
\begin{equation}\label{grad_const_coeff_sol}
\nabla \Theta(z;A_0)=\frac{1}{\omega_n\sqrt{\det A_{0,s}}}\frac{A_{0,s}^{-1}z}{(A_{0,s}^{-1}z\cdot z)^{(n+1)/2}}.
\end{equation}

The next result is proven in \cite[Lemma 2.2]{KS}.

\begin{lemm}\label{lemm_freezing}
Let $A$ be an elliptic matrix with H\"older continuous coefficients satisfying \eqref{eqelliptic1}, \eqref{eqelliptic2} and \eqref{eq:Holdercont}. Let also  $\Theta(\cdot,\cdot; \cdot)$ be given by  \eqref{fund_sol_const_matrix}. Then, for $x,y\in\Rn1$, $0<|x-y|\leq R,$
\begin{enumerate}
\item $|\mathcal{E}_A (x,y) - \Theta(x,y;A(x))| \lesssim |x-y|^{\alpha-n+1}$,
\item $|\nabla_1\mathcal{E}_A (x,y) - \nabla_1\Theta(x,y;A(x))| \lesssim |x-y|^{\alpha-n}$,
\item $|\nabla_1\mathcal{E}_A (x,y) - \nabla_1\Theta(x,y;A(y))| \lesssim |x-y|^{\alpha-n}$.
\end{enumerate}
Similar inequalities hold if we reverse the roles of $x$ and $y$ and we replace $\nabla_1$ by $\nabla_2$.
All the implicit constants depend on $\Lambda$,  $\|A\|_\alpha$, and  $R$.
\end{lemm}

\vspace{2mm}
\noindent{\bf The gradient of the fundamental solution in the periodic case.}
We denote as $\Lambda_\alpha$ the set of matrices such that \eqref{eqelliptic1}, \eqref{eqelliptic2} hold and with $\alpha$-H\"older coefficients. We say that the matrix $A\in \Lambda_\alpha$ is $\ell$-periodic, $\ell >0$, if
\begin{equation}
A(x+ \ell z)=A(x) \,\,\text{ for every }z\in\mathbb{Z}^{n+1}.
\end{equation}
For periodic matrices the estimates in Lemma \ref{lemcz} turn out to be global.
\begin{lemm}[\cite{KS}]\label{CZ_1}
Let $A\in\Lambda_\alpha$ be $1$-periodic and let $\E_A$ be the fundamental solution of $L_A$. Let $K(\cdot,\cdot)$ is given by \eqref{eq:Kdef}. Then
\begin{enumerate}
\item \label{cz1.1}$|\nabla_1\E_A (x,y)|\leq c_1 |x-y|^{-n}$ for every $x,y\in\Rn1$ with $x\neq y$.
\item \label{cz1.2}$|\n1\E_A(x,y)-\n1\E_A(x',y)|+|\n1\E_A(y,x)-\n1\E_A(y,x')|\leq c_2 |x-x'|^{\alpha}|x-y|^{-(n+\alpha)}$ for every $x,x',y\in\Rn1$ such that $2|x-x'|\leq |x-y|.$
\end{enumerate}
The constants appearing in (\ref{cz1.1}) and (\ref{cz1.2}) are such that $c_1\approx_{n,\Lambda} c_2\approx_{n,\Lambda} \|A\|_{\alpha}$.
\end{lemm}

The period of the matrix plays an important role in our construction, so it is useful to rephrase the previous lemma for matrices with a period different from $1$. We are interested in studying matrices with small period, so we only consider  the case in which it is strictly smaller than $1.$
\begin{lemm}\label{CZ_l}
Let $0<\ell<1$. Let $A\in\Lambda_\alpha$ be $\ell$-periodic and let $\E_A$ be the fundamental solution associated with $L_A$. Then 
\begin{enumerate}
\item \label{czl.1}$|\nabla_1\E_A (x,y)|\leq c'_1 |x-y|^{-n}$ for every $x,y\in\Rn1$ with $x\neq y$.
\item \label{czl.2}$|\n1\E_A(x,y)-\n1\E_A(x',y)|+|\n1\E_A(y,x)-\n1\E_A(y,x')|\leq c'_2 |x-x'|^{\alpha}|x-y|^{-n-\alpha}$ for every $x,x',y\in\Rn1$ such that $2|x-x'|\leq |x-y|.$
\end{enumerate}
The constants appearing in (\ref{cz1.1}) and (\ref{cz1.2}) are such that $c'_1\approx_{n,\Lambda} c'_2\approx_{n,\Lambda} \|A\|_\alpha$.
\end{lemm}
\begin{proof}For $\ell\in(0,1)$ and all $x\in\Rn1$ we define the rescaled matrix
\begin{equation}
\tilde{A}(x)\coloneqq   A(\ell x)
\end{equation}
and we denote by $\tilde{\E}$ the fundamental solution of $L_{\tilde{A}}.$ By the definition of fundamental solution, it is not difficult to see that
\begin{equation}\label{scaledA}
\n1\tilde{\E}(x,y)=\ell^n\n1\E_A(\ell x,\ell y)\qquad\text{ for } x,y\in\Rn1.
\end{equation}
 Moreover,
\begin{equation}
|\tilde{A}(x)-\tilde{A}(y)|=|A(\ell x)-A(\ell y)|\leq \ell^\alpha \|A\|_\alpha |x-y|^\alpha\leq \|A\|_\alpha|x-y|^\alpha,
\end{equation}
so that $\|\tilde{A}\|_\alpha\leq \|A\|_\alpha.$
Applying Lemma \ref{CZ_1} together with \eqref{scaledA} we get
\begin{equation}
|\n1\E_A(x,y)|=\ell^{-n}|\n1\tilde{\E}(\ell^{-1}x,\ell^{-1}y)|\lesssim \ell^{-n}|\ell^{-1}x-\ell^{-1}y|^{-n}=|x-y|^{-n}
\end{equation}
for any $x,y$ and
\begin{equation}
\begin{split}
&|\n1\E_A(x,y)-\n1\E_A(x',y)|=\ell^{-n}|\n1\tilde{\E}(\ell^{-1}x,\ell^{-1}y)-\n1\tilde{\E}(\ell^{-1}x',\ell^{-1}y)|\\
&\qquad\lesssim \ell^{-n}\frac{|\ell^{-1}x-\ell^{-1}x'|^{\alpha}}{|\ell^{-1}x-\ell^{-1}y|^{n+\alpha}}=\frac{|x-x'|^{\alpha}}{|x-y|^{n+\alpha}.}
\end{split}
\end{equation}
for $2|x-x'|\leq |x-y|.$ The same estimate holds for $|\n1\E_A(y,x)-\n1\E_A(y,x')|$.
\end{proof}
The following is the (global) analogue of Lemma \ref{lemm_freezing} in the $1$-periodic setting.
\begin{lemm}\label{lemma_freezing_per_1}
Let $A\in\Lambda_\alpha$ be $1$-periodic. Then for every $x,y\in\mathbb{R}^{n+1},$ $x\neq y$, we have
\begin{align}
\big|\E_A(x,y)-\Theta(x,y; A(x))\big|&\lesssim |x-y|^{\alpha-n+1}\\
\big|\nabla_1\E_A(x,y)-\nabla_1\Theta(x,y; A(x))\big|&\lesssim |x-y|^{\alpha-n}\\
\big|\nabla_1\E_A(x,y)-\nabla_1\Theta(x,y; A(y))\big|&\lesssim |x-y|^{\alpha-n},
\end{align}
the implicit constants depending on $\|A\|_\alpha$ and $\Lambda$.
Similar estimates hold if we replace $\nabla_1$ by $\nabla_2.$
\end{lemm}
Let us now recall some result from elliptic homogenization. For more details we refer to the work by Avellaneda and Lin \cite{Avell_Lin}. For this purpose, we need to recall the definition of vector of correctors $\chi$ and homogenized matrix $A_0$.
Let $\ell>0$ and let $A\in\Lambda_\alpha$ be a $1$-periodic matrix, i.e.
\begin{equation}
A(x+z)=A(x) \qquad \text{ for every }\quad z\in\mathbb{Z}^{n+1}.
\end{equation}
We will denote by $\chi(x)=(\chi^i(x)),$ for $i\in\{1,\ldots,n+1\}$ the vector of correctors, which is defined as the solution of the following cell problem
\begin{equation}
\begin{cases}
L\chi= \Div A, \\ \label{corr}
\chi \text{ is }1\text{-periodic}, \\
\int_{[0,1]^{n+1}}\chi(x)dx=0,
\end{cases}
\end{equation}
where the first condition in \eqref{corr} has to be understood in coordinates as
\begin{equation}
\sum_{i,j}\partial_{x^i}\big[a_{ij}\partial_{x^j}\chi^h\big](x)=-\sum_i \partial_{x^i} a_{ih}(x),
\end{equation}
$(a_{ij})_{i,j}$ being the coefficients of the matrix $A$.
An important fact is that that 
\begin{equation}\label{bound_grad_chi_1}
\|\nabla\chi\|_\infty\leq C,
\end{equation} the bound $C$ depending only on $n,\alpha$ and $\|A\|_{C^\alpha}$. We remark that $\nabla \chi$ denotes the matrix with variable coefficients whose entries are $\partial_i\chi^j$ for $i,j=1,\ldots,n+1$.
Now, if we consider the following family of elliptic operators 
\begin{equation}
L_\epsilon\coloneqq  \Div\big(A(x/\varepsilon)\nabla\cdot\big)
\end{equation}
depending on the parameter $\epsilon>0$, it can be proved that for any $f\in L^2(\mathbb{R}^{n+1}),$ the solutions $u_\epsilon\in W^{1,2}(\mathbb{R}^{n+1})$  of
\begin{equation}
L_\epsilon u_\epsilon=\Div f
\end{equation}
converge weakly in $W^{1,2}(\mathbb{R}^{n+1})$ to a function $u_0$ as $\varepsilon\to 0$. This function solves the equation
\begin{equation}
L_0 u_0\coloneqq  \Div(A_0\nabla u_0)=\Div f,
\end{equation}
where $A_0$ is an elliptic matrix with constant coefficients usually called \textit{homogenized matrix} (see, for example,  \cite{Sh}).

Homogenization is a powerful tool to study the fundamental solution of an elliptic equation in divergence form whose associated matrix is periodic and has $C^\alpha$ coefficients. The main result that we will use is the following (see \cite[Lemma 2]{Avell_Lin} and \cite[Lemma 2.5]{KS}).
\begin{lemm}\label{lemm_homog_1_periodic}
Let $A\in\Lambda_\alpha$. Let us assume that $A$ is $1$-periodic. Then there exists $\gamma\in(0,1)$ depending on $\alpha, \|A\|_{C^\alpha}$ and $n$ such that
\begin{equation}\label{approx_hom_fun}
\big|\E_A(x,y)-(Id+\nabla\chi(x))\Theta(x,y;A_0)\big|\lesssim \frac{c}{|x-y|^{n+\gamma-1}}
\end{equation}
and
\begin{equation}\label{approx_hom}
\big|\nabla_1\E_A(x,y)-(Id+\nabla\chi(x))\nabla_1\Theta(x,y;A_0)\big|\lesssim \frac{c}{|x-y|^{n+\gamma}},
\end{equation}
where $Id$ denotes the identity matrix and the implicit constants in \eqref{approx_hom_fun} and \eqref{approx_hom} depend just on $n,\alpha$ and $\|A\|_{\alpha}$.
\end{lemm}
The period of the coefficients of $A$ plays a crucial role in these estimates. We will be dealing with matrices with periodicity different from $1$, so we need a suitably adapted version of the previous lemma.
Let $A\in\Lambda_\alpha$ be a $\ell$-periodic matrix. Let us define the $1$-periodic matrix
\begin{equation}
\tilde{A}(x)\coloneqq   A(\ell x)
\end{equation}
for $x\in\mathbb{R}^{n+1}$ and let $\tilde \chi$ denote the vector of correctors associated with $\tilde{A}$ defined according to \eqref{corr}.
For $\ell>0$ we define
\begin{equation}
\chi_\ell(x)\coloneqq  \ell \,\tilde \chi\Big(\frac {x}{\ell}\Big).
\end{equation}
Observe that, because of \eqref{bound_grad_chi_1} there exists $C>0$ depending on the $n, \alpha$ and $\|A\|_{C^\alpha}$ such that
\begin{equation}\label{bound_grad_chi}
\|\nabla\chi_\ell\|_\infty\leq C.
\end{equation}
\begin{lemm}\label{lemm_estim_freezing_periodic}
Let $0<\ell<1$. Let $A\in\Lambda_\alpha$ be an $\ell$-periodic matrix. Then there exists $\gamma\in(0,1)$ and $c>0$, both depending just on $n,\alpha$ and $\|A\|_{\alpha}$ such that
\begin{align}
\label{freezing_l_1} \big|\nabla_1\E_A(x,y)-\nabla_1\Theta(x,y; A(x))\big|&\leq c \ell^\alpha|x-y|^{\alpha-n},\\ 
\label{freezing_l_2} \big|\nabla_2\E_A(x,y)-\nabla_2\Theta(x,y; A(y))\big|&\leq c\ell^\alpha|x-y|^{\alpha-n}.\\ 
\big|\nabla_1\E_A(x,y)-(Id+\nabla\chi_\ell(x))\nabla_1\Theta(x,y;A_0)\big|&\leq c \ell^\gamma |x-y|^{-n- \gamma	}. \label{approx_hom_par}
\end{align}
for every $x\neq y.$
\end{lemm}
\begin{proof} Let $\tilde{\E}$ denote the fundamental solution of the operator $L_{\tilde{A}}.$ As in \eqref{scaledA}, we have
\begin{equation}\label{homog}
\nabla_1\E_A(x,y)=\ell^{-n}\nabla_1\tilde{\E}(x/\ell,y/\ell),
\end{equation}
so an application of Lemma \ref{lemma_freezing_per_1} gives
\begin{equation}
\begin{split}
&\big|\nabla_1\E_A(x,y)-\nabla_1\Theta(x,y; A(x))\big|\\
&\qquad=\ell^{-n}|\nabla_1\E_{\tilde{A}}(\ell^{-1}x,\ell^{-1}y)-\nabla_1\Theta(x,y; \tilde{A}(\ell^{-1}x))|\leq c \ell^\alpha |x-y|^{\alpha-n}.
\end{split}
\end{equation}
Using \eqref{approx_hom} and \eqref{homog}, we get
\begin{equation}
\begin{split}
&|\nabla_1\E_A(x,y)-(Id+\nabla\chi_\ell(x))\nabla_1\Theta(x,y;A_0)|\\
&\qquad =\ell^{-n}|\nabla_1\tilde{\E}(x/\ell,y/\ell)-(Id+\nabla\tilde \chi(x/\ell))\nabla_1\Theta(x/\ell,y/\ell;A_0)|\\
&\qquad \lesssim \frac{c \ell^{n+\gamma}}{\ell^n|x-y|^{n+\gamma}}=\frac{c \ell^\gamma}{|x-y|^{n+\gamma}},
\end{split}
\end{equation}
where $c$ depends on $n,\alpha$ and $\|\tilde{A}\|_{\alpha}$, $\|\tilde{A}\|_{\alpha}\leq\|A\|_{\alpha}.$ Inequality \eqref{freezing_l_2} follows as \eqref{freezing_l_1}.
\end{proof}

\section{The existence of principal values}\label{section_pv}
The purpose of the present section is to prove Theorem \ref{theorem_pv_layer_pot}. The proof of the existence of principal values can be divided into the study of two different cases: the case in which $\mu$ is a rectifiable measure and the one in which $\mu$ has zero $n$-density, i.e.
\begin{equation}
\lim_{r\to 0}  \frac{\mu(B(x,r))}{r^n}=0 \qquad \text{ for } \mu\text{-a.e.}\,\,x\in\Rn1.
\end{equation}
Indeed, without providing the detailed argument, we recall that by means of \cite[Theorem 2]{PPT} we can decompose a measure $\mu$ for which $T_\mu$ is bounded on $L^2(\mu)$ into the sum of a rectifiable measure and a measure with zero $n$-density almost everywhere.

\subsection{Principal values for rectifiable measures with compact support} This subsection follows the scheme of \cite[Section 2.2]{CMT}.
The proof of the existence of principal values for $T_\mu$ if the measure $\mu$ is rectifiable and has compact support relies on the following result.
\begin{theor}\label{thm_DS_existence_pv}Let $\mu$ be a rectifiable measure. Let $K\in C^\infty(\mathbb{R}^{n+1}\setminus\{0\})$ be an odd kernel and homogeneous of degree $-n$, i.e. $K(x)=-K(-x)$ and $K(\lambda x)=\lambda^{-n}K(x)$. Assume, for some $M=M(n)$, the further regularity condition
\begin{equation}\label{dec_deriv_kernel_ds}
|\nabla_j K(x)|\lesssim_n C(j) |x|^{-n-j}\hbox{} \,\,\text{ for all }0\leq j\leq M\,\,\text{ and }x\in\mathbb{R}^{n+1}\setminus\{0\}.
\end{equation}
Then the operator $T_{K,\mu}$ is bounded on $L^2(\mu)$ with operator norm
\begin{equation}
\|T_{K,\mu}\|_{L^2(\mu)\to L^2(\mu)}\lesssim_n \|K|_{\mathbb{S}^n}\|_{C^M(\mathbb{R}^{n+1})}.
\end{equation}
Moreover, the principal value
\begin{equation}
T_{K,\mu}f(x)=\lim_{\varepsilon\to 0}\int_{|x-y|\geq \varepsilon } K(x-y)f(y)d\mu(y)
\end{equation}
exists $\mu$-almost everywhere.
\end{theor}
The proof of the boundedness of $T_{K,\mu}$ is due to David and Semmes.
The result on principal values was first proved imposing an analogous condition for all $j=0,1,2,\ldots$ (for a more detailed exposition  we refer, for example, to \cite[Chapter 20]{mat}). We remark that it has been recently improved by Mas (see \cite[Corollary 1.6]{mas}).

The previous theorem together with a spherical harmonics expansion of the kernel is the key tool to prove the following result.
\begin{lemm}\label{lemma_mt_lip}
Let $\mu$ be an $n$-rectifiable measure. There exists $M=M(n)$ such that the following holds. Let $b(x, z)$ be odd in $z$ and homogeneous of degree $-n$ in $z$, and assume
$D^\alpha_z b(x,z)$ is continuous and bounded on $\mathbb{R}^{n+1}\times\mathbb{S}^n$, for any multi-index $|\alpha|\leq M$. Then for every $f\in L^2(\mu)$, the limit
\begin{equation}
Bf(x)=\lim_{\varepsilon\to 0}\int_{|x-y|>\varepsilon}b(x,x-y)f(y)d\mu(y)
\end{equation}
exists for $\mu$-almost every $x$.
\end{lemm}
\begin{proof}{{This result is used in \cite{MT} (see for example \cite[(1.14)]{MT}). The proof is a variation of the argument in \cite[Proposition 1.2]{MT}.}}
{
{For the reader's convenience we discuss the details below.}

Let $\{\varphi_{j,l}\}_{j\geq 1, 1\leq l \leq N_j}$ be an orthonormal basis of $L^2(\mathbb{S}^n)$ consisting of surface spherical harmonics of degree $j$. Recall that (see \cite[(2.12)]{spherical_harmonics})
\begin{equation} \label{estim_nj}
N_j=O(j^{n-1}),\qquad \text{ for }j\gg 1.
\end{equation}
Using the homogeneity assumption for $b(x,\cdot)$ and the orthonormal expansion, we write
\begin{equation}\label{pvlip1}
\begin{split}
b(x,z)&=b\Big(x,\frac{z}{|z|}\Big)|z|^{-n}=\sum_{j\geq 1}\sum_{l=1}^{N_j} \langle b(x,\cdot),\varphi_{j,l}\rangle_{L^2(\mathbb{S}^n)}\varphi_{j,l}\Big(\frac{z}{|z|}\Big)|z|^{-n}\\
&=\sum_{j,l} b_{j,l}(x)\varphi_{j,l}\Big(\frac{z}{|z|}\Big)|z|^{-n},
\end{split}
\end{equation}
where $b_{j,l}(x)\coloneqq   \langle b(x,\cdot),\varphi_{j,l}\rangle_{L^2(\mathbb{S}^n)}$. Since $b(x,\cdot)$ is an odd function and $\varphi_{2j, l}$ is even for every $j$, $b_{j,l}(x)\equiv 0$ for $j$ even.  Being $b$ in $L^\infty(\mathbb{R}^{n+1}\times \mathbb{S}^n)$ by hypothesis and H\"older's inequality, we have
\begin{equation}\label{pvlip2}
|b_{j,l}(x)|\leq C(n)\|b(x,\cdot)\|_{L^\infty(\mathbb{S}^n)}\|\varphi_{j,l}\|_{L^2(\mathbb{S}^n)}\leq C(n)\|b\|_{L^\infty(\mathbb{R}^{n+1}\times \mathbb{S}^n)}\leq C(n).
\end{equation}
Moreover, recalling that we can suppose $j$ odd,  the function $\tilde{K}_{j,l}(z)\coloneqq  \varphi_{j,l}\big(z/|z|\big)|z|^{-n}$ satisfies the hypothesis in Theorem \ref{thm_DS_existence_pv}: there exists an harmonic polynomial $P_{j,l}$ of odd degree $j$ such that $\varphi_{j,l}(z/|z|)=P_{j,l}(z)/|z|^j$, so
\begin{equation}
\Big|\nabla \varphi_{j,l}\Big(\frac{z}{|z|}\Big)\Big|\lesssim \frac{1}{|z|}
\end{equation}
and
\begin{equation}
\big|\nabla \tilde K_{j,l}(z)\big|\lesssim \Big|\nabla \varphi_{j,l}\Big(\frac{z}{|z|}\Big)\Big|\frac{1}{|z|^n}+\Big|\varphi_{j,l}\Big(\frac{z}{|z|}\Big)\Big|\frac{1}{|z|^{n+1}}\lesssim \frac{1}{|z|^{n+1}}.
\end{equation}
Analogous estimates hold for higher order derivatives. So, Theorem \ref{thm_DS_existence_pv} ensures that
\begin{equation}\label{pvlip3}
T_{\tilde{K}_{j,l},\mu}f(x)=\lim_{\varepsilon\to 0}\int_{|x-y|>\varepsilon}\tilde{K}_{j,l}(x-y)f(y)d\mu(y)\equiv \lim_{\varepsilon\to 0}T_{\tilde K_{j,l}, \mu, \varepsilon}f(x)
\end{equation}
exists for $\mu$-a.e $x$. Recall also that by the Theorem \ref{thm_DS_existence_pv} there exists $M=M(n)$ such that $T_{\tilde{K}_{j,l},\mu}$ is bounded on $L^2(\mu)$ with operator norm
\begin{equation}\label{l2normharm}
\|T_{\tilde{K}_{j,l},\mu}\|_{L^2(\mu)\to L^2(\mu)}\lesssim  \|\tilde{K}_{j,l}|_{\mathbb{S}^{n}}\|_{C^M(\mathbb{S}^n)}=\|\varphi_{j,l}\|_{C^M(\mathbb{S}^n)}.
\end{equation}
Gathering \eqref{pvlip1}, \eqref{pvlip2} and \eqref{pvlip3}, to prove the lemma it is enough to show that the dominated convergence theorem applies and, in particular, that
\begin{equation} \label{thesis_dct}
\sum_{j,l}\big|b_j(x)T_{\tilde K_{j,l}, \mu, \varepsilon}f(x)\big|
 \leq C(x)<\infty,
\end{equation}
where $C(x)$ does not depend on $\varepsilon$.
By Lebesgue differentiation theorem, to prove \eqref{thesis_dct} it suffices to show that for every ball $B_0\subset\mathbb{R}^{n+1}$ we have
\begin{equation}
\begin{split}
\sum_{j,l}\int_{B_0} \big|b_{j,l}(x)T_{\tilde{K}_{j,l}, \mu, \varepsilon}f(x)\big|d\mu(x)&\lesssim_{B_0, n} \sum_{j,l, m}\|b_{j,l}\|_\infty \|\varphi_{j,l}\|_{C^m(\mathbb{S}^n)}\|f\|_{L^2(\mu)}\\ &\leq C \|f\|_{L^2(\mu)}
\end{split}
\end{equation}
for some $C>0$, where the first inequality above uses the $L^2$-boundedness \eqref{l2normharm}.

The smoothness of $b$ implies that (see \cite[3.1.5]{Stein})
$$\|b_{j,l}\|_\infty \lesssim \frac{1}{j^{\frac{3}{2}n + 2+M}},$$
where the exponent on the right hand side is chosen accordingly to what we need next.
Now, recall that the Sobolev space $H^s(\mathbb{S}^n)$, $s\in\mathbb{R}$ can be defined via spherical harmonics expansion. In particular, it is the completion of $C^\infty(\mathbb{S}^n)$ with respect to the norm
\begin{equation}\label{def_sobolev}
\|v\|_{H^s(\mathbb{S}^n)}\coloneqq \Big(\sum_{j,l}\Big(j+\frac{n-1}{2}\Big)^{2s}|v_{j,l}|^2\Big)^{1/2},
\end{equation}
where $v_{j,l}=\langle v,\varphi_{j,l}\rangle_{L^2(\mathbb{S}^n)}.$ For the definition and the properties of this space, we refer for example to \cite[Section 3.8]{spherical_harmonics} and to \cite[Section 6.3]{spherical_harmonics} for the relation of \eqref{def_sobolev} with that via the restriction of the gradient to the unit sphere. By Sobolev embedding theorem, $H^{s}(\mathbb{S}^n)$ continuously embeds into $C(\mathbb{S}^n)$ for $s>n/2$. So, choosing $s=\frac{n}{2}+1$ and using \eqref{def_sobolev} can estimate 
\begin{equation}
\|D^m \varphi_{j,l}\|_{C(\mathbb{S}^n)}\lesssim_n \|\varphi_{j,l}\|_{H^{s+m}(\mathbb{S}^n)}=\Big(\frac{2j+n-1}{2}\Big)^{\frac{n}{2}+m + 1}.
\end{equation}
Hence, using \eqref{estim_nj}
\begin{equation}
\sum_{j,l}\|b_{j,l}\|\|\varphi_{j,l}\|_{C^M(\mathbb{S}^n)} \lesssim_n \sum_{m=0}^M\sum_{j\geq 1}N_j j^{-\frac{3}{2}n - 2-M} j^{\frac{n}{2}+m +1 }\lesssim \sum_{j\geq 1}\frac{1}{j^2}<\infty,
\end{equation}
which concludes the proof.
}
\end{proof}
\begin{theor}\label{thm_pv_layer_pot_rect}
Let $\mu$ be an $n$-rectifiable measure on $\Rn1$ with compact support. Let $A$ be a matrix having the properties \eqref{eqelliptic1}, \eqref{eqelliptic2} and \eqref{eq:Holdercont}. Then for every $f\in L^2(\mu)$ the principal value
\begin{equation}
T_\mu f(x)=\lim_{\varepsilon\to 0}\int_{|x-y|>\varepsilon}\n1\E(x,y)f(y)d\mu(y)
\end{equation}
exists for $\mu$-almost every $x$.
\end{theor}
\begin{proof}
Let $\varepsilon>0$ and denote $b(x,z)\coloneqq   \n1\Theta(z,0;A(x))$. As a consequence of the explicit formula \eqref{fund_sol_const_matrix}, it is not difficult to see that each component of $b$ verifies the hypothesis of Lemma \ref{lemma_mt_lip}. So, split $T_{\mu,\varepsilon}$ as
\begin{equation}\label{split_ex_pv_lip}
T_{\mu,\varepsilon} f(x)=\int_{|x-y|>\varepsilon}b(x,x-y)f(y)d\mu(y) + \int_{|x-y|>\varepsilon}\big(\n1\E(x,y)-\n1\Theta(x,y;A(x))\Big)f(y)d\mu(y).
\end{equation}
The limit for $\varepsilon\to 0$ of the first integral in the right hand side of \eqref{split_ex_pv_lip} exists $\mu$-a.e. because of Lemma \ref{lemma_mt_lip}. On the other hand, $\n1\E(x,y)-\n1\Theta(x,y;A(x))$ defines an operator which is compact on $L^p(\mu)$ because of Lemma \ref{lemm_freezing}, which guarantees that the limit for $\varepsilon\to 0$ exists for $\mu$-a.e. $x$ and concludes the proof.
\end{proof}

\subsection{Principal values for measures with zero density}
We argue as in \cite[Chapter 8]{tolsa_book}, proving the existence of the principal values passing through the existence of the weak limit and following the approach of Mattila and Verdera \cite{mattila_verdera}. Again, we suppose that $\mu$ has compact support. 

A combination of the proof of  \cite[Theorem 1.4]{mattila_verdera} (see also \cite[Theorem 8.10]{tolsa_book}) and Lemma \ref{lemm_freezing} makes possible to prove that if $\mu$ is a Radon measure in $\mathbb{R}^{n+1}$ with growth of degree $n$,
then for every $1<p<\infty$ and $f\in L^p(\mu)$, $\{T_{\mu,\epsilon}f\}_\epsilon$ admits a weak limit $T^w_\mu f$ in $L^p(\mu)$ as $\epsilon\to 0$.
Moreover, the representation formula 
\begin{equation}\label{formula_weak_limit}
T^w_\mu f(x)=\lim_{r\to 0}\fint_{B(x,r)} T_\mu\big(f\chi_{B(x,r)^c}\big)(y)d\mu(y)
\end{equation}
holds for $\mu$-almost every $x\in\Rn1$, giving an explicit way of computing the weak limit. We remark that, in general, we can only infer that formula \eqref{formula_weak_limit} holds if $T_\mu$ has an antisymmetric kernel.

Let us recall the following theorem by Mattila and Verdera (see \cite{mattila_verdera}), here reported in the formulation of \cite[Theorem 8.11]{tolsa_book}.
\begin{theor}\label{thm_zero_density_mv}
Let $\mu$ be a Radon measure in $\mathbb{R}^d$ that has growth of degree $n$ and zero
$n$-dimensional density $\mu$-a.e. Let $\mathcal T_\mu$ be an $n$-dimensional antisymmetric Calder\'on-Zygmund operator. 
Then, for all $1 < p < \infty$ and $f \in L^p(\mu)$, $\pv \mathcal T_\mu f(x)$ exists for $\mu$-a.e. $x\in\mathbb{R}^d$ and coincides with $\mathcal T^w_\mu f(x)$. Also, for all $\nu\in M(\mathbb{C})$, $\pv \mathcal T\nu(x)$ exists for $\mu$-a.e. $x\in\mathbb{R}^d$.
\end{theor}
This result can be transferred to the gradients of the single layer potential $T_\mu$.
\begin{theor}\label{thm_zero_density_layer_pot}
Let $\mu$ be a Radon measure in $\mathbb{R}^{n+1}$ that has growth of degree $n$, zero $n$-dimensional density and compact support. Suppose that $T_\mu$ is a bounded operator from $L^2(\mu)$ to $L^2(\mu)$. Then, for all $1 < p < \infty$ and $f \in L^p(\mu)$, $\pv  T_\mu f(x)$ exists for $\mu$-a.e. $x\in\Rn1$ and coincides with $T^w_\mu f(x)$. Also, for all $\nu\in M(\mathbb{C})$, $\pv T\nu(x)$ exists for $\mu$-a.e. $x\in\Rn1$.
\end{theor}
\begin{proof}
Let $1<p<\infty$ and $f\in L^p(\mu)$. We decompose $T_\mu f$ into its symmetric and antisymmetric part. That is to say,
\begin{equation}
T_\mu f(x)=T^{(a)}_\mu f(x)+ T^{(s)}_\mu f(x),
\end{equation}
where $T^{(a)}_\mu$ is the integral operator with kernel $(\n1 \E(x,y)-\n1\E(y,x))/2$ and $T^{(s)}_\mu$ whose kernel is $(\n1 \E(x,y)+\n1\E(y,x))/2$. We can apply Theorem \ref{thm_zero_density_mv} to antisymmetric part $T^{(a)}_\mu$, obtaining that $\pv T^{(a)}_\mu f(x)$ exists for $\mu$-a.e. $x$.

On the other hand, $T^{(s)}_\mu$ defines a compact operator on $L^p(\mu)$ since
\begin{equation}
\int |\n1 \E(x,y)+\n1\E(y,x)|d\mu(y)\lesssim \diam(\supp\mu)^\alpha,
\end{equation}
so that the principal values exist.

The fact that $T^w_\mu f$ coincides with $\pv T_\mu f$ a.e. follows from the definition of weak limit together with dominated convergence theorem:
\begin{equation}
\int T^w_\mu f g \,d\mu=\lim_{\varepsilon\to 0}\int T_{\mu,\varepsilon}f g \,d\mu =\int \pv T_\mu f g \,d\mu \quad \text{for all}\quad g\in L^{p'}(\mu),
\end{equation}
$p'$ being the H\"older conjugate exponent of $p$.
\end{proof}

\vspace{2mm}
\noindent{\bf A remark on the well-posedness of the assumption (6) of Theorem \ref{main_theorem}.}
Let $T,\mu$ and $B$ be as in Theorem \ref{main_theorem}. Let $x, y\in B$ and $\varepsilon>0$ and write
\begin{equation}\label{dec_sum_1}
T_{\epsilon}\mu(x)-T_{\epsilon}\mu(y)=T_{\mu,\epsilon}\chi_{2B}(x)-T_{\mu,\epsilon}\chi_{2B}(y)+ \big[T_{\mu,\epsilon}\chi_{\Rn1\setminus 2B}(x)-T_{\mu,\epsilon}\chi_{\Rn1\setminus 2B}(y)\big].
\end{equation}
Now observe that, being the operator $T_{\mu|_B}$ bounded on $L^2(\mu|_B)$, Theorem \ref{theorem_pv_layer_pot} (2) applies with $\nu=\chi_{2B}\mu$. So, the first two summands on the right hand side of \eqref{dec_sum_1} admit a limit as $\varepsilon\to 0$ for almost every $x,y\in B$. The limit for $\varepsilon\to 0$ of the last summand exists, too. Indeed, since $x,y$ do not belong to $\Rn1\setminus 2B$, for $\varepsilon<r(B)$,
\begin{equation}
T_{\mu,\epsilon}\chi_{\Rn1\setminus 2B}(x)-T_{\mu,\epsilon}\chi_{\Rn1\setminus 2B}(y)=\int_{\Rn1\setminus 2B} \big(\nabla_1\E(x,z)-\nabla_1\E(y,z)\big)d\mu(y).
\end{equation}
If we assume $\tilde \alpha\leq \alpha$ in the statement of the main theorem, an application of the Calder\'on-Zygmund property of the kernel combined with a dyadic decomposition of the domain of integration gives
\begin{equation}\label{maj_p_gamma}
\begin{split}
\Big|\int_{\Rn1\setminus 2B}& \big(\nabla_1\E(x,z)-\nabla_1\E(y,z)\big)d\mu(z)\Big|	\lesssim |x-y|^\alpha\sum_{j=1}^{+\infty}\int_{2^{j+1}B\setminus 2^j B}\frac{1}{|x-z|^{n+\alpha}}d\mu(z)\\
&\leq P_{\mu,\alpha}(B)\leq P_{\mu,\tilde \alpha}(B)<+\infty.
\end{split}
\end{equation}
In particular, this tells that $T\mu(x)-T\mu(y)$ exists in the principal value sense for almost every $x,y\in B$.

We also want to point out that $T\mu-m_{\mu,B}(T\mu)$ defines an $L^2(\mu|_B)$-function. Indeed, for $x\in B$ and using \eqref{maj_p_gamma},
\begin{equation}
\begin{split}
|T\mu(x)-m_{\mu,B}(T\mu)|&\leq \frac{1}{\mu(B)}\int_B|T\mu(x)-T\mu(y)|d\mu(y)\\
&\leq |T(\chi_{2B}\mu)(x)| + \big(m_{\mu,B}|T(\chi_{2B}\mu)|^2\big)^{1/2} + P_{\mu,\tilde{\alpha}}(B).
\end{split}
\end{equation}
The right hand side of the previous majorization defines an $L^2(\mu|_B)$ function because of the assumptions $T(\chi_{2B}\mu)\in L^2(\mu|_B)$ and $P_{\mu, \tilde{\alpha}}(B)<+\infty$ in Theorem \ref{main_theorem}.

\section{The Main Lemma}\label{section_periodic_symmetric}
A careful read of \cite{GT} shows that the same arguments as the ones for the Riesz transform give that, in order to prove Theorem \ref{main_theorem}, it suffices to prove the following result.
\begin{lemm}[Main Lemma]\label{mainlemma_origin}Let $n>1$ and let $C_0,C_1>0$ be some arbitrary constants. There exist
$M=M(C_0,C_1, n)>0$ big enough, $\lambda(C_0,C_1, n)>0$ and $\epsilon=\epsilon(C_0,C_1,M, n)>0$ small enough such that if $\delta=\delta(M,C_0,C_1,n)>0$ is sufficiently small, then the following holds. Let $\mu$ be a Radon measure in $\mathbb{R}^{n+1}$ with compact support and $Q_0\subset\mathbb{R}^{n+1}$ a cube centered at the origin satisfying the properties:
\begin{enumerate}
\item $\ell(MQ_0)\leq \lambda.$
\item $\mu(Q_0)=\ell(Q_0)^n$.
\item $P_{\mu,\tilde{\alpha}}(MQ_0)\leq C_0$.
\item For all $x\in 2Q_0$ and $0<r\leq \ell(Q_0),$ $\Theta_{\mu}(B(x,r))\leq C_0.$
\item $Q_0$ has $C_0$-thin boundary.
\item $\alpha^L_\mu(3MQ_0)\leq\delta,$ for some hyperplane $L$ through the origin.
\item $T_{\mu|_{2Q_0}}$ is bounded on $L^2(\mu|_{2Q_0})$ with $\|T_{\mu|_{2Q_0}}\|_{L^2(\mu|_{2Q_0})\to L^2(\mu|_{2Q_0})}\leq C_1.$
\item We have
\begin{equation}\label{bmotypee}
\int_{Q_0}|T\mu(x)-m_{\mu,Q_0}(T\mu)|^2d\mu(x)\leq \epsilon\mu(Q_0).
\end{equation}
\end{enumerate}
Then there exists some constant $\tau>0$ and a uniformly $n$-rectifiable set $\Gamma\subset \mathbb{R}^{n+1}$ such that
\begin{equation}
\mu(Q_0\cap\Gamma)\geq \tau \mu(Q_0),
\end{equation}
where the constant $\tau$ and the uniform rectifiability constants of $\Gamma$ depend on all the constants above.
\end{lemm}
The matrix $A$ may have a very general form. In particular, we need some additional argument to overcome the lack of ``symmetries'' of the matrix with respect to reflections and to periodization (the exact meaning of this sentence will be clear after the reading of Section \ref{section_change}, where we recall how second order PDE's in divergence form are affected by a change of variable). Indeed, this is a crucial point for our proof to work. A similar problem has been faced in \cite{PPT}. First, in order to be able to argue via a change of variables, we have to show that we can assume the matrix $A$ to be symmetric.\\

We recall Schur's lemma for integral operators with a reproducing kernel. The proof is a standard application of Cauchy-Schwarz's inequality.
\begin{lemm}\label{schur_lemma}Let $K\colon\mathbb{R}^{n+1}\times \mathbb{R}^{n+1}\to \mathbb{R}^{n+1}$ be a function such that, for a constant $C>0$, we have
\begin{equation}\label{int_kernel_x}
\int |K(x,y)|d\mu(x)\leq C
\end{equation}
and
\begin{equation}\label{int_kernel_y}
\int |K(x,y)|d\mu(y)\leq C.
\end{equation}
Then the operator $Tf=K\ast f$ is a continuous operator from $L^2(\mu)$ to $L^2(\mu)$ and
\begin{equation}\label{norm_estim_schur}
\|T\|_{L^2(\mu)\to L^2(\mu)}\leq C.
\end{equation}
\end{lemm}
Let $A$ be a matrix as before. We denote by $A_s=(A+A^T)/2$ its symmetric part  and by $T^{A_s}_\mu$ its correspondent gradient of the single layer potential.

Recalling that, for any matrix $A_0$ with constant coefficients we have $\Theta(\cdot,\cdot;A_0)=\Theta(\cdot,\cdot;A_{0,s}),$ we can formulate the following lemma.
\begin{lemm} \label{lemm_symetrization}Let $Q$ be a cube in $\Rn1$ such that, for $M>1,$ 
$
P_{\mu,\alpha}(MQ)\leq C_1.
$
The operator $T^{(s)}_{\mu|2Q}$ is bounded on $L^2(\mu|2Q)$ if and only if $T_{\mu|2Q}$ is bounded on $L^2(\mu|2Q)$. In particular
\begin{equation}\label{estim_A_A_S}
\big\|T_{\mu|_{2Q}}\big\|_{L^2(\mu|2Q)\to L^2(\mu|2Q)} = \big\|T^{A_s}_{\mu|_{2Q}}\big\|_{L^2(\mu|2Q)\to L^2(\mu|2Q)}  + O(\ell(Q)^\alpha).
\end{equation}
Moreover
\begin{equation}\label{estim_mean_symmetric_matrix}
\begin{split}
\int_{Q}\big|T^{A_s}\mu(x)-&m_{\mu,Q}(T^{A_s}\mu)\big|^2d\mu(x)\\
&\lesssim_{\Lambda,\|A\|_\alpha} \int_{Q}\big|T\mu(x)-m_{\mu,Q}(T\mu)\big|^2 d\mu(x) + \big(M^\alpha\ell(Q)^\alpha + M^{-\alpha}\big)^2\mu(Q).
\end{split}
\end{equation}
\end{lemm}
\begin{proof}
Let us first prove \eqref{estim_A_A_S}.
The identity \eqref{a_eq_a_sym} for  matrices with constant coefficients leads to
\begin{equation}\label{lemma_symmetric_matrix_potential}
\begin{split}
&T^{A_s}_{\mu|_{2Q}} f(x)=\int_{2Q}\nabla_1\E_{A_s}(x,y)f(y)d\mu(y)\\
&\qquad =\int_{2Q}\big(\nabla_1\E_{A_s}(x,y)-\nabla_1\Theta(x,y;A_s(x))\big)f(y)d\mu(y) \\
&\qquad + \int_{2Q}\big(\nabla_1\Theta(x,y;A(x))-\nabla_1\E(x,y)\big)f(y)d\mu(y)
+ \int_{2Q}\nabla_1\E(x,y)f(y)d\mu(y)\\
&\qquad \equiv I + II + T_{\mu|_{2Q}} f(x).
\end{split}
\end{equation}
To estimate $I$ and $II$ in \eqref{lemma_symmetric_matrix_potential} it suffices, then, to invoke Lemma \ref{lemm_estim_freezing_periodic} and Schur's Lemma. This finishes the proof of the first part of the lemma.

Let us now prove \eqref{estim_mean_symmetric_matrix}. We split
\begin{equation}\label{AAA}
\begin{split}
&T\mu(x)-m_{\mu,Q}(T\mu)\\
&\qquad=\Big(T(\chi_{MQ}\mu)(x)-m_{\mu,Q}\big(T(\chi_{MQ}\mu)\big)\Big)+\Big(T(\chi_{(MQ)^c}\mu)(x)-m_{\mu,Q}\big(T(\chi_{(MQ)^c}\mu)\big)\Big).
\end{split}
\end{equation}
Let us estimate the two terms in the right hand side  separately. Again, as a consequence of \eqref{lemma_symmetric_matrix_potential} and  Lemma \ref{lemm_freezing} we can write
\begin{equation}
\begin{split}
&\Big|T(\chi_{MQ}\mu)-m_{\mu,Q}(T(\chi_{MQ}\mu)) - \Big(T^{A_s}(\chi_{MQ}\mu)+m_{\mu,Q}\big(T^{A_s}(\chi_{MQ}\mu)\big)\Big)\Big|\lesssim M^\alpha\ell(Q)^\alpha.
\end{split}
\end{equation}
To bound the second term in the right hand side of \eqref{AAA}, notice that for $x,y\in Q$ standard estimates together with Lemma \ref{lemm_estim_freezing_periodic} give
\begin{equation}\
\begin{split}
&\big|T_\mu \chi_{(MQ)^c}(x)-T_\mu\chi_{(MQ)^c}(y)\big|\lesssim \int_{(MQ)^c}\frac{|x-y|^{\alpha}}{|x-z|^{n+{\alpha}}}d\mu(z)\\
&\qquad \lesssim \frac{|x-y|^{\alpha}}{\ell(MQ)^{\alpha}}P_{\mu,{\alpha}}(MQ)\lesssim \frac{1}{M^{\alpha}}P_{\mu,{\alpha}}(MQ)\lesssim \frac{1}{M^{\alpha}},
\end{split}
\end{equation}
so that, averaging over $y$ in $Q$ we have
\begin{equation}
\big|T(\chi_{(MQ)^c}\mu)(x)-m_{\mu,Q}\big(T(\chi_{(MQ)^c}\mu)\big)\big|\lesssim M^{-{\alpha}}
\end{equation}
The same calculations lead to
\begin{equation}
\Bigl|T^{A_s}(\chi_{(MQ)^c}\mu)(x)-m_{\mu,Q}\big(T^{A_s}(\chi_{(MQ)^c}\mu)\big)\Bigr|\lesssim M^{-{\alpha}},
\end{equation}
so the inequality  \eqref{estim_mean_symmetric_matrix}  in the statement of the lemma follows by gathering all the previous considerations.
\end{proof}
Gathering Lemma \ref{mainlemma_origin} and Lemma \ref{lemm_symetrization} shows that it suffices to prove Theorem \ref{main_theorem} under the additional assumption that the matrix $A$ is symmetric. Indeed, proving Lemma \ref{mainlemma_origin} with $A=A_s$ gives it in the non-symmetric case with worse assumptions on the parameters involved. We omit further details.
\begin{rem} Arguing as in Lemma \ref{lemm_symetrization}, one could prove that 
\begin{equation}
\big\|T_{\mu|_{2Q}}\big\|_{L^2(\mu|_{2Q})\to L^2(\mu|_{2Q})} = \big\|T^{a}_{\mu|_{2Q}}\big\|_{L^2(\mu|_{2Q})\to L^2(\mu|_{2Q})}+ O(\ell(Q)^\alpha),
\end{equation}
where $T^a$ is the operator corresponding to the antisymmetric part of the kernel $K(\cdot,\cdot)$, that is to say $K^a(x,y)=(K(x,y)-K(y,x))/2$. However, as in \cite{PPT} and \cite{CMT}, we prefer not to make this reduction because it would create problems later on in the proof. In particular, it would be an obstacle to the application of the maximum principle, which is a crucial tool in Section \ref{section_var_arg}.
\end{rem}

\section{The modification of the matrix}\label{section_change}
\subsection{The change of variable}
The following lemma deals with how the fundamental solution and its gradient are affected by a change of variable. 
\begin{lemm}[see \cite{PPT}, Lemma 13]\label{lemma_change_var_fund_sol} Let $\phi\colon\Rn1\to\Rn1$ be a locally bilipschitz map and let $A\in\Lambda_\alpha$. Let $\E_A$ be the fundamental solution of $L_A=-\Div(A\nabla\cdot).$ Set
${A}_\phi\coloneqq  |\det\phi|D({\phi}^{-1})(A\circ\phi)D({\phi}^{-1})^T.$ Then
\begin{equation}
\E_{{A}_\phi}(x,y)=\E_A(\phi(x),\phi(y)) \text{ for }x,y\in\Rn1,
\end{equation}
and
\begin{equation}
\n1 \E_{A_\phi}(x,y)=D(\phi)^T(x)\n1\E_A(\phi(x),\phi(y)) \quad \text{for}\quad x\in\Rn1.
\end{equation}

\end{lemm}
Let us state a lemma concerning how the gradient of the fundamental solution transforms under a change of variable $\phi$ as in Lemma \ref{lemma_change_var_fund_sol}. We use the notation
\begin{equation}
T_\phi\mu(x)\coloneqq  \int \nabla_1\E_{A_\phi}(x,y)d\mu(y).
\end{equation}
\begin{lemm}[see \cite{PPT}, Lemma 14]\label{pot_changed}Let $\phi\colon\Rn1\to\Rn1$ be a bilipschitz change of variables. For every $x\in\Rn1$ we have
\begin{equation}
T_\phi\mu(x)=D(\phi)^T(x)T{\phi\sharp\mu}(\phi(x)).
\end{equation}
\end{lemm}
A particularly useful change of variable is the one that turns the symmetric part of the matrix at a given point into the identity. For the following statement we refer to \cite{AGMT}.
\begin{lemm}\label{lemma_change_matrix_identity}
Let $\Omega\subset\Rn1$ be an open set, and assume that $A$ is a uniformly elliptic matrix with real entries. Let $A_s = (A + A^T
)/2$ be the symmetric part of $A$ and for a fixed point $y_0 \in \Omega$ define $S =\sqrt{A_s(y_0)}$. If
\begin{equation}
\tilde{A}(\cdot)=S^{-1}(A\circ S)(\cdot)S^{-1},
\end{equation}
then $\tilde{A}$ is uniformly elliptic, $\tilde{A}_s(z_0)=Id$ if $z_0=S^{-1}y_0$ and $u$ is a weak solution of
$L_A u = 0$ in $\Omega$ if and only if $\tilde{u} = u \circ S$ is a weak solution of $L_{\tilde{A}}\tilde{u}=0$ in $S^{-1}(\Omega).$
\end{lemm}
As a remark, we want to point out that the change of variables defined in Lemma \ref{lemma_change_matrix_identity} is a linear map and, in particular, a bilipschitz map of $\Rn1$ to itself. Its bilipschitz constant depends on the ellipticity of the matrix $A$.

We need the notion of flatness for images of cubes via maps of the aforementioned type.
For a set $E\subset\Rn1$, we define the $\alpha$-number in an analogous ways as for cubes. In particular, for any hyperplane $L$ and any measure $\nu$, we denote
\begin{equation}
\alpha^L_\nu(E)\coloneqq  \frac{1}{\diam(E)^{n+1}}\inf_{c\geq 0}d_E(\nu,c\Hcal^n|_L).
\end{equation}
This particular notation will be used only in this section.
\begin{lemm}\label{lemma_mod_alpha}
Let $\varphi$ be an affine, bilipschitz change of variables of $\Rn1$. Let $L$ be a hyperplane in $\Rn1$. Let $J_\varphi>0$ be the Jacobian of $\varphi$. Then, for any Radon measure $\nu$, for any cube $Q\subset \Rn1$ and any constant $c\geq 0$ we have that
\begin{equation}\label{measure_distance_change}
d_Q(\nu,c\Hcal^n|_L)\approx_{n,C} d_{\varphi(Q)}\big(\varphi\sharp\nu, c\Hcal^n|_{\varphi(L)}\big).
\end{equation}
Hence,
\begin{equation}\label{alpha_change}
\alpha^L_\nu(Q)\approx_{n,C} \alpha^{\varphi(L)}_{\varphi\sharp\nu}(\varphi(Q)\big).
\end{equation}
\end{lemm}
\begin{proof}
Formula \eqref{alpha_change} is an immediate consequence of \eqref{measure_distance_change} and the fact that $\ell(Q)\approx_C \diam(\varphi(Q)).$

Let us prove \eqref{measure_distance_change}.
For every $c\geq 0$
\begin{equation}
\varphi\sharp \big(c\Hcal^n|_L\big)=c(\varphi\sharp \Hcal^n)|_{\varphi(L)}.
\end{equation}
Indeed for any $\varphi\sharp \Hcal^n|_L$-measurable set $E$ we have
\begin{equation}
\begin{split}
\varphi\sharp (c\Hcal^n|_L) (E)=c\Hcal^n\big(\varphi^{-1}(E)\cap L\big)= c\Hcal^n\big(\varphi^{-1}(E\cap\varphi(L))\big)= c(\varphi\sharp \Hcal^n)|_{\varphi(L)}(E).
\end{split}
\end{equation}
Moreover, as a consequence of the Radon-Nikodym differentiation theorem (see \cite[Lemma 1, p. 92]{EG}), we have
\begin{equation}
\Hcal^n\big(\varphi^{-1}(E)\big)=J_\varphi\Hcal^n(E).
\end{equation}
So,
\begin{equation}
\begin{split}
d_Q(\nu,c\Hcal^n|_L)\approx_C d_{\varphi(Q)}\big(\varphi\sharp\nu, \varphi\sharp c\Hcal^n|_L\big)\approx_{n,C} d_{\varphi(Q)}\big(\varphi\sharp\nu, c\Hcal^n|_{\varphi(L)}\big),
\end{split}
\end{equation}
which proves the lemma.
\end{proof}
\subsection{Reduction of the Main Lemma to the case $A(0)=Id$}

Recall that by Lemma \ref{lemm_symetrization} we can assume $A$ to be a symmetric matrix.

Let us begin with a preliminary observation.
Let $Q_0\subset\Rn1$ be a cube as in the Main Lemma and let us denote $S\coloneqq   A_s(z_{Q_0})^{1/2}$, where $z_{{Q_0}}$ is the center of ${Q_0}$. We choose the map $\varphi$ so that $\varphi(x)=Sx$. By Lemma \ref{lemma_change_matrix_identity}  we have that $A_\varphi(\varphi^{-1}(z_{Q_0}))=Id$.
Denoting $\nu= \varphi^{-1}\sharp \mu$ and arguing as in \cite[Section 6]{PPT}, Lemma \ref{pot_changed} gives
\begin{equation}
\int_{{Q_0}}\big|T\mu(x)-m_{\mu, {{Q_0}}}(T\mu)\big|^2 d\mu(x)\approx \int_{\varphi^{-1}({Q_0})}\big|T_\varphi\nu(x)-m_{\nu,\varphi^{-1}({Q_0})}(T_\varphi \nu)\big|^2d\nu(x)
\end{equation}
and
\begin{equation}
\|T_\varphi\nu\|_{L^2\big(\nu|_{\varphi^{-1}(2{Q_0})}\big)}\approx \|T\mu\|_{L^2\big(\mu|_{(2{Q_0})}\big)},
\end{equation}
the implicit constants in the formulas above depending only on $\varphi$ and, hence, on the ellipticity of the matrix $A$.

Using these facts and Lemma \ref{lemma_mod_alpha}, in order to prove Lemma \ref{mainlemma_origin} it suffices to study the variant that stated below.
\begin{lemm}\label{main_lemma}
Let $n>1$ and let $C_0,C_1>0$ be some arbitrary constants. There exists $M=M(C_0,C_1, n)>1$ big enough, $\lambda(C_0,C_1, n)>0$ small enough and $\tilde\epsilon=\tilde\epsilon(C_0,C_1,M, n)>0$ small enough such that if $\delta=\delta(M,C_0,C_1,n)>0$ is small enough, then the following holds. Let $\mu$ be a Radon measure in $\mathbb{R}^{n+1}$, $Q_0\subset\mathbb{R}^{n+1}$ a cube centered at the origin and $\nu\coloneqq \varphi^{-1}\sharp\mu$, $\varphi$ being as in the comments before the lemma, satisfying the following properties:
\begin{enumerate}
\item $A_\varphi\big(\varphi^{-1}(0)\big)=Id$.
\item $\ell(MQ_0)\leq \lambda.$
\item $\nu\big(\varphi^{-1}(Q_0)\big)=\ell(Q_0)^n$.
\item $P_{\nu,\alpha/2}\big(\varphi^{-1}(MQ_0)\big)\leq C_0$.
\item For all $x\in 2Q_0$ and $0<r\leq \ell(\widetilde{Q}),$ $\Theta_{\mu}(B(x,r))\leq C_0.$
\item $Q_0$ has $C_0$-thin boundary.
\item $\alpha^{\varphi^{-1}(H)}_\nu\big(\varphi^{-1}(3MQ_0)\big)\leq\delta,$ where $H=\{x\in\mathbb{R}^{n+1}\colon x_{n+1}=0\}.$
\item $T_{\varphi,\nu|_{\varphi^{-1}(2 Q_0)}}$ is bounded on $L^2(\nu|_{\varphi^{-1}(2 Q_0)})$ with 
\begin{equation}
\big\|T_{\varphi,\nu|_{\varphi^{-1}(2Q_0)}}\big\|_{L^2(\nu|_{\varphi^{-1}(2 Q_0)})\to L^2(\nu|_{\varphi^{-1}(2Q_0)})}\leq C_1.
\end{equation}
\item we have
\begin{equation}
\int_{\varphi^{-1}(Q_0)}\big|T_\varphi\nu(x)-m_{\nu,\varphi^{-1}({Q_0})}(T_\varphi \nu)\big|^2d\nu(x)\leq \tilde\epsilon\nu\big(\varphi^{-1}(Q_0)\big).
\end{equation}
\end{enumerate}
Then there exists some constant $\tau>0$ and a uniformly $n$-rectifiable set $\Gamma\subset \mathbb{R}^{n+1}$ such that
\begin{equation}
\mu(Q_0\cap\Gamma)\geq \tau \mu(Q_0),
\end{equation}
where the constant $\tau$ and the UR constants of $\Gamma$ depend on all the constants above.
\end{lemm}

The aim of most of the rest of the paper is to provide the proof of this result.

In what follows, for the sake of simplicity of the notation, we will assume that $A(0)=A(z_{Q_0})=Id$, which in particular gives that $\varphi=Id$, $\mu=\nu$ and $T_{\varphi,\mu}=T_\mu$.
Indeed, if this is not the case, we should carry the following proofs for the image of cubes via $\varphi^{-1}$, periodize with respect to the image of a lattice of standard cubes and work with $T_\varphi$ instead of $T$. This would be a merely notational complication that we prefer to avoid to make the arguments more accessible.



\vspace{2mm}
\noindent{\bf Reduction to a periodic matrix.}\label{section_reduction_periodic_matrix}
The forthcoming lemma shows, roughly speaking, that the local structure of  $A$ close to $Q_0$ is what matters to the purposes of Lemma \ref{mainlemma_origin}. An immediate consequence of this fact is that, without loss of generality, we can replace the matrix $A$ with a periodic matrix, provided that the new matrix coincides with $A$ in a suitable neighborhood of the cube $Q_0$.

In what follows, we assume the matrix $\bar A$ to have H\"older continuous coefficients of exponent $\tilde\alpha<\alpha$ for technical reasons that will result clearer later on.
\begin{lemm}\label{lemm_modification}
Let $\bar{A}\in\Lambda_{\tilde\alpha}$ be such that $\bar{A}(x)=A(x)$ for every $x\in 2Q_0$. Let $\bar{T}$ denote the gradient of the single layer potential associated with $\bar{A}$. The operator $T_{\mu|2Q_0}$ is bounded in $L^2(\mu|_{2Q_0})$ if and only if $\bar{T}_{\mu|_{2Q_0}}$ is bounded in $L^2(\mu|_{2Q_0})$ and
\begin{equation}\label{estim_norm_modification}
\|T_{\mu|_{2Q_0}}\|_{L^2(\mu|_{2Q_0})\to L^2(\mu|_{2Q_0})} = \|\bar{T}_{\mu|_{2Q_0}}\|_{L^2(\mu|_{2Q_0})\to L^2(\mu|_{2Q_0})}  + O\big(\ell(Q_0)^{\tilde\alpha}\big).
\end{equation}
Moreover we have
{\begin{equation}\label{estim_mean_modification}
\begin{split}
\int_{Q_0}|T\mu(x)-m_{\mu,Q_0}(T\mu)|&^2d\mu(x)\\
&\lesssim \int_{Q_0}|\bar{T}\mu(x)-m_{\mu,Q_0}(\bar{T}\mu)|^2 d\mu(x) + \big(\ell(M Q_0)^{2\tilde\alpha} + M^{-\tilde\alpha}\big)^2\mu(Q_0),
\end{split}
\end{equation}}
where $M$ is as in the statement of Lemma \ref{mainlemma_origin}{ and the implicit constant in \eqref{estim_mean_modification} depends on $\diam(\supp\mu)$.}
\end{lemm}
The proof of Lemma \ref{lemm_modification} relies on the fact that $\Theta(\cdot,\cdot;A(x))=\Theta(\cdot,\cdot;\bar{A}(x))$ for every $x\in 2Q_0$ and it is very similar to the one of Lemma \ref{lemm_symetrization}, so that we omit it.

In the rest of the paper, without additional specifications, we will deal with a matrix $\bar A$ periodic with period $\ell$, $2\ell(Q_0)<\ell\lesssim \ell(Q_0)$.

\vspace{2mm}
\noindent{\bf The definition of the matrix $\bar A$.}
The construction in the present subsection is dictated by the necessity of having an auxiliary matrix which agrees with $A$ on $2Q_0$ and has the further properties of being periodic (which is crucial to use the estimates of the theory of homogenization) and of presenting `additional simmetries' with respect to reflections (see the forthcoming Lemma \ref{lemm_symm_mod_mat}). For a scheme of this construction we also refer to Figure 1.

Let $e_j$ denote the $j$-th element of the canonical basis of $\Rn1$. 
We denote by $\psi_j\colon\Rn1\to\Rn1$ the map
\begin{equation}\label{psijdef}
\psi_j(x)\coloneqq   x+(3\ell(Q_0)-2x_j)e_j,
\end{equation}
which corresponds to the reflection across the hyperplane $P_j$ orthogonal to $e_j$ and which passes through the point $\frac{3}{2}\ell(Q_0) e_j.$
Let $0<\delta<1/10$.
\begin{center}
\begin{figure}[h!]
\includegraphics[scale=0.5]{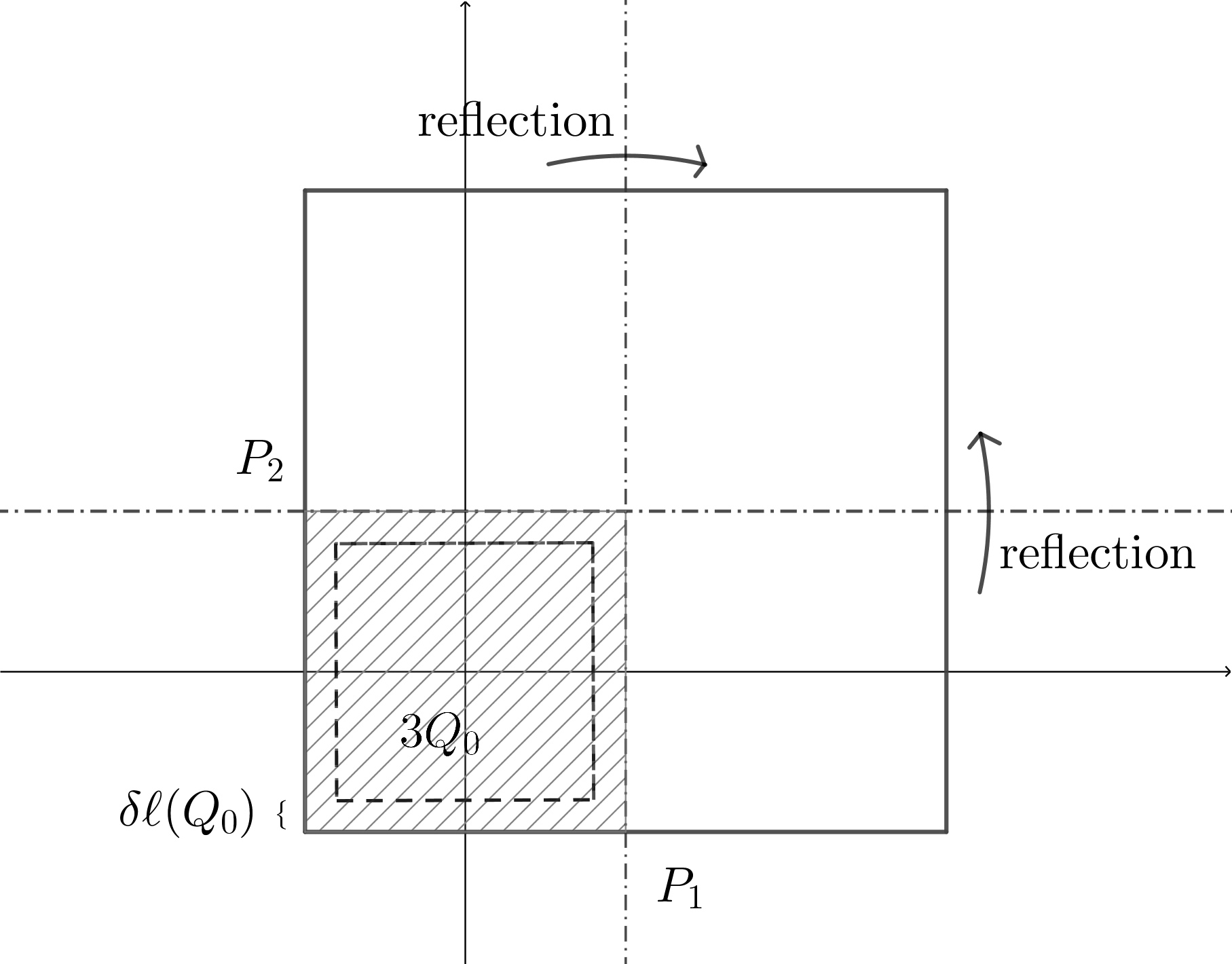}
\caption{A schematization of the construction of $\bar A$ at the level of the periodic unit.}
\end{figure}
\end{center}
Given a matrix $B(x)$ with variable coefficients, we define $B_{j}$ as
\begin{equation}\label{ch_var:A}
B_j=B_{\psi_j}=D({\psi_j^{-1}})(B\circ \psi_j)D(\psi_j^{-1})^T.
\end{equation}
Moreover, we define the matrix $\widetilde{B}$ as
\begin{equation}\label{def_a_hat_j}
\widetilde{B}(x)=
\begin{cases}
B(x) \qquad &\text{ for } \dist\big(x,\partial(3Q_0)\big)\geq \delta \ell(Q_0),\\
\frac{\dist(x,\partial(3Q_0))}{\delta \ell(Q_0)}B(x)+ \Big(1-\frac{\dist(x,\partial(3Q_0))}{\delta \ell(Q_0)}\Big)Id \qquad &\text{ for } \dist\big(x,\partial(3Q_0)\big)< \delta \ell(Q_0).
\end{cases}
\end{equation}
It is also useful to introduce the notation
\begin{equation}\label{mod_hat_j}
\widehat{B}_j(x)=
\begin{cases}
B(x) \qquad &\text{ for } x_j\leq \frac{3}{2}\ell(Q_0),\\
B_j(x) \qquad &\text{ for } x_j>\frac{3}{2}\ell(Q_0).
\end{cases}
\end{equation}
Let us apply the previous constructions to the matrix $A$. First, observe that the matrix $\hat{A}_j$ is not necessarily continuous. However, $\widehat{(\tilde{A})}_j$ is continuous because $Id_j=Id$ and $\tilde{A}|_{\partial(3Q_0)}\equiv Id.$
Our aim, now, is to define the final auxiliary matrix $\bar A$ by an iteration of the construction in \eqref{mod_hat_j} along every direction and which is followed by a periodization. 
Before doing so, let us observe that for $i, j\in\{1,\ldots,n+1\}$,
\begin{equation}
(\widetilde A_i)_j(x)=(\widetilde A_j)_i(x), \qquad x\in\Rn1.
\end{equation}
This follows directly from \eqref{ch_var:A} using the facts that $\psi_i\big(\psi_j(x)\big)=\psi_j\big(\psi_i(x)\big)$ and that the matrices $D(\psi_i^{-1}), D(\psi_j^{-1})$ are diagonal.
Thus by the linearity of the interpolation in \eqref{def_a_hat_j} we have that
\begin{equation}\label{welldefhat}
\widehat{\big(\widehat{(\widetilde A)}_i\big)}_j=\widehat{\big(\widehat{(\widetilde A)}_j\big)}_i \eqqcolon \widehat{(\widetilde A)}_{i,j},
\end{equation}
so the order of the modifications is not relevant.

Let us now construct the matrix $\bar A$ in two steps:
\begin{itemize}
\item For $x$ belonging to the cube of side length $6\ell(Q_0)$ centered at the point with coordinates $\frac{3}{2}\ell(Q_0)(1,\ldots, 1)$ we define
\begin{equation}
\bar A(x)\coloneqq \widehat{(\widetilde A)}_{1, \ldots, n+1}.
\end{equation}
\item By \eqref{def_a_hat_j}, the matrix $\bar A$ defined in the first step coincide with $Id$ for $x$ belonging to the boundary of the cube with side length $6\ell(Q_0)$ and centered at $\frac{3}{2}\ell(Q_0)(1,\ldots, 1)$. Hence, $\bar A$ admits a continuous and $6\ell(Q_0)$-periodic extension to $\mathbb{R}^{n+1}$ so that
\begin{equation}
\bar{A}(x)=\bar{A}\big(x+6\vec{k}\ell(Q_0)\big)
\end{equation}
for every $\vec k \in \mathbb{Z}^{n+1}$.
\end{itemize}

The following holds.
\begin{lemm}\label{lemma_second_modification}The matrix $\bar{A}$ is well-defined, H\"older continuous with exponent $\alpha/2^{n+1}$ and periodic with period $6\ell(Q_0)$.
\end{lemm}
The well-definition of $\bar{A}$ follows from \eqref{welldefhat}. The proof of the H\"older regularity is a minor variation of that of \cite[Lemma 8.1]{PPT}, where a similar modification of the matrix was involved. In particular, the exponent $\alpha/2^{n+1}$ is given by the fact that every reflection of the matrix across a hyperplane halves the order of the H\"older regularity. We also point out that, being $\bar A$ periodic, there is no need to introduce a radial cut-off for the matrix as in \cite{PPT}.

For the rest of the paper we use the notation $\tilde{\alpha}\coloneqq \alpha/2^{n+1}.$

\vspace{2mm}
\noindent{\bf Properties of $\mathcal E_{\bar A}$.}
As a consequence of the definition of $\bar{A}$ and, more specifically, of its periodicity and the fact that by construction	
\begin{equation}
\bar A_j(x)=\bar A(x)
\end{equation}
for every $x\in\mathbb{R}^{n+1}$ and $j=1,\ldots, n+1$, we have the following.
\begin{lemm}\label{lemm_symm_mod_mat}
\begin{equation}\label{wdabar}
\E_{\bar A}(x,y)=\E_{\bar A}(\psi_j(x),\psi_j(y)) \qquad\text{ for }\qquad j=1,\ldots,n+1
\end{equation}
and
\begin{equation}\label{wdabarper}
\E_{\bar A}(x,y)=\E_{\bar A}\big(x+6\vec{k}\ell(Q_0),y+6\vec{k}\ell(Q_0)\big) \qquad \text{ for }\qquad\vec{k}\in\mathbb{Z}^{n+1}.
\end{equation}
\end{lemm}

By Lemma \ref{CZ_1}, the function $\bar K=\n1\E_{\bar A}(\cdot,\cdot)$ is (globally) a Calder\'on-Zygmund kernel. In particular
\begin{itemize}
\item[(a)] $|\bar K(x,y)| \lesssim |x-y|^{-n}$ for all $x,y\in \R^{n+1}$ with $x \not=y$.
\item[(b)] $|\bar K(x,y)-\bar K(x,y')| + |\bar K(y,x) - \bar K(y',x)| \lesssim |y-y'|^{\tilde\alpha} |x-y|^{-n-{\tilde\alpha} } \,$ for $\,2|y-y'| \leq |x-y|$.
\end{itemize}
Let $\bar T_\mu$ denote the singular integral operator associated with $\bar K$,
\begin{equation}
\bar T_\mu f(x)=\int \bar K(x,y)f(y)d\mu(y).
\end{equation}
Lemma \ref{lemm_modification} tells that we can prove the Main Lemma for $\bar{T}$ instead of $T$, possibly by slightly worsening the parameters involved.

\section{A first localization lemma}
It is useful to provide a local analogue of the BMO-type estimate \eqref{bmotypee}. This is possible because of the smallness of the $\alpha$-number and the bound for the $P_{\mu,\tilde\alpha}$-density. Also, recall that because of the assumptions in Lemma \ref{mainlemma_origin}, we have $\mu(MQ_0)\lesssim M^n\mu(Q_0)$. In what follows we sketch the proof of the localization of \eqref{bmotypee} for $\bar T_\mu$, highlighting the differences with respect to the case of the Riesz transform (see \cite[Lemma 4.2]{GT}).

In the rest of the paper we omit to indicate the dependence of the implicit constants in our estimates on $C_0$ and $C_1$.
\begin{lemm}\label{localization_BMO_lemma}For $\delta$ small enough depending on $M$, the following inequality holds
\begin{equation}\label{localization_BMO}
\int_{Q_0}|\bar T_\mu \chi_{MQ_0}|^2d\mu\lesssim \Big(\epsilon+\frac{1}{M^{2\tilde\alpha }}+ M^{4n+2}\delta^{1/(4n+4)} + (M\ell(Q_0))^{2\tilde\alpha}\Big)\mu(Q_0).
\end{equation}
\end{lemm}
\begin{proof}
First, observe that
\begin{equation}\label{first_eq}
\int_{Q_0}|\bar T_\mu(\chi_{MQ_0})|^2d\mu\leq 2 \int_{Q_0}|\bar T_\mu(\chi_{MQ_0})- m_{\mu,Q_0}(\bar T_\mu \chi_{MQ_0}) |^2 d\mu + 2|m_{\mu,Q_0}(\bar T_\mu \chi_{MQ_0})|^2\mu(Q_0).
\end{equation}
Let us estimate the two summands on the right hand side of \eqref{first_eq} separately.
To study the first one, we write
\begin{equation}\label{first_part}
\begin{split}
&\int_{Q_0}\big|\bar T_\mu\chi_{MQ_0}- m_{\mu,Q_0}(\bar T_\mu \chi_{MQ_0}) \big|^2 d\mu \\
&\qquad\leq 2 \int_{Q_0}\big|\bar T_\mu\chi_{{(MQ_0)}^c}(x)-m_{\mu,Q_0}(\bar T_\mu\chi_{{(MQ_0)}^c})\big|^2d\mu(x)
+ 2\int_{Q_0}|\bar T\mu -m_{\mu,Q_0}(\bar T\mu)|^2d\mu.
\end{split}
\end{equation}
Applying Lemma \ref{lemcz}, it follows that for $x,y\in Q_0$
\begin{equation}
\begin{split}
&|\bar T_\mu\chi_{(MQ_0)^c}(x)-\bar T_\mu\chi_{(MQ_0)^c}(y)|\leq \int_{(MQ_0)^c}|\bar K(x,z)-\bar K(y,z)|d\mu(z)\\
&\qquad\lesssim |x-y|^{\tilde\alpha }\int_{(MQ_0)^c}\frac{1}{|x-z|^{n+{\tilde \alpha} }}d\mu(z)\\
&\qquad\lesssim |x-y|^{\tilde\alpha}\sum_{j=1}^\infty\int_{2^{j+1}MQ_0\setminus 
2^jMQ_0}\frac{1}{|x-z|^{n+{\tilde\alpha} }}d\mu(z)
\lesssim \frac{|x-y|^{\tilde\alpha }}{\ell(MQ_0)^{\tilde\alpha}}P_{\mu, {\tilde\alpha} }(MQ_0)
\lesssim \frac{1}{M^{\tilde\alpha }},
\end{split}
\end{equation}
being $P_{\mu, \tilde\alpha}(MQ_0)\lesssim 1.$ Then, averaging the previous inequality over the variable $y$, we get
\begin{equation}
\big|\bar T_\mu\chi_{{(MQ_0)}^c}(x)-m_{\mu,Q_0}(\bar T_\mu\chi_{{(MQ_0)}^c})\big|\lesssim \frac{1}{M^{\tilde\alpha }}
\end{equation}
and
\begin{equation}
\int_{Q_0}|\bar T_\mu\chi_{{(MQ_0)}^c}(x)-m_{\mu,Q_0}(\bar T_\mu\chi_{{(MQ_0)}^c})|^2d\mu(x)\lesssim\frac{1}{M^{2\tilde\alpha}}\mu(Q_0).
\end{equation}
Recalling that by hypothesis we have
\begin{equation}
\int_{Q_0}\big|\bar T\mu-m_{\mu,Q_0}(\bar T\mu)\big|^2d\mu\leq \epsilon\mu(Q_0),
\end{equation}
we can estimate \eqref{first_part} as
\begin{equation}\label{second_part}
\int_{Q_0}\big|\bar T_\mu(\chi_{(MQ_0)^c})- m_{\mu,Q_0}(\bar T_\mu \chi_{MQ_0}) \big|^2 d\mu \lesssim \Big(\epsilon+\frac{1}{M^{2\tilde\alpha}}\Big)\mu(Q_0).
\end{equation}
An application of Lemma \ref{lemm_estim_freezing_periodic} together with the antisimmetry of $\n1\Theta(\cdot,\cdot;\bar A(x))$ also gives
\begin{equation}\label{mean_error_loc_cube}
\big|m_{\mu,Q_0}(\bar T_\mu\chi_{Q_0})\big|\lesssim \frac{1}{\mu(Q_0)}\int_{Q_0}\int_{Q_0}|x-y|^{-n+\tilde\alpha}d\mu(x)d\mu(y)\lesssim \ell(Q_0)^{\tilde\alpha}.
\end{equation}
Minor variations of the arguments which prove \cite[(4.2)]{GT} show that that
\begin{equation}\label{mean_lemma_local_first}
\begin{split}
|m_{\mu,Q_0}(\bar T_\mu \chi_{MQ_0})|&\overset{\eqref{mean_error_loc_cube}}{\lesssim} |m_{\mu,Q_0}(\bar T_\mu \chi_{MQ_0\setminus Q_0})| + \ell(Q_0)^{\tilde{\alpha}}\\
&\lesssim  M^{2n+1}\delta^{1/8(n+1)} + \big(M\ell(Q_0)\big)^{\tilde\alpha}+ \ell(Q_0)^{\tilde{\alpha}}\lesssim 
M^{2n+1}\delta^{1/8(n+1)} + \big(M\ell(Q_0)\big)^{\tilde\alpha}.
\end{split}
\end{equation}
For the sake of brevity we omit the details and we just point out that the presence of the second summand on the right hand side comes from the estimate
\begin{equation}\label{estim_planes}
\Big|\int_{Q_0}\bar T\big(\varphi\mathcal H^n|_H\big)\, d\mathcal H^n|_H\Big|\lesssim \big(M\ell(Q_0)\big)^{\tilde{\alpha}}\ell(Q_0)^n,
\end{equation}
where $\varphi$ is a proper even $C^1$ function with $0\leq \varphi\leq 1$ and supported on $MQ_0\setminus Q_0$.
{To get the estimate \eqref{estim_planes}, we just write
\begin{equation}
\begin{split}
\Big|\int_{Q_0}&\bar T\big(\varphi\mathcal H^n|_H\big)\, d\mathcal H^n|_H\Big|\\
& \leq \int_{Q_0}\int_{MQ_0}\Big|\frac{1}{2}\bar K(x,y)-\frac{1}{2}\nabla_1\Theta(x,y;\bar A(x))\Big|\,d\mathcal{H}^n|_H(x)d\mathcal{H}^n|_H(y)\\
&\qquad + \int_{Q_0}\int_{MQ_0}\Big|\frac{1}{2}\bar K(x,y)-\frac{1}{2}\nabla_1\Theta(x,y;\bar A(y))\Big|\,d\mathcal{H}^n|_H(x)d\mathcal{H}^n|_H(y)\\
&\qquad + \frac{1}{2}\Big|\int_{Q_0}\int_{MQ_0}\big(\nabla_1\Theta(x,y;\bar A(x))+\nabla_1\Theta(x,y;\bar A(y))\big)\,d\mathcal{H}^n|_H(x)d\mathcal{H}^n|_H(y)\Big|.
\end{split}
\end{equation}
Then, the third summand is null because of the antisymmetry of its integrand and the first two terms can be estimated via Lemma \ref{lemm_freezing}.}

Gathering \eqref{first_eq}, \eqref{second_part} and \eqref{mean_lemma_local_first} we are able to conclude the proof of the lemma.
\end{proof}
\section{The David and Mattila lattice associated with $\mu$ and its properties}\label{DM_section}
The dyadic lattice constructed by David and Mattila \cite[Theorem 3.2]{DM} is a powerful tool in the study of the geometry of Radon measures. Its main properties are listed in the following lemma, that we state for a general Radon measure with compact support.
\begin{lemm}[David and Mattila]
Let $\sigma$ be a compactly supported Radon measure in $\mathbb{R}^{n+1}$. Consider
two constants $K_0>1$ and $A_0>5000 K_0$ and denote $W=\supp \sigma$. Then there exists a sequence of
partitions of $W$ into Borel subsets $Q$, $Q \in \D_{\sigma,k}$, which we will refer to as cells, with the following
properties:
\begin{itemize}
\item  For each integer $k\geq 0$, $W$ is the disjoint union of the cells $Q$, $Q\in\D_{\sigma,k}$. If $k < l$,
$Q \in \D_{\sigma,l}$, and $R\in\D_{\sigma,k}$ , then either $Q\cap R=\emptyset$ or $Q\subset R$.
\item 
 For each $k\geq 0$ and each
cell $Q \in \D_{\sigma,k}$, there is a ball $B(Q) = B\big(z_Q ,r(Q)\big)$ such that
\begin{align}
z_Q\in W, \,\, &A_0^{-k}\leq r(Q)\leq K_0A_0^{-k}\\
W\cap B(Q)\subset Q \subset W\cap& 28 B(Q)= W\cap B\big(z_Q,28 r(Q)\big),
\end{align}
and the balls $5B(Q),Q\in\D_{\sigma,k}$ are disjoint.
\item The cells $Q\in\D_{\sigma,k}$ have small boundaries. By this, we mean that for each $Q\in\D_{\sigma,k}$ and each integer $l\geq 0$, if we set
\begin{align}
N_l^{\inter}&\coloneqq   \big\{x\in Q: \dist (x,W\setminus Q)<A_0^{-k-l}\big\}\\
N_l^{\ext}(Q)&\coloneqq   \big\{x\in W\setminus Q:\dist(x,Q)<A_0^{-k-l}\big\}
\end{align}
and
\begin{equation}
N_l(Q)\coloneqq   N_l^{\inter}(Q)\cup N_l^{\ext}(Q),
\end{equation}
we get
\begin{equation}
\sigma\big(N_l(Q)\big)\leq \big(C^{-1}K_0^{-3(n+1)-1}A_0\big)^{-l}\sigma\big(90B(Q)\big)
\end{equation}
\item Denote by $\D^{\db}_{\sigma,k}$
the family of cells $Q\in \D_{\sigma,k}$ for which
\begin{equation}
\sigma\big(100B(Q)\big) \leq K_0\sigma\big(B(Q)\big).
\end{equation}
We have that $r(Q)= A_0^{-k}$ when $Q\in\D_{\sigma,k}\setminus \D^{\db}_{\sigma,k}$
and

\begin{equation}\label{ndb}
\sigma\big(100B(Q)\big) \leq K_0^{-1}\sigma\big(100^{l+1}B(Q)\big)
\end{equation}
for all $l \geq 1$ with $100^l \leq K_0$ and $Q \in \D_{\sigma,k}\setminus \D^{\db}_{\sigma,k}.$
\end{itemize}
\end{lemm}
Let us denote $\D_\sigma\coloneqq  \bigcup_k\D_{\sigma, k}$. Let us choose $A_0$ big enough so that
\begin{equation}\label{assumption_a_0}
C^{-1}K_0^{-3(n+1)-1}A_0>A_0^{1/2}>10.
\end{equation}
Here we list some useful quantities associated with each cell $Q\in\D_{\sigma,k}$:
\begin{itemize}
\item $J(Q)\coloneqq   k$, which may be interpreted as the \textit{generation} of $Q$.
\item $\ell(Q)\coloneqq   56 K_0 A^{-k}_0$, that we also call \textit{side length}. Notice that
\begin{equation}
\frac{1}{28}K_0^{-1}\ell(Q)\leq \diam(28 B(Q))\leq \ell(Q)
\end{equation}
and $r(Q)\approx\diam(Q)\approx \ell(Q).$
\item calling $z_Q$ the \textit{center} of $Q$, we denote $B_Q\coloneqq   28B(Q)=B(z_Q,28r(Q)),$ which in particular gives
\begin{equation}
Q\cap\frac{1}{28}B_Q\subset Q\subset B_Q.
\end{equation}
\end{itemize}
We recall, now, some of the properties of the cells in the David and Mattila lattice.

The choice in \eqref{assumption_a_0} implies, for $0<\lambda\leq 1$, the estimate
\begin{equation}
\begin{split}
\sigma\big(\{x\in Q:\dist(x,W\setminus Q)\leq \lambda \ell(Q)\}\big) + \sigma\big(\{x\in 3.5B_Q\setminus Q:\dist(x, Q)\leq \lambda \ell(Q)\}\big)\\ \leq c\lambda^{1/2}\sigma(3.5 B_Q).
\end{split}
\end{equation}
We denote $\D^{\db}_\sigma\coloneqq   \bigcup_{k\geq 0}\D^{\db}_{\sigma,k}$ and we say that it is the lattice of \textit{doubling cells}. This notation is justified by the fact that, for $Q\in\D^{\db}_\sigma,$ we have
\begin{equation}
\sigma(3.5 B_Q)\leq \sigma(100 B(Q))\leq K_0 \sigma(B(Q))\leq K_0\sigma(Q).
\end{equation}
An important feature of the David and Mattila lattice is that every cell $Q\in\D_\sigma$ can be covered by doubling cells up to a set of $\sigma$-measure zero (\cite[Lemma 5.28]{DM}). Moreover, if we have two cells $R,Q\in D_\sigma$ with $Q\subset R$ and such that every intermediate cell $Q\subsetneq S \subsetneq R$ belongs to $\D_\sigma\setminus\D^{\db}_\sigma,$ we have the control
\begin{equation}\label{dec_measure_chain_non_doubling}
\sigma(100B(Q))\leq A_0^{-10n(J(Q)-J(R)-1)}\sigma(100 B(R))
\end{equation}
on the decay of the measure. The estimate \eqref{dec_measure_chain_non_doubling} is proved via an iterated application of the inequality
\begin{equation}\label{estim_DM_for_decay}
\sigma(100B(Q))\leq A_0^{-10n}\sigma(100 B(\hat{Q})),
\end{equation}
where $\hat{Q}$ is the cell from $\D_{\sigma,J(Q)-1}$ containing $Q$ (also called \textit{parent} of $Q$). We remark that \eqref{estim_DM_for_decay} follows by \eqref{ndb} and a proper choice of $A_0$ and $K_0$ (see \cite[Lemma 5.31]{DM}).

For $Q\in\D_\sigma,$ we denote by $\D_\sigma(Q)$ the cells in $\D_\sigma$ which are contained in $Q$ and $\D^{\db}_\sigma(Q)\coloneqq   \D_\sigma(Q)\cap \D^{\db}_\sigma.$ 

\section{The Key Lemma, the stopping time condition and a first modification of the measure}\label{section_key_lemma}
{The hearth of the proof of Lemma \ref{main_lemma} consists in providing a control on the abundance of cells with low density (in some sense that we clarify below). The whole construction that we are about to discuss depends on some auxiliary parameter to be chosen properly later in the proof.}
\begin{defin}[Low density cells] Let $0<\theta_0\ll 1$. A cell $Q\in\D_\mu$ is said to be of low density if
\begin{equation}
\Theta_\mu(3.5 B_Q)\leq \theta_0
\end{equation}
and it has maximal side length. We denote by $\LD$ the family of low density cells.
\end{defin}
Most of the rest of the paper deals with the proof of the fact that the low density cells fail to cover a significant portion of $Q_0$.
\begin{lemm}[Key Lemma]\label{key_lemma} Let $\epsilon,\delta$ and $M$ be as in Lemma \ref{mainlemma_origin}.
There exists $\epsilon_0>0$ such that if $M$ is big enough and $\theta_0, \delta$ and $\epsilon$ are small enough, then
\begin{equation}\label{abundance_LD}
\mu\Bigg(Q_0\setminus \bigcup_{Q\in\LD}Q\Bigg)\geq \epsilon_0 \mu(Q_0).
\end{equation}
\end{lemm}
To prove the main Lemma \ref{main_lemma} using the results in the Key Lemma, it suffices to refer to the construction in \cite[Section 10]{GT}, which relies on a subtle covering argument together with the connection between uniform rectifiability and the Riesz transform, and invoke \cite[Theorem 1.1 and Theorem 1.2]{PPT} in place of the results of Nazarov, Tolsa and Volberg. So, the rest of the present article (a part from the last section) is devoted to the proof of Lemma \ref{key_lemma}.\\

We argue by contradiction: assume that \eqref{abundance_LD} does not hold, that is to say
\begin{equation}\label{contrad_LD_hyp}
\mu\Bigg(\bigcup_{Q\in\LD}Q\Bigg)>(1-\epsilon_0) \mu(Q_0).
\end{equation}
More specifically, we want to show that a choice of $\epsilon_0$ small enough leads to an absurd. The proof is based on a stopping time argument. Roughly speaking, for $Q\in\LD$, we say that a cell $R$ belongs to its associated stopping family if it is a descendant of $Q$ (i.e. $R\subset Q$) and it is sufficiently small. The definition of stopping cells depends on a parameter $t$, which has to be thought  small and that will be appropriately chosen later.
\begin{defin}[Stopping cells]\label{def_stop_cells}
Let $Q\in\LD.$ Let $0<t<1.$ We say that $R\in\Stop(Q)$ if the two following conditions are verified and it has maximal side length
\begin{itemize}
\item $R\in\D^{\db}_\mu$, $R\subset Q$.
\item $\ell(R)\leq t \ell(Q)$.
\end{itemize}
We also denote $\Stop\coloneqq  \bigcup_{Q\in\LD}\Stop(Q)$ the family of all the stopping cells.
\end{defin}
Assuming that the stopping cells in $\Stop(Q)$ are doubling makes sense in light of the fact that doubling cells cover $Q$ up to a set of $\mu$-measure zero. In particular, this implies that \eqref{contrad_LD_hyp} is equivalent to
\begin{equation}
\mu\Bigg(\bigcup_{Q\in\Stop}Q\Bigg)>(1-\epsilon_0) \mu(Q_0).
\end{equation}
The proof of the Key Lemma \ref{key_lemma} involves a periodization of the measure $\mu,$ which is essentially carried out by replicating $\mu|Q_0$ on the horizontal plan according to the periodicity of the matrix $\bar A$. 

The cells close to the boundary of $Q_0$ may give problems, so that our first temptation would be to try not to incorporate them into the contruction. This is possible just in the case their contribution to the measure of $Q_0$ is negligible. So, we say that $P\in\Bad$ if $P\in\Stop$ and $1.1B_P\cap \partial Q_0\neq\emptyset.$

Another technical problem is that $\Stop$ may contain infinitely many cells. This second difficulty can be easily overcome considering a finite family of cells, named $\Stop_0,$ which contains a big portion of the measure of $\Stop$, e.g.
\begin{equation}\label{def_stop_0}
\mu\Bigg(\bigcup_{Q\in\Stop_0}Q\Bigg)>(1-2\epsilon_0) \mu(Q_0).
\end{equation}
The rest of the section is devoted to a justification of the last affirmations concerning $\Bad$ and the first modification of the measure $\mu$. It is essentially a rewriting of \cite[Lemma 6.2, Lemma 6.3, Lemma 6.4]{GT} in our context, in which we highlight the right homogeneities coming from our elliptic setting.

The following lemma contains an estimate of the density $P_{\mu,\tilde\alpha}$ of the stopping cells in terms of the low density parameter $\theta_0$.
\begin{lemm}\label{alpha_dens}
Let $Q\in\Stop$ and let $t=\theta_0^{1/(n+\tilde\alpha)}$. We have
\begin{equation}
\Theta_{\mu}(2B_Q)\leq P_{\mu,{\tilde\alpha} }(2B_Q)\lesssim \theta_0^{\frac{\tilde\alpha}{n+\tilde\alpha}}.
\end{equation}
\end{lemm}
\begin{proof}
The first inequality is an immediate consequence of the definition of $P_{\mu,\tilde\alpha}$.  To prove the second inequality, we consider the maximal cell $R'\in\D_\mu$ such that $Q\subset R'\subset R$ and $\ell(R')\leq t \ell(R)$ and write
\begin{equation}\label{lemm_poisson_density_stopping}
\begin{split}
P_{\mu,{\tilde\alpha} }(2B_Q)\lesssim & \sum_{P\in\D_\sigma:Q\subset P\subset R'}\Theta_\mu(2B_P)\Big(\frac{\ell(Q)}{\ell(P)}\Big)^{\tilde\alpha} +\sum_{P\in\D_\sigma:R'\subset P\subset R}\Theta_\mu(2B_P)\Big(\frac{\ell(Q)}{\ell(P)}\Big)^{\tilde\alpha} \\
&+\sum_{P\in\D_\sigma:R\subset P\subset Q_0}\Theta_\mu(2B_P)\Big(\frac{\ell(Q)}{\ell(P)}\Big)^{\tilde\alpha} 
+\sum_{k\geq 1}2^{-k{\tilde\alpha} }\Theta_\mu(2^kB_P) \\
= &I + II + III + IV.
\end{split}
\end{equation}
Then, the estimates work as in the case of the Riesz transform. In particular, the same arguments prove
\begin{equation}
I + II\lesssim \frac{\theta_0}{t^n}
\end{equation}
and
\begin{equation}
III + IV\lesssim t^{\tilde\alpha} ,
\end{equation}
which justifies the choice of $t$ in the statement of the lemma.
\end{proof}
For the rest of the paper we assume $t =\theta_0^{1/(n+\tilde\alpha)}$ in Definition \ref{def_stop_cells}.

Using the estimates in Lemma \ref{alpha_dens}, one can prove (see \cite[Lemma 6.3]{GT}) that
\begin{equation}\label{measure_bad_cubes}
\mu\bigg(\bigcup_{\Bad} Q\bigg) \lesssim \theta_0^{\frac{\tilde\alpha}{n+\tilde\alpha}}\mu(Q_0).
\end{equation}
\vspace{2mm}
\noindent{\bf First modification of the measure.}
As already mentioned, for technical purposes it is useful to modify the measure inside $Q_0$ by taking just finitely many stopping cells and getting rid of the cells in $\Bad$. To make the previous statement rigorous, we choose a small parameter $0<\kappa_0\ll 1$ to be fixed later and, after denoting
\begin{equation}
I_{\kappa_0}(Q)\coloneqq  \{x\in Q: \dist(x,\supp\sigma\setminus Q)\geq\kappa_0 \ell(Q)\},
\end{equation}
we define the modified measure
\begin{equation}
\mu_0\coloneqq  \mu|_{Q^c_0}+\sum_{Q\in\Stop_0\setminus \Bad}\mu|_{I_{\kappa_0}(Q)}.
\end{equation}
Using \eqref{def_stop_0} and \eqref{measure_bad_cubes}, it is not difficult to prove that $\mu_0$ differs from  $\mu$, in the sense of the total mass, possibly by a very small quantity. Indeed,
\begin{equation}\label{drop_measure}
\|\mu-\mu_0\|\leq \Big(2\epsilon_0 + C\theta^{\tilde\alpha/(n+\tilde\alpha)}_0+\kappa_0^{1/2}\Big)\mu(Q_0).
\end{equation}
For this modification to be valid to our purposes, we need the gradient of the single layer potential associated with this measure to satisfy a localization estimate analogous to \eqref{localization_BMO}.
This is easily proved by gathering the $L^2(\mu|_{Q_0})$-boundedness of $\bar T_{\mu|_{Q_0}}$, the estimate \eqref{drop_measure} and the localization estimate \eqref{localization_BMO} for $\mu$ (see \cite[Lemma 6.4]{GT}).
\begin{lemm} \label{localization_lemma_first_modification}If $\delta$ is chosen small enough (depending on $M$), then
\begin{equation}
\int_{Q_0}|\bar T(\chi_{MQ_0}\mu_0)|^2d\mu_0\lesssim \Big(\epsilon+\frac{1}{M^{2\tilde\alpha}}+M^{4n+2}\delta^{1/(4n+4)}+(M\ell(Q_0))^{2\tilde\alpha}+\epsilon_0+\theta_0^{\tilde\alpha/(n+ \tilde\alpha)}+\kappa_0^{1/2}\Big)\mu(Q_0).
\end{equation}
\end{lemm}

\section{Periodization and smoothing of the measure	}
\vspace{2mm}
\noindent{\bf The periodization.}
{We want to get rid of the truncation at the level of $M\ell(Q_0)$ present in Lemma \ref{localization_lemma_first_modification}. This can be done replicating the measure periodically by means of horizontal translations. The localization of the gradient of the single layer potential associated with this auxiliary measure will eventually make us able to implement a variational argument in Section \ref{section_var_arg}.}

We denote by 
\begin{equation}
\M\coloneqq   \big\{Q_0+z_P:z_P\in 6\ell(Q_0)\mathbb{Z}^n\times \{0\}\big\}
\end{equation}
the family of disjoint cubes covering $H$ and obtained translating $Q_0$ along the coordinate (horizontal) axes. The factor 6 is chosen in order for this periodization to be coherent with the period of the matrix $\bar A$. Given $P\in\M$ we denote by $z_P$ its center and by $T_P\colon\Rn1\to\Rn1$ the translation
\begin{equation}
T_P(x)\coloneqq   x+z_P,
\end{equation}
so that the periodization of the measure reads
\begin{equation}
\widetilde{\mu}\coloneqq  \sum_{P\in\M}T_P\sharp \mu_0|_{Q_0}.
\end{equation}
{Observe that $\mu_0(\partial Q_0)=0,$ which implies $\chi_{Q_0}\widetilde{\mu}=\mu_0.$}

As for the first modification of the measure, we have to prove the equivalent of the localization Lemma \ref{localization_lemma_first_modification}.
This can be done as for the Riesz transform (see \cite[Lemma 7.2]{GT}), using that $\widetilde{\mu}$ is very flat at the level of $3MQ_0$.
\begin{lemm}
Let $\kappa_0,\theta_0$ and $\epsilon_0$ be as in Section \ref{section_key_lemma} and $\delta$ as in the Main Lemma.
Letting
\begin{equation}
\tilde{\delta}\coloneqq   M^{n+1}\Big(\epsilon_0+\theta_0^{\tilde\alpha/(n+\tilde \alpha)}+\kappa_0^{1/2}+\delta^{1/2}\Big),
\end{equation}
we have
\begin{equation}\label{estim_alpha_and_second_localization}
\alpha^H_{\widetilde{\mu}}(3MQ_0)\lesssim \tilde{\delta}.
\end{equation}
Moreover, for
\begin{equation}
\tilde{\epsilon}\coloneqq   \epsilon+\frac{1}{M^{2\tilde\alpha}}+M^{4n+2}\delta^{1/(4n+4)}+\epsilon_0+\theta_0^{\tilde\alpha/(n+ \tilde\alpha)}+\kappa_0^{1/2}+ M^{2n+2}\tilde{\delta}^{1/(4n+5)}+ {(M\ell(Q_0))^{2\tilde\alpha}}
\end{equation}
we have
\begin{equation}
\int_{Q_0}\big|\bar T(\chi_{MQ_0}\widetilde{\mu})\big|^2d\widetilde{\mu}\lesssim \tilde{\epsilon}\widetilde{\mu}(Q_0).
\end{equation}
\end{lemm}
It is not difficult to see that the measure $\widetilde{\mu}$ has polynomial growth:
\begin{equation}
\widetilde{\mu}(B(x,r))\lesssim r^n\quad\text{ for every $x\in\Rn1$ and $r>0.$}
\end{equation}
The following lemma contains a technical estimate for a suitably modified version of the density $P_{\widetilde{\mu},{\tilde\alpha}}(2B_Q)$.
\begin{lemm}\label{lemma_technical_int_dens}For every $Q\in\Stop_0\setminus \Bad$ the inequality
\begin{equation}
\int_{1.1B_Q\setminus Q}\int_Q\frac{1}{|x-y|^n}d\widetilde{\mu}(x)d\widetilde{\mu}(y)\lesssim \theta_0^{\frac{\tilde\alpha}{(n+\tilde\alpha)(1+2n)}}\widetilde{\mu}(Q)
\end{equation}
holds. Moreover, the function
\begin{equation}
p_{\widetilde{\mu},{\tilde\alpha}}(x)\coloneqq  \sum_{Q\in\Stop_0\setminus \Bad}\chi_Q P_{\widetilde{\mu},{\tilde\alpha}}(2B_Q)
\end{equation}
satisfies
\begin{equation}\label{estim_p_tilde_mu}
\int_{Q_0}p^2_{\widetilde{\mu},{\tilde\alpha}}d\widetilde{\mu}\lesssim \theta_0^{\frac{2\tilde\alpha}{(n+\tilde\alpha)(1+2\tilde\alpha)}}\widetilde{\mu}(Q_0).
\end{equation}
\end{lemm}
\begin{proof}[Remark on the proof.] 
In order to prove \eqref{estim_p_tilde_mu} it suffices to follow the path of \cite[Lemma 7.4]{GT} taking into consideration the right homogeneity given by $\tilde \alpha$, which leads to
\begin{equation}\label{ineq_lemma_density_proof}
\int_{Q_0}p^2_{\widetilde{\mu},{\tilde\alpha}}d\widetilde{\mu}\lesssim\Big({\bar\kappa} + \frac{\theta_0^{\frac{{2\tilde\alpha}}{(n+\tilde\alpha)}}}{{\bar\kappa}^{{2\tilde\alpha}}}+\theta_0^{\frac{\tilde\alpha}{n+\tilde\alpha}}\Big)\widetilde{\mu}(Q_0),
\end{equation}
where $0<\bar\kappa<1$ is a small constant. Inequality \eqref{ineq_lemma_density_proof} gives the desired estimate after making the choice $\bar\kappa=\theta_0^{{{2\tilde\alpha}}/[(n+\tilde\alpha)(1+2\tilde\alpha)]}.$
\end{proof}
\vspace{2mm}
\noindent{\bf The smoothing.}
A priori, the measure $\mu_0$ may not be absolutely continuous with respect to the Lebesgue measure on $\Rn1$. This would constitute a problem when trying to implement the variational techniques. For this reason, it it useful to consider the following further modification of the measure
\begin{equation}
\eta_0\coloneqq  \sum_{Q\in\Stop_0\setminus\Bad}\frac{\mu_0(Q)}{\mathcal H^{n+1}\big(\frac{1}{4}B(Q)\big)}\mathcal H^{n+1}|_{\frac{1}{4}B(Q)}
\end{equation}
and its periodization
\begin{equation}
\eta\coloneqq  \sum_{P\in\M}T_P\sharp\eta_0.
\end{equation}
We remark that, being $\Stop_0$ a finite family, the measures $\eta_0$ and $\eta$ both have bounded density with respect to $\mathcal H^{n+1}$. A specific control on the density is not relevant to the purposes of our proof. The following lemma contains a localization estimate for the potential associated with $\eta$.
\begin{lemm}
Denoting
\begin{equation}
\epsilon'\coloneqq  \widetilde{\epsilon} + \ell(Q_0)^{2\tilde\alpha} + M^n\kappa_0^{-2n-2\tilde\alpha}\theta_0^{\frac{2\tilde\alpha}{(n+\tilde\alpha)(1+2\tilde\alpha)}} + \theta_0^{\frac{2\tilde\alpha}{(n+\tilde\alpha)(1+2n)}},
\end{equation}
we have
\begin{equation}
\int_{Q_0}|\bar T(\chi_{MQ_0}\eta)|^2d\eta\lesssim\epsilon'\eta(Q_0).
\end{equation}
\end{lemm}
The presence of the summand $\ell(Q_0)^{2\tilde\alpha}$ in $\epsilon'$ (already taken into account in $\widetilde{\epsilon}$) to point out that, as in \eqref{mean_lemma_local_first}, the lack of antisimmetry of $\bar K(\cdot,\cdot)$ gives the error term
\begin{equation}
\big|m_{\widetilde{\mu},Q}\big(\bar T_{\widetilde{\mu}}\chi_Q\big)\big|\lesssim \ell(Q)^{\tilde\alpha}\lesssim \ell(Q_0)^{\tilde\alpha}
\end{equation}
for every $Q\in\Stop_0\setminus\Bad.$ This contribution is not present in the case of an elliptic matrix with constant coefficients. The rest of the proof is analogous to the one of \cite[Lemma 8.1]{GT} and all is needed is a careful check that Lemma \ref{lemma_technical_int_dens} applies and the new homogeneity does not affect the final result. We omit further details.
\begin{rem}Observe that the expressions of $\widetilde{\delta},\widetilde{\epsilon}$ and $\epsilon'$ all include a summand which depends on $\epsilon_0$. In particular, the quantities in question are small if $\epsilon_0$ and $M\ell(Q_0)$ are chosen small enough. Then the choice of $\varepsilon_0\ll 1$ (which is possible because we assumed \eqref{contrad_LD_hyp} to hold) gives the localization for the potentials associated with the auxiliary measures.
\end{rem}
\section{The localization of $\bar T\eta$}
Let $L^\infty_{\mathcal{M}}$ denote the set of functions $f\in L^\infty(\eta)$ such that
\begin{equation}
f(x+z_P)=f(x)
\end{equation}
for every $x\in\mathbb{R}^{n+1}$ and $P\in\mathcal{M}.$

Let $\varphi\in C^1(\mathbb{R}^{n+1})$ be a non-negative radial function whose support is contained in $B(0,2)$ and that equals $1$ on 
$B(0,1).$ For $r>0$ and $x\in\mathbb{R}^{n+1}$ let us set $\varphi_r(x)\coloneqq  \varphi(x/r).$ Observe that $\|\nabla \varphi\|_\infty\lesssim 1.$ For $x,y\in\mathbb{R}^{n+1}$ we define the regularized kernel
\begin{equation}
\tilde{K}_{r}(x,y)=\bar K(x,y)\varphi_r(x-y)
\end{equation}
and its associated operator
\begin{equation}\label{reg_op}
\tilde{T}_r(f\eta)(x)\coloneqq  \int \tilde{K}_r(x,y)f(y)d\eta(y) \qquad\text{for}\,\, f\in L^\infty_{\mathcal{M}}(\eta),
\end{equation}
where the integral above is absolutely convergent.

We are interested in getting an existence result for the limit
\begin{equation}\label{pv_lim}
\pv \bar T(f\eta)(x)=\lim_{r\to \infty}\tilde{T}_r(f\eta)(x).
\end{equation}
For simplicity, we denote the principal value in \eqref{pv_lim} just as $\bar T(f\eta)(x).$

\begin{lemm} \label{pv_lemm}Let $f\in L^\infty_{\mathcal{M}}.$ The principal value $\bar T(f\eta)(x)$ exists for every $x\in \mathbb{R}^{n+1}.$ Moreover, given any compact set $F\subset\mathbb{R}^{n+1}$, there exist $r_0=r_0(F)>0$ and a constant $c_F$ depending on $F$ such that for $s>r\geq r_0$
\begin{equation}
\big\|\tilde{T}_r(f\eta)-\tilde{T}_s(f\eta)\big\|_{\infty,F}\lesssim \frac{c_F}{r^{{\gamma}}}\|f\|_\infty,
\end{equation}
where $\gamma\in(0,1)$ is as in Lemma \ref{lemm_estim_freezing_periodic}.
\end{lemm}
\begin{rem}Lemma \ref{pv_lemm} implies that the limit in \eqref{pv_lim} converges uniformly on compact sets and in $\supp\eta.$
\end{rem}
\begin{proof}
Recall that we can assume $\ell(Q_0)<1$.
Let $s>r.$ Let us denote $\nu\coloneqq  f\eta$, $\varphi_{r,s}(x)\coloneqq  \varphi_r(x)-\varphi_s(x)$ for every $x\in\mathbb{R}^{n+1}$ and $\tilde K_{r,s}(x,y)\coloneqq\bar K(x,y)\varphi_{r,s}(x-y)$. Because of the periodicity of $f$ and the definition of $\eta,$ we have
\begin{equation}
\nu=\sum_{P\in\mathcal{M}}(T_P)_\sharp(\chi_{Q_0}\nu)
\end{equation}
so that
\begin{equation}\label{diff_sum}
\begin{split}
\tilde{T}_r(f\eta)(x)-\tilde{T}_s(f\eta)(x)&=\int \tilde{K}_{r,s}(x,y)d\bigg(\sum_{P\in\mathcal{M}}(T_P)_\sharp(\chi_{Q_0}\nu)\bigg)(y)\\
&=\sum_{P\in\mathcal{M}}\int_{Q_0}\tilde{K}_{r,s}(x,y+z_p)\,d\nu(y),
\end{split}
\end{equation}
the last equality being a consequence of $\tilde{K}_{r,s}$ having compact support, which implies that the sum has only finitely many non-zero terms.

Let $A_0$ be the homogenized matrix associated with $\{{L}_\epsilon\}_\epsilon>0$ and $\chi_\ell$ be as in Section \ref{section_preliminaries}, with $\ell=6\ell(Q_0)$. Recall that $$\|\nabla \chi_\ell\|_\infty\lesssim 1$$ (see \eqref{bound_grad_chi}). The matrix $A_0$ is an elliptic matrix whose coefficients are constants and can be expressed in terms of  $\chi$ and those of $A$. We denote by $\Theta(\cdot, \cdot;A_0)$ the fundamental solution of the operator ${L}_0=  -\diver(A_0\nabla).$ 
We decompose the right hand side of \eqref{diff_sum} as
\begin{align}
&\sum_{P\in\mathcal{M}}\int_{Q_0}\tilde{K}_{r,s}(x,y+z_p)d\nu(y)\\
 &= \sum_{P\in\mathcal{M}}\int_{Q_0}\big(\bar K(x,y+z_P)-(Id+\nabla\chi_\ell(x))\nabla_1\Theta(x,y+z_P;A_0)\big)\varphi_{r,s}(x-y-z_P)d\nu(y) \\
 &\qquad + \sum_{P\in\mathcal{M}}\int_{Q_0}\big(Id+\nabla\chi_\ell(x)\big)\nabla_1\Theta(x,y+z_P;A_0)\varphi_{r,s}(x-y-z_P)d\nu(y)
  \\
 &\equiv I_{r,s}(x) + II_{r,s}(x)
 .
\end{align}Let us observe that since $F$ is compact and $y\in Q_0,$ there exists a compact set $\tilde{F}$ such that $\pm(x-y)\in\tilde{F},$ so that if we choose $r_0\geq 2\diam\big(\tilde{F}\big),$ both $\varphi_{r,s}(x - y - z_P)$ and $\varphi_{r,s}(x-y +z_P)$ vanish for $|z_P|< r.$ Moreover, $|x-y|\leq \diam\big(\tilde{F}\big)\leq r/2\leq |z_P|$ and
\begin{equation}
|(x-y)-z_P|\approx |(x-y)+z_P|\approx |z_P|.
\end{equation}
Let us now estimate $I_{r,s}(x)$.
As stated in Lemma \ref{lemm_estim_freezing_periodic}, there exist $C>0$ and ${\gamma}\in(0,1)$ depending only on $n$ and $\alpha$ such that
\begin{equation}
\big|\nabla_1\E_{\bar A}(x,y+z_P)-\big(Id+\nabla\chi_\ell(x)\big)\nabla_1\Theta(x,y+z_P;A_0)\big|\leq C \ell(Q_0)^\gamma |x-y-z_P|^{-(n+{\gamma})}
\end{equation}
for every $x,y\in\mathbb{R}^{n+1}.$
Then, exploiting the linear growth of $\eta$ and the considerations on the support of $\varphi_{r,s}$, we get
\begin{equation}\label{estim_I}
|I_{r,s}(x)|\lesssim \sum_{P\in\mathcal{M},|z_P|\geq r}\int_{Q_0}\frac{\ell(Q_0)^{{\gamma}} d|\nu|(y)}{|x-y-z_P|^{n+{\gamma}}}\lesssim\|f\|_{\infty}\sum_{P\in\mathcal{M},|z_P|\geq r}\frac{\ell(P)^{n+{\gamma}}}{|z_P|^{n+{\gamma}}}\lesssim \frac{\|f\|_\infty \ell(Q_0)^{{\gamma}}}{r^{{\gamma}}}.
\end{equation}
In the last inequality of \eqref{estim_I} we used the convergence of $\sum_{P\in\mathcal{M}} {\ell(P)^n}{|z_P|^{-n}}$.

We are left with the estimate of $II_{r,s}(x).$
Using the antisymmetry of $\nabla_1\Theta(\cdot,\cdot;A_0)$ and the properties of standard Calder\'on-Zygmund kernels, the same argument of \cite[Lemma 8.2]{GT} proves that there exists a constant $c_F>0$ such that
\begin{equation}\label{estim_II}
\begin{split}
|II_{r,s}(x)|&\leq \|Id+\nabla\chi_{\ell}\|_{\infty}\sum_{P\in\tilde{\M}}\int_{Q_0}\nabla_1\Theta(x,y+z_P;A_0)\varphi_{r}(x-y-z_P)d\nu(y)\\
&\underset{\eqref{bound_grad_chi}}{\lesssim} \sum_{P\in\tilde{\M}}\int_{Q_0}\nabla_1\Theta(x,y+z_P;A_0)\varphi_{r}(x-y-z_P)d\nu(y) \\
&\lesssim\frac{c_F \|f\|_\infty}{r}.
\end{split}
\end{equation}
We conclude the proof of the lemma gathering \eqref{estim_I}, \eqref{estim_II} and observing that, being $\gamma\in(0,1)$ and $r>1$, $r^{-1}<r^{-\gamma}.$
\end{proof}
The measure $\eta$ is $\M$-periodic and the matrix $\bar A$, by construction, is  $6\ell(Q_0)$-periodic. This implies that for every $f\in L^\infty_{\M}(\eta)$ and $r>0$, the function $\widetilde{T}_r(f\eta)$ is $\M$-periodic, too. The same holds for $\pv T(f\eta).$
Using Lemma \ref{pv_lemm}, the following result is immediate.
\begin{coroll}$\bar T_\eta$ is a bounded operator from $L^\infty_{\mathcal{M}}$ to $L^\infty_{\mathcal{M}}.$ For $r>0$ big enough and for every $f\in L^\infty_{\mathcal{M}}(\eta)$ we have
\begin{equation}
\big\|\bar T(f\eta)-\tilde{T}_r(f\eta)\big\|_{\infty,F}\lesssim_F\frac{\|f\|_\infty}{r^{{\gamma}}}.
\end{equation}
\end{coroll}

Our next intent is to prove the final localization estimate
\begin{equation}\label{full_loc}
\int_{Q_0}\big|\bar T\eta\big|^2d\eta\ll\eta(Q_0).
\end{equation}
We have already proved that for $M$ big enough there exists $\epsilon'\ll 1$ such that 
\begin{equation}\label{loc_inside}
\int_{Q_0}\big|\bar T(\chi_{MQ_0}\eta)\big|^2d\eta\lesssim \epsilon'\eta(Q_0).
\end{equation}
Then, in order to prove \eqref{full_loc}, it suffices to use the estimate in the following lemma.
\begin{lemm}\label{lem102}Let $f\in L^1_{loc}(\eta)$ be a $\mathcal{M}$-periodic function and let $\tilde{M}=6\tilde{N}$, where $\tilde{N}\geq 3$ is an odd number. For all $x\in 2Q_0$ we have
\begin{equation}\label{loc_outside}
\big|\bar T\big(\chi_{(\tilde{M}Q_0)^c}f\eta\big)(x)\big|\lesssim\frac{1}{\tilde{M}^\gamma \ell(Q_0)^n}\int_{Q_0}|f|d\eta.
\end{equation}
\end{lemm}
\begin{proof}
Being $\widetilde{N}$ odd, there exists a subfamily $\widetilde{\M}\subset\M$ such that
\begin{equation}
\chi_{(\widetilde{M}Q_0)^c}\eta =\sum_{P\in\widetilde{\M}}T_P\sharp\eta
\end{equation}
and whose elements $P\in\widetilde{\M}$ satisfy $|z_P|\gtrsim \widetilde{M}\ell(Q_0).$ In particular
\begin{equation}\label{approz_z_p}
|x-y-z_P|\approx |z_p|\quad \text{for  $x,y\in 2Q_0.$}
\end{equation}
Let $r>0$ and $x\in 2Q_0$. Denote $\nu\coloneqq  f\eta$  and observe that there are just finitely many cubes $P\in \widetilde{\M}$ such that $|z_P|<r$. Arguing as in the proof of Lemma \ref{pv_lemm}, this justifies the following writings.
\begin{equation}
\begin{split}
&\tilde{T}_r\big(\chi_{(\tilde{M}Q_0)^c}f\eta\big)(x)=\int \bar K(x,y)\varphi_r(x-y)d\nu(y)\\
&\qquad=\sum_{P\in \widetilde{M}}\int_{Q_0}\bar K(x,y+z_P)\varphi_r({x-y-z_P})d\nu(y)\\
&\qquad= \sum_{P\in\tilde{\M}}\int_{Q_0}\big(\bar K(x,y+z_P)-(Id+\nabla\chi_\ell(x))\nabla_1\Theta(x,y+z_P;A_0)\big)\varphi_{r}(x-y-z_P)d\nu(y) \\
 &\qquad\quad+ \sum_{P\in\tilde{\M}}\int_{Q_0}\big(Id+\nabla\chi_\ell(x)\big)\nabla_1\Theta(x,y+z_P;A_0)\varphi_{r}(x-y-z_P)d\nu(y)
  \\
 &\qquad\equiv I_{r}(x) + II_{r}(x)
\end{split}
\end{equation}
Let us estimate $I_r(x).$ Using \eqref{approz_z_p} together with Lemma \ref{lemm_estim_freezing_periodic} and the estimate $|z_P|\gtrsim \tilde{M}\ell(Q_0)$ for $P\in\tilde{\M}$, we can write
\begin{equation}\label{estim_I_r}
\begin{split}
|I_r(x)|&\lesssim \sum_{P\in\tilde{\M}}\int_{Q_0}\frac{\ell(Q_0)^{{\gamma}}}{|x-y-z_P|^{n+{\gamma}}}d\nu(y)\approx
\sum_{P\in\tilde{\M}}\int_{Q_0}\frac{\ell(Q_0)^{{\gamma}}}{|z_P|^{n+{\gamma}}}d\nu(y)\\
&=\sum_{P\in\tilde{\M}}\frac{\ell(Q_0)^{{\gamma}}}{|z_P|^{n+{\gamma}}}|\nu|(Q_0)\lesssim \frac{|\nu|(Q_0)}{\tilde{M}^{{\gamma}} \ell(Q_0)^n} \bigg(\sum_{P\in\tilde{\M}}\frac{\ell(Q_0)^n}{|z_P|^n}\bigg)\lesssim \frac{1}{\tilde{M}^{{\gamma}}\ell(Q_0)^n}\int_{Q_0}|f|d\eta.
\end{split}
\end{equation}
We claim that
\begin{equation}\label{estim_II_r}
|II_r(x)|\lesssim \frac{1}{\tilde{M}\ell(Q_0)^n}\int_{Q_0}|f|d\eta.
\end{equation}
The calculations to prove \eqref{estim_II_r} exploit \eqref{bound_grad_chi} and the antisimmetry of $\n1\Theta(\cdot,\cdot;A_0)$; they resemble those of \cite[Lemma 8.4]{GT}, so that we leave the verification to the reader.\\
The estimates \eqref{estim_I_r} and \eqref{estim_II_r}, together with the observation that $\tilde{M}^{-1}\leq\tilde{M}^{-{\gamma}},$ conclude the proof of the lemma after taking the limit for $r\to \infty$.
\end{proof}
\begin{coroll}[Final localization estimate]We have
\begin{equation}
\int_{Q_0}\big|\bar T\eta\big|^2d\eta \lesssim \Big(\frac{1}{M^{2{\gamma}}}+\epsilon'\Big)\eta(Q_0).
\end{equation}
\end{coroll}
\begin{proof}
Inequality \eqref{loc_outside} in the case $f\equiv 1$ reads
\begin{equation}
\big|\bar T(\chi_{MQ_0}\eta)\big|\lesssim \frac{1}{M^{{\gamma}}},
\end{equation}
so that applying it together with \eqref{loc_inside}, we have
\begin{equation}
\int_{Q_0}\big|\bar T\eta\big|^2d\eta\lesssim \int_{Q_0}\big|\bar T(\chi_{MQ_0}\eta)\big|^2d\eta + \int_{Q_0}\big|\bar T(\chi_{(MQ_0)^c}\eta)\big|^2d\eta\lesssim\Big(\frac{1}{M^{2{\gamma}}}+\epsilon'\Big)\eta(Q_0),
\end{equation}
which finishes the proof.
\end{proof}

\section{A pointwise inequality and the conclusion of the proof}\label{section_var_arg}
The following lemma implements a variational technique inspired by potential theory that allows to obtain a pointwise inequality for the potential of a proper auxiliary measure.  We denote as $\bar{T}^*\vec{\xi}$ the operator that, given a vector-valued measure $\vec{\xi}$, is defined by
\begin{equation}
\bar T^*\vec \xi(x)=\int\nabla_1 \bar{\E}(y,x)\cdot d\vec{\xi}(y)
\end{equation}
and which corresponds to the adjoint of $\bar T$.
\begin{lemm}\label{var_lemma}
Suppose that for some $0<\lambda\leq 1$ the inequality
\begin{equation}
\int_{Q_0}|\bar T\eta|^2d\eta\leq\lambda\eta(Q_0)
\end{equation}
holds. Then there is a function $b\in L^\infty(\eta)$ such that
\begin{itemize}
\item $0\leq b\leq 2.$
\item $b$ is $\mathcal{M}$-periodic.
\item $\int_{Q_0}b \,d\eta=\eta(Q_0).$
\end{itemize}
and such that the measure $\nu=b\eta$ satisfies
\begin{equation}\label{loc_measure_variational_argument}
\int_{Q_0}|\bar T\nu|^2d\nu\leq \lambda\nu(Q_0)
\end{equation}
and
\begin{equation}\label{pw_ineq_ae}
|\bar T\nu|^2(x)+2 \bar T^*\big((\bar T\nu)\nu\big)(x)\leq 6\lambda \text{ for }\nu\text{-a.e. } x\in\mathbb{R}^{n+1}.
\end{equation}
\end{lemm}
\begin{proof}
The proof is a minor variation of the proof of \cite[Lemma 9.1]{GT}. In particular, we recall that the way to prove \eqref{pw_ineq_ae} consists in defining an adapted energy functional
\begin{equation}
J(a)=\lambda\|a\|_{L^\infty(\eta)}\eta(Q_0)+\int_{Q_0}|\bar T(a\eta)|^2 d\eta,
\end{equation}
where $a$ ranges in
\begin{equation}
\mathcal A =\Big\{a\in L^\infty(\eta): a\geq 0 , a \text{ is }\mathcal M\text{-periodic, and }\int_{Q_0}a\,d\eta=\eta(Q_0)\Big\}.
\end{equation}
Then, one proves that $J$ admits a minimizer in $\mathcal{A}$ and gets \eqref{pw_ineq_ae} by taking proper competitors. The proof does not use the antisymmetry of the kernel of $\bar T$ but just its $\mathcal M$-periodicity which follows by the construction of $\bar A$.
\end{proof}
\subsection{A maximum principle}
Let $\lambda, b$ and $\nu$ be as in Lemma \ref{var_lemma}. In order to be able to perform the final argument to get the contradiction, we need to extend the inequality \eqref{pw_ineq_ae} out of the support of $\nu$.
More precisely, the next step consists in proving that a inequality similar to that provided by Lemma \ref{var_lemma} holds in a suitable strip. To this purpose, some version of maximum principle is needed. The elliptic setting of the problem makes this procedure slightly more technical then the one adopted by Girela-Sarri\'on and Tolsa in the case of the Riesz transform.

Before presenting the main result of the section, we introduce some notation.  We denote by $\tilde{H}$ the hyperplane
\begin{equation}
\tilde{H}\coloneqq  \{x\in\Rn1: x_{n+1}= 3\ell(Q_0)/2\},
\end{equation}
which corresponds to the translate of $H$ that contains the upper face of $3Q_0.$
Let $K_S\gg 1$ to be chosen later and let $S$ denote the strip
\begin{equation}
S\coloneqq  \{x\in\Rn1: \dist(x,\tilde{H})< K_S \ell(Q_0)\}.
\end{equation}
Its boundary $\partial S$ is given by the union of two hyperplanes $\partial S_+$ and $\partial S_-$ which lay in the upper and lower half spaces respectively. Let
\begin{equation}\label{x_s_pm}
x_{S\pm}{=\frac{3}{2}\ell(Q_0)(1,\ldots,1,1)\pm \big(0,\ldots,0,K_S\ell(Q_0)\big).}
\end{equation}
For the proof of our next lemma we need to invoke a result on elliptic measure. Suppose that $\Omega \subsetneq \Rn1$ is an open set with $n$-AD-regular boundary and consider a point $p\in\Omega$. Let $\omega^p_\Omega$ denote the elliptic measure on $\partial\Omega$ associated with the operator $L_{\bar A}$ with pole at $p$.
{For the proof of the following standard result we refer to \cite[Lemma 2.3]{AM}.}
\begin{lemm}Let $\Omega\subsetneq \Rn1$ be open with $n$-AD-regular boundary with constant $C_{AD}$. There exists $\vartheta=\vartheta(n,A,C_{AD})\in (0,1)$ such that for every $x\in\partial\Omega$ and $0<r<\diam\Omega$, we have
\begin{equation}\label{growth_elliptic_measure}
\omega_{\Omega}^y\big(B(x,r)^c\big)\leq C\Big(\frac{|x-y|}{r}\Big)^{\vartheta}\qquad \text{ for }y\in\Omega\cap B(x,r).
\end{equation}
\end{lemm}
An application of \eqref{growth_elliptic_measure} gives a boundary regularity result for $L_{\bar{A}}$-harmonic functions, see e.g Lemma 2.10 in \cite{AGMT}.
\begin{lemm}\label{lemma_bd_estim}
Let $\Omega\subsetneq \Rn1$ be open with $n$-AD-regular boundary with constant $C_{AD}$.
Let $u\geq 0$ be $L_{\bar A}$-harmonic function in $B(x,4r)\cap \Omega$ and
continuous in $B(x,4r)\cap \bar \Omega$. Suppose, moreover, that $u\equiv 0$ in $\partial \Omega \cap B(x,4r)$. Then, extending $u$ by zero in $B(x,4r)\setminus \bar \Omega,$ there exists $\vartheta=\vartheta(n,A,C_{AD})\in (0,1)$ such that $u$ is $\vartheta$-H\"older continuous in $B(x,4r)$ and, in particular,
\begin{equation}
u(y)\lesssim_{n,A,C_{AD}} \Big(\frac{\dist(y,\partial\Omega)}{r}\Big)^\vartheta \sup_{B(x,2r)}u \qquad \text{ for all }y\in B(x,r).
\end{equation}
\end{lemm}

\begin{lemm}[Maximum principle on the strip]\label{max_princ_strip} Let $S$ be the strip as before and let $f$ be a bounded continuous $L_{\bar A}$-harmonic function on $S$ so that $f|_{\partial S}\equiv 0$. Then $f\equiv 0$ on $S$.
\end{lemm}
\begin{proof}
Choose $R> 100 K_S$ and set $S_R\coloneqq  S\cap [-R,R]^{n+1}$. For $p\in S,$ denote $h_p\coloneqq  \dist(p,\partial S)$ and let $x_p$ be a point that realizes the distance. We choose $p$ far from the ``vertical" parts $\partial S_R\setminus (\partial S_+\cup\partial S_-)$ of $\partial S_R$, in particular such that $B(x_p, R/10)\cap (\partial S_R\setminus \partial S)=\emptyset$. Let $\omega^p_R$ denote the elliptic measure with pole at $p$ associated with $L_{\bar A}$ on $S_R$. The family $\{S_R\}_R$ is a collection of AD-regular sets whose AD-regularity constants do not depend on $R$. Then inequality \eqref{growth_elliptic_measure} implies that there exit two constants $C$ and $\vartheta$, both independent on $R$, such that
\begin{equation}
\omega^p_R(\partial S_R\setminus \partial S)\leq \omega^p_R\big(B(x_p,R/10)^c\big)\leq C\Big(\frac{h_p}{R}\Big)^{\vartheta}.
\end{equation}
By hypothesis we may assume that $f\leq 1$ on ${\partial S_R\setminus \partial S}.$ Thus, we have
\begin{equation}\label{end_lemm:elliptic_meas}
|f(p)|=\Big|\int f d\omega^p_R\Big|\leq \norm{f|_{\partial S_R\setminus \partial S}}_\infty\omega^p_R(\partial S_R\setminus \partial S)\leq C\Big(\frac{h_P}{R}\Big)^{\vartheta}.
\end{equation}
The results stated in the lemma follows by passing to the limit in \eqref{end_lemm:elliptic_meas} for $R\to \infty$.
\end{proof}
Now, we prove an existence result on the infinite strip $S$.
\begin{lemm}\label{lemm_aux_function_f}There exists a function $f_S\colon \bar{S}\to\mathbb{R}$ such that:
\begin{enumerate}
\item $f_S$ is $L_{{\bar A}}$-harmonic in the strip $S$ and continuous in  $\bar S$.
\item $f_S$ is $\M$-periodic.
\item $f_S(x)= \pm 1$ on $\partial S_{\pm}$ and $f_S(x)=0$ for $x\in\tilde{H}$.
\end{enumerate}
\begin{proof}
Let $k\in \mathbb{N},$ $k\geq 100 K_S$ and denote $S_k= S\cap [-k,k]^{n+1}$. We define the continuous functions $f_k$ on $\partial S_k$ as
\begin{equation}
f_k(x)=\frac{x_{n+1}-\frac{3}{2}\ell(Q_0)}{K_S\ell(Q_0)}.
\end{equation}
In particular, observe that $f_k(x)=\pm 1$ for $x\in \partial S_\pm$ and
\begin{equation}
f_k(x)=-f_k\Big(\big(x_1,\ldots,x_n,-x_{n+1}+ {3}\ell(Q_0)\big)\Big),
\end{equation}
i.e. it is antisymmetric with respect to $\tilde{H}$.

Define $u_k$ be the $L_{\bar{A}}$-harmonic function such that $u_k|_{\partial S_k}=f_k$, whose existence is guaranteed by the continuity of $f_k$ and the AD-regularity of $S_k$. Our aim is to prove that, a part from possibly considering a proper subsequence, $u_j$ converges uniformly in the compact subsets of $S_k$, for every $k$ to an $L_{\bar A}$-harmonic function in $S$.

We claim that there exist $\gamma\in(0,1)$ and $C_k>0$ such that
\begin{equation}\label{equi_holder}
|u_j(x)-u_j(y)|\leq C_k |x-y|^\gamma \qquad\text{for}\qquad x,y\in \bar S_k, \quad j\geq k+2.
\end{equation}
Assume that \eqref{equi_holder} holds. As a consequence of Ascoli-Arzel\`a's theorem together with a standard diagonalization argument, there is a function $f_S$ so that $u_k$ converges to $f_S$ uniformly on the compact subsets of $S$. The $L_{\bar{A}}$-harmonicity of $f_S$ is a consequence of Caccioppoli's esimate (cfr. \cite[Theorem 3.77]{HKM}).

To prove (2), define $\vec{v}=(6\ell(Q_0), 0,\ldots,0)$ and observe that, being the matrix $\bar{A}$ $\mathcal{M}$-periodic and since $f_S$ is constant on $\partial S_\pm$, the function $ f(x)=f_S(x)-f_S(x+\vec{v})$ satisfies the hypotesis of Lemma \ref{max_princ_strip}. So, $f\equiv 0$ and $f_S$ is $\mathcal{M}$-periodic.

To prove (3), first observe that $A(x)=A_\phi\Big(\big(x_1,\ldots,x_n,-x_{n+1}+ {3}\ell(Q_0)\big)\Big)$, where $\phi$ is the function that maps a point to its reflected with respect to $\tilde{H}$ and $A_\phi$ is defined as in \eqref{ch_var:A}.
Then we can apply again Lemma \ref{max_princ_strip} to
\begin{equation}
\tilde{f}(x)=f_S(x)+f_S\Big(\big(x_1,\ldots,x_n,-x_{n+1}+ {3}\ell(Q_0)\big)\Big),
\end{equation}
which is $L_{\bar A}$-harmonic and vanishes on $\partial S$.

We are left with the proof of the claim \eqref{equi_holder}. By Lemma \ref{lemma_bd_estim}, there exists $\vartheta \in (0,1)$ depending only on $n$, $\bar A$ and the AD-regularity of $\partial \Omega$ (hence independent both on $j$ and $k$) such that $u_j$ is $\vartheta$-H\"older continuous in the set $\{x\in\bar \Omega: \dist(x,\partial\Omega)\leq 2\ell(Q_0)\}$.
Being $\|u_j\|_\infty \leq 2$ for every $j$, by De Giorgi-Nash interior estimates we can infer that there exists $\gamma_k$ independent on $j$ such that, for every $j\geq k+2$, $u_j$ is $\gamma_k$-H\"older continuous in $\{x\in \Omega_{k+1}:\dist(x,\partial\Omega)> \ell(Q_0)\}.$ Gathering the interior and the boundary regularity of $u_j$ proves \eqref{equi_holder}.
\end{proof}
\end{lemm}
By the previous lemma, Lemma \ref{lemma_bd_estim} and the fact that $f_S\equiv 0$ on $\tilde{H}$, we have the estimate
\begin{equation}
|f_S(y)|\lesssim \Big(\frac{\dist(y,\tilde{H})}{K_S \ell(Q_0)}\Big)^{\vartheta}, \qquad\text{ for }\quad y\in S \text{ with }
\dist(y,\tilde H)\leq 10\ell(Q_0).
\end{equation}
Let us define the auxiliary function
\begin{equation}
F_S(x)\coloneqq  f_S(x)\bar T\nu(x_{S+}).
\end{equation}
Observe that $F_S|_{\partial S_\pm}\equiv \pm\bar{T}\nu(x_{S+})$.
The rest of the present section is devoted to the proof of the following, which may be regarded as an approximated maximum principle on $S$.
\begin{lemm}[Pointwise bound for the potential on the strip]\label{main_lemma_maximum_principle} For $x\in S$ we have
\begin{equation}
|\bar T\nu(x)-F_S(x)|^2 + 4 \bar T^*\big((\bar T\nu)\nu\big)(x)\lesssim \lambda^{1/2} + \frac{1}{K_S^{2\tilde\alpha}} + \frac{1}{K_S^{\vartheta}} +(C_S \ell(Q_0))^{\tilde\alpha} + \Big(\frac{K_S}{C_S}\Big)^{\tilde\alpha},
\end{equation}
where $C_S$ is a constant chosen so that $C_S \gg K_S$.
\end{lemm}
Before proving this lemma, we need some auxiliary result.
\begin{lemm}\label{lemm_comparison_plus_minus}Let $x_{S+}$ and $x_{S-}$ as in \eqref{x_s_pm}. Then:
\begin{enumerate}
\item \label{lpot1} For $x\in\partial S_+$, $\dist(x,x_{S+})\lesssim \ell(Q_0)$ we have the estimate
\begin{equation}\label{oscillation_strip}
|\bar T\nu(x)-\bar T\nu(x_{S+})|\lesssim \frac{1}{K_S^{\tilde\alpha}}.
\end{equation}
The analogous estimate holds in $x\in\partial S_-$ replacing $x_{S+}$ with $x_{S-}$.
\item \label{lpot2} The difference of $-\bar T\nu(x_{S+})$ and $\bar T\nu(x_{S-})$ can be estimated as
\begin{equation}
|\bar T\nu(x_+)+\bar T\nu(x_{S-})|\lesssim \frac{1}{K_S^{\tilde\alpha}}.
\end{equation}
\item  For $x$ with $\dist(x,\tilde H)\geq 2\ell(Q_0)$ we have
\begin{equation}\label{esti_T_star}
\bar T^*\big((\bar T\nu)\nu\big)(x)\lesssim \lambda^{1/2}.
\end{equation}
\end{enumerate}
\end{lemm}
\begin{proof}Let us begin with the proof of (\ref{lpot1}). Because of the $\mathcal{M}$-periodicity of $\bar T\nu$, we can assume without loss of generality that $x_H\in [-3\ell(Q_0),3\ell(Q_0)]^n\times\{0\}$, $x_H$ denoting the projection of $x$ on $H$.
We claim that for $P\in\M$ and $y\in Q_0$ we have
\begin{equation}
|\bar K(x,y+z_P)-\bar K (x_{S+},y+z_P)|\lesssim \frac{\ell(Q_0)^{\tilde\alpha}}{(K_S \ell(Q_0))^{n+\tilde\alpha}+|z_P|^{n+\tilde\alpha}}.
\end{equation}
This follows from the (global) Calder\'on-Zygmund estimates for $\bar K(\cdot,\cdot)$ once we observe that $|x-x_{S+}|\lesssim |x-y-z_P|\approx K_S \ell(Q_0) + |z_P|.$ So, for $r>0$, standard calculations give
\begin{equation}
\begin{split}
\big|\tilde T_r\nu(x)-\tilde T_r\nu(x_{S+})\big|&\lesssim \sum_{P\in\M}\int_{Q_0}\frac{\ell(Q_0)^{\tilde\alpha}}{(K_S \ell(Q_0))^{n+\tilde\alpha}+|z_P|^{n+\tilde\alpha}}d\nu(y)\\
&=\sum_{P\in\M}\frac{\ell(Q_0)^{n+\tilde\alpha}}{(K_S \ell(Q_0))^{n+\tilde\alpha}+|z_P|^{n+\tilde\alpha}}\lesssim \frac{\ell(Q_0)^{n+\tilde\alpha}}{(K_S \ell(Q_0))^{n+\tilde\alpha}}=\frac{1}{K_S^{\tilde\alpha}}.
\end{split}
\end{equation}
Being this estimate independent on the choice of $r$, in the limit for $r\to \infty$ we have \eqref{oscillation_strip}. The proof of the analogous estimate for $x_{S-}$ is identical, so we omit it and we go to the proof of (\ref{lpot2}).

Denote by $x^*$ the reflection of the point $x$ across $x_0=\frac{3}{2}\ell(Q_0)(1,\ldots,1)$, i.e.
\begin{equation}
x^*=2x_0-x.
\end{equation}
By the specific choice of $x_0$, this transformation can be obtained via a composition of the reflections $\psi_j$'s with respect to the hyperplanes passing through $x_0$ which we defined in \eqref{psijdef}:
\begin{equation}
x^*=\psi_1 \circ \cdots\circ \psi_{n+1}(x).
\end{equation}
Moreover,
\begin{equation}\label{lem1171}
(x_{S+})^*=3\ell(Q_0)(1,\ldots,1)-\frac{3}{2}\ell(Q_0)(1,\ldots,1)-\big(0,\ldots,0,K_S\ell(Q_0)\big)=x_{S_-}.
\end{equation}
Thus, an immediate application of Lemma \ref{lemma_change_var_fund_sol} and \eqref{wdabar} gives that, for $y\in Q_0$,
\begin{equation}\label{lem1172}
\bar K (x_{S_+},y+z_P)=-\bar K(x_{S_-},y^*+z_P^*), \qquad \qquad P\in\mathcal M.
\end{equation}
Observe that 
\begin{equation}
|y+z_P-(y^*-z_P^*)|\leq|y-y^*| + |z_P-(-z^*_{P})|\lesssim \ell (Q_0)
\end{equation}
which, combined with Lemma \ref{CZ_l}, gives
\begin{equation}\label{eq_lem_refl_a}
\begin{split}
&\big|\bar K (x_{S_+},y+z_P)+\bar K(x_{S_-},y-z_P)\big|\\
&\qquad \overset{\eqref{lem1172}}{=}|\bar K(x_{S_+},y+z_P) -\bar K(x_{S_-}^*,y^*-z_P^*)|\\
&\qquad {=}|\bar K (x_{S_+},y+z_P)-\bar K (x_{S_+},y^*-z_P^*)|
\lesssim \frac{\ell(Q_0)^{\tilde\alpha}}{(K_S\ell(Q_0))^{n+\tilde\alpha}+|z_P|^{n+\tilde\alpha}},
\end{split}
\end{equation}
where the second equality uses \eqref{lem1171} (with $-z_P=z_{-P}$ replacing $z_P$).
Taking $r>0$ and applying \eqref{eq_lem_refl_a}, we have
\begin{equation}
\big|\tilde{T}_r\nu(x_{S^+})+\tilde{T}_r\nu(x_{S^-})\big|\lesssim \sum_{P\in\mathcal{M}}\frac{\ell(Q_0)^{n+\tilde\alpha}}{(K_S\ell(Q_0))^{n+\tilde\alpha}+|z_P|^{n+\tilde\alpha}}\lesssim \frac{1}{K_S^{\tilde\alpha}}
\end{equation}
which, taking the limit for $r\to \infty$, proves \eqref{lpot2}.

We are left with the proof of (3). Set $\sigma=(\bar T \nu)\nu$ and observe that this measure is $\mathcal{M}$-periodic. So, without loss of generality, we can assume that $x_H\in [-3\ell(Q_0),3\ell(Q_0)]^n\times \{0\}$. Let $r>0$ and, denoting by $A_0$ the homogenized matrix associated with $\bar A$, by $\chi$ the vector of correctors and $\ell=6\ell(Q_0)$, write
\begin{equation}
\begin{split}
\tilde T_r\sigma(x)&=\sum_{P\in \mathcal{M}}\int_{Q_0}\tilde{K}_r(y+z_P,x)d\sigma(y)\\
&=\sum_{P\in \mathcal{M}}\int_{Q_0}\big(\bar K(y+z_P,x)-(Id+\nabla\chi_\ell(y+z_P))\nabla_1\Theta(y+z_P,x;A_0)\big)\varphi_r(x-y-z_P)d\sigma(y)\\
&\quad+\sum_{P\in \mathcal M}\int_{Q_0}(Id+\nabla\chi_\ell(y+z_P))\nabla_1\Theta(y+z_P,x;A_0)\varphi_r(x-y-z_P)d\sigma(y)\\
&\equiv I_r + II_r.
\end{split}
\end{equation}
Recalling that $\|\nabla\chi_\ell\|_\infty\lesssim 1$ and using Lemma \ref{lemm_estim_freezing_periodic}, we can proceed with the following estimates
\begin{equation}\label{part_I_r}
\begin{split}
|I_r|&\lesssim \sum_{P\in \mathcal{M}}\int_{Q_0}\frac{\ell(Q_0)^\gamma}{|x-y-z_P|^{n+\gamma}}d|\sigma|(y)\lesssim \sum_{P\in \mathcal{M}} \int_{Q_0}\frac{\ell(Q_0)^\gamma}{(\dist(x,\tilde{H})+|z_P|)^{n+\gamma}}d|\sigma|(y)\\
&\lesssim \frac{\ell(Q_0)^{n+\gamma}}{(\dist(x,\tilde H)+|z_P|)^{n+\gamma}}\frac{|\sigma|(Q_0)}{\ell(Q_0)^{n}}
\lesssim \frac{\ell(Q_0)^{n+\gamma}}{\dist(x,\tilde{H})^\gamma}\frac{|\sigma|(Q_0)}{\ell(Q_0)^n}\lesssim \frac{|\sigma|(Q_0)}{\ell(Q_0)^n},
\end{split}
\end{equation}
where the last inequality holds because we assumed $\dist(x,\tilde{H})\geq 2\ell(Q_0)$. We claim that
\begin{equation}\label{part_II_r}
|II_r|\lesssim \frac{|\sigma|(Q_0)}{\ell(Q_0)}.
\end{equation}
It is possible to prove this estimate analogously to the case of the Riesz transform. We omit its proof in order not to make the presentation too lengthy. We remark that the calculations that lead to  \eqref{part_II_r} solely relies  on the Calder\'on-Zygmung property of the kernel and some geometric considerations that are independent on its specific expression. We refer to \cite[(8.20)]{GT} for more details.

Hence, gathering \eqref{part_I_r}, \eqref{part_II_r} and passing to the limit on $r$, we get
\begin{equation}
\bar T^*\big((\bar T\nu)\nu\big)(x)\lesssim \frac{1}{\ell(Q_0)^n}\int_{Q_0} |\bar T\nu|d\nu.
\end{equation}
Then, recalling \eqref{loc_measure_variational_argument}, the growth of $\nu$ and using Cauchy-Schwarz's inequality,
\begin{equation}
\bar T^*\big((\bar T\nu)\nu\big)(x)\lesssim \frac{1}{\ell(Q_0)^n}\Big(\int_{Q_0}|\bar T\nu|^2d\nu\Big)^{1/2}\nu(Q_0)\lesssim \lambda^{1/2},
\end{equation}
which finishes the proof of (3).
\end{proof}
The following result is a direct consequence of Lemma \ref{lemm_comparison_plus_minus}.
\begin{coroll} For $x\in\partial S$ 
\begin{equation}\label{estim_stric}
|\bar T\nu(x)-F_S(x)|^2\lesssim \frac{1}{K_S^{2\tilde\alpha}},
\end{equation}
where the implicit constant does not depend on $S$.
\end{coroll}
Another result which is needed for the application of the maximum principle is the estimate of $|F_S(x)|$ for $x$ close to the support of the measure $\nu$.
\begin{lemm} \label{lem_estim_FS_close}{ Let $K_S$ be an odd natural number, $K_S\geq 3$.} For $x\in \Rn1$ with {$\dist(x,\tilde H)\leq 10\ell(Q_0)$} we have
\begin{equation}
|F_S(x)|\lesssim \frac{1}{K_S^{\vartheta}}.
\end{equation}
\end{lemm}
\begin{proof}
Because of the H\"older continuity of $f_S$, we can write
\begin{equation}
|F_S(x)|\lesssim \Big(\frac{\dist(x,\tilde{H})}{K_S\ell(Q_0)}\Big)^\vartheta|\bar T\nu(x_{S+})|\lesssim \frac{|\bar T\nu(x_{S+})|}{K_S^\vartheta}.
\end{equation}
So, to prove the lemma, it suffices to show that
\begin{equation}
|\bar T\nu(x_{S+})|\leq C
\end{equation}
for some constant $C>0$ not depending on $K_S$. Recall now that $\nu=b\, \eta$.
Applying Lemma \ref{lem102} with $\tilde M=6 K_S$ and $f=b$ to the point $0\in 2Q_0$, we have the estimate
\begin{equation}\label{eq1181}
\big|\bar T \big(\chi_{(6K_SQ_0)^c}\nu\big)(0)\big|\lesssim \frac{1}{(6K_S)^\gamma \ell(Q_0)^n}\int_{Q_0}|b|\,d\eta \lesssim 1,
\end{equation}
where the implicit constant in the last inequality does not depend on $K_S$.
Now, we observe that the (global) Calder\'on-Zygmund properties of $\bar K$ and the fact that $|x_{S+}|\lesssim K_S\ell(Q_0)$ imply
\begin{equation}\label{eq1182}
\begin{split}
\big|\bar T \big(\chi_{(6K_SQ_0)^c} \nu\big)(0)-\bar T\big(\chi_{(6K_SQ_0)^c}\nu\big)(x_{S+})\big|& \lesssim \int_{(6K_S Q_0)^c}\big|\bar K (0,y)-\bar K(x_{S+},y)\big|d\nu(y)\\
&\lesssim \int_{(6K_SQ_0)^c}\frac{|x_{S+}|^{\tilde\alpha}}{(|y|+|x_{S+}|)^{n+\tilde{\alpha}}}d\nu(y)\lesssim 1.
\end{split}
\end{equation}
Then, by \eqref{eq1181}, \eqref{eq1182} and the triangle inequality, we have
\begin{equation}\label{eq1183}
\begin{split}
&\big|\bar T\big(\chi_{(6K_S Q_0)^c}\nu(x_{S+})\big)\big|\\
&\qquad\leq \big|\bar T \big(\chi_{(6K_SQ_0)^c}\nu\big)(0)\big| + \big|\bar T \big(\chi_{(6K_SQ_0)^c} \nu\big)(0)-\bar T\big(\chi_{(6K_SQ_0)^c}\nu\big)(x_{S+})\big|\lesssim 1.
\end{split}
\end{equation}
Moreover, since $\dist(x_{S+},\supp \nu)^n\gtrsim K_S\ell(Q_0)$ and estimating the kernel via Lemma \ref{CZ_l},
\begin{equation}\label{eq1184}
\begin{split}
\big|\bar T\big(\chi_{6K_SQ_0}\nu \big)(x_{S+})\big|&\leq \int_{6K_S Q_0} |\bar K(x_{S+}, y)|d\nu(y)\overset{}{\lesssim} \int_{6 K_S Q_0}\frac{1}{|x_{S+}-y|^n}d\nu(y)\\
&\lesssim \frac{\nu(6K_S Q_0)}{\dist(x_{S+},\supp \nu)^n}\lesssim \frac{K_S^n \ell(Q_0)^n}{\dist(x_{S+},\supp \nu)^n} \lesssim 1.
\end{split}
\end{equation}
Thus, gathering \eqref{eq1183} and \eqref{eq1184} we obtain 
\begin{equation}
|\bar T\nu(x_{S+})|\leq \big|\bar T\big(\chi_{6K_SQ_0}\nu \big)(x_{S+})\big| + \big|\bar T\big(\chi_{(6K_S Q_0)^c}\nu(x_{S+})\big)\big| \lesssim 1,
\end{equation}
which proves the lemma.
\end{proof}

In order to be able to use the previous lemma, from now on we assume without loss of generality $K_S\geq 3$ and we suppose it to be an odd number.
Observe that for $x\in\supp \nu$, Lemma \ref{lem_estim_FS_close} and \eqref{pw_ineq_ae} give
\begin{equation}\label{sup_on_the_measure}
\begin{split}
&\sup_{x\in\supp\nu}\big|\bar T\nu(x)-F_S(x)\big|^2+4\bar{T}^*((\bar T\nu)\nu)(x)\\
&\qquad\leq \sup_{x\in\supp\nu}2\big|\bar T\nu(x)|^2+4\bar{T}^*((\bar T\nu)\nu)(x) + 2|F_S(x)|^2\\
&\qquad\leq 12 \lambda + 2|F_S(x)|^2\lesssim \lambda +\frac{1}{K_S^{\vartheta}}.
\end{split}
\end{equation}
Moreover, by \eqref{esti_T_star} and \eqref{estim_stric},
\begin{equation}
\begin{split}
&\sup_{x\in\partial S}\big|\bar T\nu(x)-F_S(x)\big|^2+4\bar{T}^*((\bar T\nu)\nu)(x)\lesssim \frac{1}{K_S^{2\tilde\alpha}}+\lambda^{1/2}
\end{split}
\end{equation}
which, together with \eqref{sup_on_the_measure} brings us to
\begin{equation}\label{sup_on_meas_and_support}
\sup_{x\in \partial S\cup\supp\nu} \big|\bar T\nu(x)-F_S(x)\big|^2+4\bar{T}^*((\bar T\nu)\nu)(x)\lesssim \lambda^{1/2}+\frac{1}{K_S^{2\tilde\alpha}}+\frac{1}{K_S^{\vartheta}}.
\end{equation}
Finally, we provide the proof of Lemma \ref{main_lemma_maximum_principle}.
\begin{proof}
We recall that $\bar A=\bar A^T$.
Let $\vec{g}\in L^\infty(S;\Rn1).$
We claim that $\bar T^*\big(\vec g \mathcal L^{n+1}\big)$ is a $L_{\bar A^T}$-harmonic (vector valued) function. 
This would imply the maximum principle
\begin{equation}\label{max_prin_t_star}
\sup_{x\in S}\bar T^*\big(\vec g \mathcal L^{n+1}\big)(x)=\sup_{x\in\partial S\cap\supp\nu}\bar T^*\big(\vec g \mathcal L^{n+1}\big)(x).
\end{equation}
Observe that, because of Lemma \ref{lemm_aux_function_f}, the same 	equality holds with $F_S(x)$ in place of $\bar T^*(\vec g\mathcal L^{n+1})(x)$.
Let $\varphi\in C^\infty_c(S\setminus \supp \vec{g})$ be a test function.
To prove the claim, apply the definition of $\bar T^*$ together with Fubini's theorem together with the fact that $\bar \E(x,y)=\E_{A^T}(y,x)$:
\begin{equation}
\begin{split}
\int \bar A^T\nabla \bar{T}^*(\vec{g}\mathcal{L}^{n+1})\cdot \nabla \varphi&=\int \bar A^T\nabla_x\Big(\int \nabla_y\bar\E(y,x)\cdot \vec g (y)dy\Big)\cdot \nabla\varphi(x)dx\\
&=\int\nabla_y\Big(\int\bar A^T\nabla_x\bar \E(y,x)\cdot\nabla\varphi(x)dx\Big)\cdot \vec g(y)dy\\
&=\int\nabla_y\Big(\int\bar A^T\nabla_x\E_{\bar{A}^T}(x,y)\cdot\nabla\varphi(x)dx\Big)\cdot \vec g(y)dy\\
&=\int\nabla\varphi\cdot\vec g =0.
\end{split}
\end{equation}

Notice that for every $z\in\Rn1$ we have the elementary relation
\begin{equation}
|z|^2=\sup_{\beta\geq 0, e\in\mathbb{S}^n}2\langle e,z\rangle-\beta^2,
\end{equation}
so that, choosing $z=\bar T\nu(x)-F_S(x),$ it reads
\begin{equation}\label{legendre_transform}
\big |\bar T\nu(x)-F_S(x)\big |^2=\sup_{\beta\geq 0, e\in\mathbb{S}^n}2\langle e,\bar T\nu(x)\rangle -2\langle e, F_S(x)\rangle-\beta^2.
\end{equation}
We want to show that the argument of the supremum in the right hand side of \eqref{legendre_transform} differs from a $L_{\bar A}$-harmonic function possibly by a small term. This will allow to apply the maximum principle on the strip and to finish the proof.

For a fixed $e\in\mathbb{S}^n$ and $x\in\supp \nu$, we split
\begin{equation}
\langle e, \bar T\nu(x)\rangle =-\bar T^*(\nu e)(x)+\big(\bar T^*(\nu e)(x)+ \langle e, \bar T\nu(x)\rangle\big)
\end{equation}
and consider that, {claiming that the dominated convergence theorem applies},
\begin{equation}\label{antisym_part}
\bar T^*(\nu e)(x)+ \langle e, \bar T\nu(x)\rangle=\lim_{r\to \infty}\int \big(\tilde K_r(x,y)+\tilde{K}_r(y,x)\big)\cdot e\, d\nu(y).
\end{equation}
To prove that the previous identity holds, set $C_S\gg K_S$ to be chosen later.
By the triangle inequality, the antisimmetry of $\n1\Theta(x,y;\bar A(x))$ and the linear growth of $\nu$, we have
\begin{equation}\label{near_estim}
\begin{split}
&\int_{|x-y|<C_S \ell(Q_0)}\big|\tilde K_r(x,y)+\tilde K_r(y,x)\big| d\nu(y)\\
&\qquad\lesssim \int_{|x-y|<C_S \ell(Q_0)}\big|\bar K(x,y)-\n1{\Theta}(x,y;\bar A(x))\big| d\nu(y) \\
&\qquad\qquad+ \int_{|x-y|<C_S \ell(Q_0)}\big|\bar K(y,x)-\n1{\Theta}(y,x;\bar A(x))\big| d\nu(y)\\
&\qquad\lesssim \int_{|x-y|<C_S \ell(Q_0)}\frac{1}{|x-y|^{n-\tilde\alpha}}d\nu(y)\lesssim \big(C_S \ell(Q_0)\big)^{\tilde\alpha}.
\end{split}
\end{equation}
So, to bound \eqref{antisym_part} we have to estimate the integral on its rights hand side for $|x-y|>C_S \ell(Q_0).$
As before, by the periodicity of $M_S$ we can assume that $x_H\in [-3\ell(Q_0),3\ell(Q_0)]^n\times\{0\}.$
Hence, using arguments analogous to the ones in Lemma \ref{lemm_comparison_plus_minus}, it is possible to prove that for $y\in Q_0$ and $z_P$ such that $|x-y-z_P|>C_S \ell(Q_0),$ we have
\begin{equation}
\big|\bar K(x,y+z_P)+\bar K(x,y-z_P)\big|\lesssim \frac{(K_S \ell(Q_0))^{\tilde\alpha}}{|z_P|^{n+\tilde\alpha}+|x|^{n+\tilde\alpha}},
\end{equation}
hence, calling $\M_S$ the subset of $P\in\M$ such that $|x-y-z_P|>C_S \ell(Q_0)$ for every $y\in Q_0$, we have
\begin{equation}\label{far_estim_1}
\sum_{P\in\M_S}\int_{Q_0}|\tilde K_r(x,y+z_P)+\tilde K_r(x,y-z_P)|d\nu(y)\lesssim \Big(\frac{K_S}{C_S}\Big)^{\tilde\alpha}.
\end{equation}
Analogously, one can prove
\begin{equation}\label{far_estim_2}
\sum_{P\in\M_S}\int_{Q_0}|\tilde K_r(y+z_P,x)+\tilde K_r(y-z_P,x)|d\nu(y)\lesssim \Big(\frac{K_S}{C_S}\Big)^{\tilde\alpha},
\end{equation}
so, gathering \eqref{near_estim}, \eqref{far_estim_1} and \eqref{far_estim_2} and letting $r\to \infty$, we can use the dominated convergence theorem and we estimate \eqref{antisym_part} as
\begin{equation}\label{estim_error_adjoint}
\big|\bar T^*(\nu e)(x)+ \langle e, \bar T\nu(x)\rangle|\lesssim (C_S \ell(Q_0))^{\tilde\alpha}+ \Big(\frac{K_S}{C_S}\Big)^{\tilde\alpha}.
\end{equation}
We are now ready to proceed with the calculations for the maximum principle. Indeed, taking $x\in S$, an application of \eqref{legendre_transform} and \eqref{estim_error_adjoint} gives
\begin{equation}
\begin{split}
\big|\bar T\nu(x)-F_S(x)\big|^2+4 \bar T^*((\bar T\nu)\nu)(x)=\sup_{\beta\geq 0, e\in\mathbb{S}^n}2\langle e,\bar T\nu(x)\rangle -2\langle e, F_S(x)\rangle-\beta^2 + \bar T^*((\bar T\nu)\nu)(x)\\
\lesssim \sup_{\beta\geq 0, e\in\mathbb{S}^n}-2 \bar T^*(\nu e)(x) -2\langle e, F_S(x)\rangle-\beta^2 + \bar T^*((\bar T\nu)\nu)(x) +  (C_S \ell(Q_0))^{\tilde\alpha} + \Big(\frac{K_S}{C_S}\Big)^{\tilde\alpha}.
\end{split}
\end{equation}
Then, using the maximum principle \eqref{max_prin_t_star} we have
\begin{equation}
\begin{split}
&|\bar T\nu(x)-F_S(x)|^2+4 \bar T^*((T\nu)\nu)(x)\\
&\lesssim \sup_{\beta\geq 0, e\in\mathbb{S}^n}2 -\bar T^*\big(\nu e + (\bar T\nu)\nu\big)(x) -2\langle e, F_S(x)\rangle-\beta^2 +  (C_S \ell(Q_0))^{\tilde\alpha} + \Big(\frac{K_S}{C_S}\Big)^{\tilde\alpha}\\
&\leq \sup_{z\in \partial S\cup \supp\nu}\sup_{\beta\geq 0, e\in\mathbb{S}^n}-2 \bar T^*\big(\nu e + (\bar T\nu)\nu\big)(z) -2\langle e, F_S(z)\rangle-\beta^2 +  (C_S \ell(Q_0))^{\tilde\alpha} + \Big(\frac{K_S}{C_S}\Big)^{\tilde\alpha}.
\end{split}
\end{equation}
So, another application of \eqref{legendre_transform} and \eqref{estim_error_adjoint} concludes the proof of the lemma. Indeed, recalling the estimate \eqref{sup_on_meas_and_support} on $\partial S\cup \supp \nu$,
\begin{equation}
\begin{split}
&|\bar T\nu(x)-F_S(x)|^2+4 \bar T^*((T\nu)\nu)(x)\\
&\lesssim \sup_{z\in \partial S\cup \supp\nu}\sup_{\beta\geq 0, e\in\mathbb{S}^n}2\langle e,\bar T\nu(x)\rangle -2\langle e, F_S(x)\rangle-\beta^2 \\
&\qquad\qquad+ \bar T^*((\bar T\nu)\nu)(x)+  (C_S \ell(Q_0))^{\tilde\alpha} + \Big(\frac{K_S}{C_S}\Big)^{\tilde\alpha}\\
&\lesssim \sup_{z\in\partial S\cup \supp \nu}|\bar T\nu(z)-F_S(z)|^2 + 4\bar T^*((\bar T\nu)\nu)(z) + (C_S \ell(Q_0))^{\tilde\alpha} + \Big(\frac{K_S}{C_S}\Big)^{\tilde\alpha}\\
&\lesssim \lambda^{1/2} + \frac{1}{K_S^{2\tilde\alpha}} +\frac{1}{K_S^{\vartheta}}+ (C_S \ell(Q_0))^{\tilde\alpha} + \Big(\frac{K_S}{C_S}\Big)^{\tilde\alpha}.\qedhere
\end{split}
\end{equation}
\end{proof}
\subsection{The conclusion of the proof of the Key Lemma}
To simplify the notation, set
\begin{equation}
\texttt{Err}(K_S,C_S,\ell(Q_0))\coloneqq \frac{1}{K_S^{2\tilde\alpha}}+ \frac{1}{K_S^{\vartheta}}+(C_S\ell(Q_0))^{\tilde\alpha}+\Big(\frac{K_S}{C_S}\Big)^{\tilde\alpha}.
\end{equation}
Notice that if $x\in 2Q_0$, Lemma \ref{main_lemma_maximum_principle} together with Lemma \ref{lem_estim_FS_close} allows to majorize $|\bar T\nu(x)|^2$ as
\begin{equation}
\begin{split}\label{maj_2_cube}
\big|\bar T\nu(x)\big|^2&\lesssim |\bar T\nu(x)-F_S(x)|^2+|F_S(x)|^2+4\bar T^*((\bar T\nu)\nu)(x)-4\bar T^*((\bar T\nu)\nu)(x)\\
&\lesssim\lambda^{1/2}+ \texttt{Err}(K_S,C_S,\ell(Q_0)) +|F_S(x)|^2 - \bar T^*((\bar T\nu)\nu)(x)\\
&\lesssim \lambda^{1/2}+ \texttt{Err}(K_S,C_S,\ell(Q_0))  - \bar T^*((\bar T\nu)\nu)(x).
\end{split}
\end{equation}\label{inv_t_star}
Let $\varphi$ be a smooth function such that $\chi_{Q_0}\leq \varphi\leq\chi_{2Q_0}$ and $\|\nabla\varphi\|_{\infty}\lesssim \ell(Q_0)^{-1}$. Set $\psi\coloneqq \bar A^T\nabla\varphi$ and observe that it verifies
\begin{equation}
\begin{split}
\bar T^*[\psi \,\mathcal{L}^{n+1}](x)&=\bar T^*[\bar A^T\nabla\varphi \,\mathcal{L}^{n+1}](x)=\int\n1\E_{\bar{A}}(y,x)\cdot\bar A^T(y)\nabla\varphi(y)dy\\
&=\int \bar A(y)\n1\E_{\bar{A}}(y,x)\cdot\nabla\varphi(y)dy=\varphi(x),
\end{split}
\end{equation}
the last equality being a consequence of the definition of fundamental solution.

The choice of $\varphi\geq \chi_{Q_0}$, together with Cauchy-Schwarz's inequality, gives
\begin{equation}\label{formula_contradiction}
\begin{split}
\nu(Q_0)&\leq\int\varphi d\nu=\int \bar T^*(\psi\mathcal{L}^{n+1})d\nu=\int \bar T\nu\cdot\psi d\mathcal{L}^{n+1}\\ &\leq \Big(\int |\bar T\nu|^2|\psi|d\mathcal{L}^{n+1}\Big)^{1/2}\Big(\int|\psi|d\mathcal{L}^{n+1}\Big)^{1/2}.
\end{split}
\end{equation}
Now, observe that
\begin{equation}\label{inf_norm_psi}
\|\psi\|_\infty\leq \|\bar A^T\|_\infty \|\nabla\varphi\|_\infty\lesssim \ell(Q_0)^{-1}
\end{equation}
and
\begin{equation}\label{estim_int_psi}
\int |\psi|d\mathcal{L}^{n+1}\lesssim \frac{1}{\ell(Q_0)}\mathcal{L}^{n+1}(2Q_0)\lesssim \ell(Q_0)^n.
\end{equation}
We claim that
\begin{equation}\label{claim_contradiction}
\int |\bar T\nu|^2|\psi|d\mathcal{L}^{n+1} \ll \ell(Q_0)^n.
\end{equation}
Applying \eqref{maj_2_cube} and \eqref{estim_int_psi}, we can write
\begin{equation}\label{cont_estim_int}
\begin{split}
&\int |\bar T\nu|^2|\psi|d\mathcal{L}^{n+1}\\
&\lesssim \big(\lambda^{1/2}+ \texttt{Err}(K_S,C_S,\ell(Q_0))\big)\int|\psi|d\mathcal{L}^{n+1} + \Big|\int \bar T^*\big((\bar T\nu)\nu\big)|\psi| d\mathcal{L}^{n+1} \Big|\\
&\lesssim \big(\lambda^{1/2}+ \texttt{Err}(K_S,C_S,\ell(Q_0))\big)\int|\psi|d\mathcal{L}^{n+1} \\
&\qquad + \Big|\int \bar T^*\big(\chi_{(30Q_0)^c}(\bar T\nu)\nu\big)|\psi| d\mathcal{L}^{n+1} \Big| + \Big|\int \bar T^*\big(\chi_{30Q_0}(\bar T\nu)\nu\big)|\psi| d\mathcal{L}^{n+1} \Big|\\
&\lesssim \big(\lambda^{1/2}+ \texttt{Err}(K_S,C_S,\ell(Q_0))\big)\ell(Q_0)^n + \Big|\int \bar T^*\big(\chi_{(30Q_0)^c}(\bar T\nu)\nu\big)|\psi| d\mathcal{L}^{n+1} \Big| \\&\qquad+ \Big|\int \bar T^*\big(\chi_{30Q_0}(\bar T\nu)\nu\big)|\psi| d\mathcal{L}^{n+1} \Big|\\
&= \big(\lambda^{1/2}+ \texttt{Err}(K_S,C_S,\ell(Q_0))\big)\ell(Q_0)^n + I + II,
\end{split}
\end{equation}
where $I$ and $II$ are defined by the last equality.\\
The estimate for $I$ is an application of \eqref{loc_outside} with $\tilde{M}=30$. In particular, recalling \eqref{loc_measure_variational_argument},
\begin{equation}
\begin{split}
\big|\bar T^*\big(\chi_{(30Q_0)^c}(\bar T\nu)b\eta)\big)(x)\big|&\lesssim \frac{1}{\ell(Q_0)^n}\int_{Q_0}|(\bar T\nu) b|d\eta\\
&\leq\frac{\nu(Q_0)^{1/2}}{\ell(Q_0)^n}\Big(\int_{Q_0}|\bar T\nu|^2 d\nu\Big)^{1/2}{\leq} \lambda^{1/2}\frac{\nu(Q_0)}{\ell(Q_0)^n},
\end{split}
\end{equation}
which, together with \eqref{estim_int_psi}, implies
\begin{equation}\label{cont_estim_I}
I\lesssim \lambda^{1/2} \nu(Q_0).
\end{equation}
For the estimate of $II$, recall that $|\bar K(x,y)|\lesssim |x-y|^{-n}.$ This and \eqref{inf_norm_psi} imply
\begin{equation}
\big|\bar T\big(|\psi|\mathcal{L}^{n+1}\big)(x)\big|=\Big|\int \bar K(x,y)|\psi|(y)dy\Big|\lesssim \frac{1}{\ell(Q_0)}\int_{2Q_0}\frac{1}{|x-y|^n}dy\lesssim \frac{\ell(Q_0)}{\ell(Q_0)}=1.
\end{equation}
Then, by Cauchy-Schwarz's inequality, the periodicity of $\bar T\nu$ and the localization \eqref{loc_measure_variational_argument},
\begin{equation}\label{cont_estim_II}
II\leq \Big|\int_{30 Q_0}\bar T(|\psi|\mathcal L^{n+1})\cdot \bar T \nu d\nu\Big|\lesssim \nu(Q_0)^{1/2}\Big(\int_{30 Q_0}|\bar T\nu|^2d\nu\Big)^{1/2}\lesssim \lambda^{1/2} \nu(Q_0).
\end{equation}
So, gathering \eqref{formula_contradiction}, \eqref{cont_estim_int}, \eqref{cont_estim_I} and \eqref{cont_estim_II}, we have\begin{equation}\label{con_fin_ineq}
\nu(Q_0)\lesssim \big(\texttt{Err}(K_S,C_S,\ell(Q_0)) + \lambda^{1/2}\big)^{1/2}\nu(Q_0).
\end{equation}
Choosing $K_S$ big enough, $K_S/C_S$ small enough, $C_S\ell(Q_0)$ and $\lambda$ small enough, we have 
\begin{equation}
\texttt{Err}(K_S,C_S,\ell(Q_0)) + \lambda^{1/2}\ll 1,
\end{equation} so \eqref{con_fin_ineq} brings us to the contradiction
\begin{equation}
\nu(Q_0)\ll \nu(Q_0).
\end{equation}
This proves the Key Lemma and, hence, completes the proof of Theorem \ref{main_theorem}.

\section{The two-phase problem for the elliptic measure}

To the purpose of the application to the study of the elliptic measure, it is useful to reformulate Theorem \ref{main_theorem} under slightly different hypothesis. The proof of the following closely resembles that of \cite[Theorem 3.3]{AMT}. 
\begin{theor}\label{main_thm_new}
Let $\mu$ be a Radon measure in $\Rn1$ and let $B\subset \Rn1$ be a ball with $\mu(B)>0$. Assume that, for some constants $C_0,C_1>0$ and $0<\lambda,\delta, \tau\ll 1$ the following conditions hold:
\begin{enumerate}
\item $r(B)\leq\lambda$.
\item $P_{\mu,\tilde\alpha}(B)\leq C_0 \Theta_\mu(B)$.
\item There is some $n$-plane $L$ through the center of $B$ such that $\beta^L_{\mu,1}(B)\leq \delta \Theta_\mu(B)$.
\item There is $G_B\subset B$ such that for all $x\in G_B$
\begin{equation}
\sup_{0<r\leq 2r(B)} \frac{\mu\big(B(x,r)\big)}{r^n}+T_*(\chi_{2B}\mu)(x)\leq C_1\Theta_\mu (B).
\end{equation}
\item $\int_{G_B}|T\mu(x)-m_{\mu,G_B}|^2d\mu(x)\leq \tau \Theta_\mu(B)^2\mu(B).$
\end{enumerate}
There exists $\vartheta>0 $ such that, if $\delta,\tau$ and $\lambda$ are small enough (depending on $C_0$ and $C_1$), there is a $n$-uniformly rectifiable set $\Gamma$ such that
\begin{equation}
\mu(B\cap \Gamma)\geq \vartheta \mu(B).
\end{equation}
\end{theor}
The proof in the case $A\equiv Id$ is based on a $Tb$ theorem for \textit{suppressed kernels} by Nazarov, Treil and Volberg. To replicate the proof of Azzam, Mourgoglou and Tolsa in the elliptic context, we define the suppressed kernel associated with $K(\cdot, \cdot)$ as
\begin{equation}
\tilde{K}_\Phi(x,y)=\tilde{\chi}\Big(\frac{|x-y|^2}{\Phi(x)\Phi(y)}\Big)K(x,y),
\end{equation}
where $\tilde{\chi}\colon [0,+\infty)\to [0,1]$ is a smooth, vanishes identically in $[0,1/2]$ and equals $1$ in $[1,+\infty)$ and $\Phi$ is a $1$-Lipschitz function to be chosen as in the proof of \cite{AMT}.
Then, one can split $$K(x,y)=\frac{1}{2}\big(K(x,y)+K(y,x)\big)+ \frac{1}{2}\big(K(x,y)-K(y,x)\big)= K^{(s)}(x,y)+K^{(a)}(x,y),$$ apply the $Tb$ theorem for suppressed kernels (see also \cite[Section 5.12]{tolsa_book} and the references therein) to the antisymmetric part  of $K$ and exploit the $L^2$-boundedness of the symmetric part guaranteed by the freezing technique of Lemma \ref{lemm_freezing}. We leave to the interested reader to check that there is no further difficulty in the proof Theorem \ref{main_thm_new}.

The rest of the present section is devoted to show how to apply Theorem \ref{main_thm_new} to prove the two-phase problem for the elliptic measure.


After possibly splitting the set $E$, we can assume $\diam E\leq \frac{1}{10}\min\big(\diam \Omega_1,\diam \Omega_2\big)$. We choose the poles $p_i$, $i=1,2$ such that $p_i\in \Omega_i\cap 2\tilde{B}\setminus \tilde{B}$, where $\tilde{B}$ is a ball centered at $E$ with radius $r(\tilde{B})=2\diam E$.

We are going to apply Theorem \ref{main_thm_new} to the measure $\omega_1$: we are going to prove that we can find an $n$-rectifiable set $F\subset E$ such that $\omega_1|_F\ll \mathcal H^n|_F\ll\omega_1|_F$. In particular, we can suppose that $\Omega_1$ is such that
\begin{equation}\label{estim_Omega_1}
\mathcal H^{n+1}\big(\tilde{B}\cap \Omega_1\big)\approx r(\tilde{B}).
\end{equation}
By the so-called Bourgain's estimates (see \cite[Lemma 32]{PPT} for the statement in the elliptic case and \cite{seven_authors} for a proof in the case $A\equiv Id$) together with \eqref{estim_Omega_1}, we can infer that there exists $\delta_0$ such that \begin{equation}
\omega_1\big(2\delta^{-1}\tilde{B}\big)\approx 1, \quad \text{ for }\,0<\delta<\delta_0.
\end{equation}
Let $a,\tilde\gamma>0$ and $i=1,2$. We say that a ball $B$ is $a$-$P_{\omega_i,\tilde\gamma}$-doubling if
\begin{equation}
P_{\omega_i,\tilde\gamma}(B)\leq a \Theta_{\omega_i}(B).
\end{equation}
The following lemma is important for the applicability of the doubling condition.
\begin{lemm}
Let $\tilde{\gamma}\in (0,1)$. Let $\Omega_1,\Omega_2$ be Wiener regular domains in $\Rn1$ and let $E\subset\partial\Omega_1\cap\partial\Omega_2$ be a set on which $\omega_1|_E\ll\omega_2|_E\ll\omega_1|_E$. Then there exists a constant $a=a(\tilde\gamma, n)$ big enough such that for $\omega_1|_E$-almost every $x\in\Rn1$ we can find a sequence of $a$-$P_{\omega_i,\tilde\gamma}$-doubling balls $B(x,r_i)$ with $r_i\to 0 $ as $i\to\infty$.
\end{lemm}
\begin{proof}
Let $i=1,2$. Let $m\in\mathbb Z, m\geq 1$ and denoting
\begin{equation}
Z_m\coloneqq \big\{x\in\partial \Omega_i: \text{ for all }j\geq m, \, B(x, 2^{-j}) \, \text{ is not }a\text{-}P_{\omega_i,\tilde\gamma}\text{-doubling}\big\}
\end{equation}
it suffices to prove that $\omega_i|_E(Z_m)=0$ for every $m$. Arguing as in \cite[Lemma 6.1]{AMTV} we have that, for $x\in Z_m$, we can estimate the elliptic measure of $B(x,r)$ as
\begin{equation}
\omega_i(B(x,r))\leq C(m)r^{n+\tilde{\gamma}} \qquad \text{ for }r\leq 2^{-m}.
\end{equation}
Then
\begin{equation}
\omega|_E(A)\leq \omega(A)\leq C(m)\mathcal H^{n+\tilde{\gamma}}(A) \qquad\text{ for any } A\subset Z_m.
\end{equation}
We recall that the dimension of $\omega|_E$ can be defined as
\begin{equation}
\begin{split}
\dim \omega|_E \coloneqq \inf \big\{s\colon &\exists\,  F\subset\partial \Omega \text{ s.t. } \mathcal{H}^s(F)=0 \\&\text{ and }\omega|_E(F\cap K)=\omega|_E(\partial \Omega\cap K)\, \forall K\subset\mathbb{R}^{n+1} \text{ compact}\big\}
\end{split}
\end{equation}
First let us bound $\dim \omega|_E$ from below.
Let $F\subset\partial \Omega$ be such that  $\mathcal{H}^{n+\tilde{\gamma}}(F)=0$. For $K\subset Z_m$ compact and such that $\omega|_E(K)>0$, we have $\omega|_E(F\cap K)\leq C(m) \mathcal H^{n+\tilde{\gamma}}(F\cap K)=0$. This in turn implies
\begin{equation}\label{estim_dim_omega}
\dim \omega|_E\geq n+\tilde{\gamma}.
\end{equation} 
Conversely, \cite{AM2} gives that $\dim \omega|_E = n$, which gathered with \eqref{estim_dim_omega} tells that
\begin{equation}
n\geq n+\tilde \gamma.
\end{equation}
Being $\tilde\gamma>0$, this brings to a contradiction and, in particular, this proves that $\omega(Z_m)=0$ for every $m$.
\end{proof}

Let $i=1,2$. Denote by $u_i(\cdot)=G_i(p_i,\cdot)$ the Green function associated with $\Omega_i$ with pole at $p_i$. We understand that $u_i$ is extended by zero to $\Omega_i^c$.
As a corollary of \cite[Theorem 1.5]{AGMT}, which was formulated under weaker assumptions on the regularity of the matrix $A$, we can state the following monotonicity formula.
\begin{lemm}[Monotonicity formula]\label{monotonicity}
Let $\Omega_i$ and $u_i$ be as above and let $R>0$. Suppose that that $A_s(\xi)=Id$ for $\xi\in \partial\Omega_1\cap \partial\Omega_2$. Then, setting
\begin{equation}
\gamma(\xi,r)=\Big(\frac{1}{r^2}\int_{B(\xi,2r)}\frac{|\nabla u_1(y)|^2}{|y-\xi|^{n-1}}dy\Big)\cdot \Big(\frac{1}{r^2}\int_{B(\xi,2r)}\frac{|\nabla u_2(y)|^2}{|y-\xi|^{n-1}}dy\Big),
\end{equation} 
we have that, for some $c>0$, 
\begin{equation}
\gamma(\xi,r)\leq \gamma(\xi,s)e^{c(s^\alpha-r^\alpha)}<\infty \quad \text{ for }0<r\leq s<R.
\end{equation}
\end{lemm}
We remark that Azzam, Garnett, Mourgoglou and Tolsa proved their result under the hypothesis $A(\xi)=Id$. However, the same proof works under our assumption.\footnote{It suffices to define the matrix $D$ in \cite[Appendix A.1]{AGMT} as $D=A(\xi)-A$ and observe that $L_{A(\xi)}=L_{A_s(\xi)}=Id$.}

The following lemma is crucial to prove the elliptic variant version of the blowups.
\begin{lemm}\label{lemm_estim_omega_delta}Let $\Omega_1$ be a Wiener regular domain and denote by $\omega_1=\omega_{1}^{p_1}$ its associated elliptic measure with pole at $p_1 \in \Omega_1$. Let $B$ be a ball centered at $\partial\Omega_1$ and such that $p_1\not\in 10 B$.
Assuming that $\omega_1(8B)\leq C \omega_1(\delta_0 B)$ and $\mathcal H^{n+1}(B\setminus \Omega_1)\geq C^{-1}r(B)^{-1}$, we have
\begin{equation}\label{estim_3_3}
\mathcal H^{n+1}(\Omega_1\cap 2\delta_0 B)\gtrsim  r(B)^{n+1}.
\end{equation}
Moreover
\begin{equation}\label{estim_3_4}
\mathcal H^{n+1}(2\delta_0 B\setminus\Omega_1)\approx\mathcal H^{n+1}(2\delta_0 B\setminus \Omega_2)\approx r(B)^{n+1}.
\end{equation}
\end{lemm}

\begin{proof}
Denote $r=r(B)$. Let us first prove \eqref{estim_3_3}.
Consider a smooth function $\varphi\geq 0$ such that $\varphi\equiv 1$ on $\delta_0 B$ and $\supp \varphi\subset 2\delta_0 B$. In particular, suppose that $\|\varphi\|_\infty\lesssim (\delta_0 r)^{-1}$. Then, recalling that, by the properties of Green's function and being $x_1$ outside of the support of $\varphi$,
\begin{equation}
\int \varphi d\omega_1=-\int A^T \nabla u_1\cdot\nabla\varphi,
\end{equation}
we use the ellipticity of the matrix $A$ and write
\begin{equation}
\begin{split}
\omega_1(2\delta_0 B)&\leq \int \varphi d\omega_1\leq \int |\nabla u_1\cdot A\nabla \varphi|\\ &\lesssim \int |\nabla u_1| |\nabla\varphi| = \int_{\Omega_1\cap 2\delta_0 B} |\nabla u_1| |\nabla\varphi|\lesssim \frac{1}{\delta_0 r}\int_{\Omega_1\cap 2\delta_0 B} |\nabla u_1|.
\end{split}
\end{equation}
Then applying, in order,  H\"older's and Caccioppoli's inequalities,
\begin{equation}
\begin{split}
\frac{1}{\delta_0 r}\int_{\Omega_1\cap 2\delta_0 B} |\nabla u_1|&\leq \frac{\mathcal H^{n+1}(\Omega_1\cap 2\delta_0 B)^{1/2}}{\delta_0 r}\Big(\int_{2\delta_0 B}|\nabla u_1|^2\Big)^{1/2}\\ &\lesssim \frac{\mathcal H^{n+1}(\Omega_1\cap 2\delta_0 B)^{1/2}}{\delta_0 r}\frac{1}{\delta_0 r}\Big(\int_{4\delta_0 B}|u_1|^2\Big)^{1/2},
\end{split}
\end{equation}
so
\begin{equation}
\omega_1(2\delta_0 B)\lesssim {\mathcal H^{n+1}(\Omega_1\cap 2\delta_0 B)^{1/2}}\frac{(\delta_0 r)^{(n+1)/2}}{(\delta_0 r)^2}\sup_{4\delta_0 B}|u_1|.
\end{equation}
At this point, recalling that (see \cite[Lemma 32]{PPT})
\begin{equation}
\sup_{y\in 4\delta_0 B}u_1(y)\lesssim \frac{\omega_1(8B)}{r^{n-1}},
\end{equation}
we have
\begin{equation}
\omega_1(\delta_0 B)\lesssim {\mathcal H^{n+1}(\Omega_1\cap 2\delta_0 B)^{1/2}}\frac{(\delta_0 r)^{(n-3)/2}}{(\delta_0 r)^{n-1}}\omega_1(8 B)
\end{equation}
which, since we suppose $\omega_1(8B)\leq C \omega_1(\delta_0 B)$, concludes the proof of \eqref{estim_3_3}.

The second estimate in the statement of the lemma is a direct application of the first one (see also \cite[Lemma 3.4]{AMTV}).
\end{proof}
The following lemma provides the connection between the function $\gamma$ in Lemma \ref{monotonicity} and elliptic measure. 
\begin{lemm}\label{transition_monotonicity_elliptic_measure}
Let $i=1,2$ and $\Omega_i$, $p_i$ be as above. Let $0<R<\min_i\dist(p_i,\partial\Omega_i).$ Then, for $0<r<R/4$ and $\xi\in\partial\Omega_1\cap\partial\Omega_2$ we have
\begin{equation}\label{mon_1}
\frac{\omega_1(B(\xi,r))}{r^n}\frac{\omega_2(B(\xi,r)}{r^n}\lesssim \gamma(\xi,2r)^{1/2}.
\end{equation}
Moreover, if $r<\delta_0R/8$ and $\omega_i(B(\xi,8r))\lesssim \omega_i(B(\xi,\delta_o r))$,
\begin{equation}\label{mon_2}
\gamma(\xi,r)^{1/2}\lesssim \frac{\omega_1(B(\xi,16\delta_0^{-1}))}{r^n}\cdot \frac{\omega_2(B(\xi,16\delta_0^{-1}))}{r^n}.
\end{equation}
\end{lemm}
The proof of \eqref{mon_1} is analogous to that for the harmonic measure in \cite{KPT}. The proof of \eqref{mon_2} is an application of Caccioppoli's inequality together with Lemma \ref{lemm_estim_omega_delta} (see also \cite[Lemma 3.5]{AMTV}).

The blowup technique for the elliptic measure developed in \cite{AM2} is crucial to prove the next lemma. We remark that the authors formulated this result under more general assumptions on the matrix $A$ then the ones of the present work.
\begin{lemm}
Let $\Omega_1, \Omega_2$ and $E$ be as above. Let $\varepsilon<1/100$ and, for $m\geq 1$, define $E_m$ as the set of $\xi\in E$ such that
for all $\xi \in E$, $0<r<1/m$ and $i=1,2$ the following properties hold:
\begin{enumerate}[label=(\subscript{E}{{\arabic*}})]
\item $\omega_i(B(\xi,2r))\leq m\, \omega_i(B(\xi,r)).$
\item $\mathcal H^{n+1}(B(\xi,r)\cap \Omega_i)\geq \frac{1}{m}r^{n+1}.$
\item $\beta_{\omega_1,1}(B(\xi,r))<\varepsilon r^{-n} \omega_1(B(\xi,r))$.
\end{enumerate}
The sets $E_m$ cover $E$ up to a set of $\omega_1$-measure $0$, i.e.
\begin{equation}
\omega_1\Big(E\setminus \bigcup_{m\geq 1} E_m\Big)=0.
\end{equation}
\end{lemm}
 The proof follows by known results in the literature. However, we think that it may be useful to the reader to dispose of precise references.
\begin{proof}[Sketch of the proof]
Set
\begin{equation}
E^*=\Big\{\xi \in E: \lim_{r\to 0}\frac{\omega_1(E\cap B(\xi,r))}{\omega_1(B(\xi,r))}=\lim_{r\to 0}\frac{\omega_2(E\cap B(\xi,r))}{\omega_2(B(\xi,r))}=1\Big\}.
\end{equation}
One can see that $\omega_i(E\setminus E^*)=0$, $i=1,2$. Now, for $\xi\in E^*$, set $h(\xi)=\frac{d\omega_1}{d\omega_2}(\xi)$, 
\begin{equation}
\Lambda=\big\{\xi \in E^*: 0<h(\xi)<\infty \big\}
\end{equation}
and
\begin{equation}
\Gamma = \{\xi \in \Lambda: \xi \text{ is a Lebesgue point for $h$ with respect to $\omega_1$}\}.
\end{equation}
By Lebesgue differentiantion theorem, $\omega_i(E\setminus \Gamma)=\omega_i(E^*\setminus \Gamma)$ for $i=1,2$.
Then, in order to prove the lemma it suffices to show that for $\omega_1$-almost every $\xi \in \Gamma$:
\begin{enumerate}[label=(\subscript{$P$}{{\arabic*}})]
\item \label{P1} $\omega_1$ is locally doubling, i.e.
\begin{equation}
\limsup_{r\to 0}\frac{\omega_1(B(\xi,2r))}{\omega_1(B(\xi,r))}<\infty.
\end{equation}
\item \label{P2} For $i=1,2$
\begin{equation}
\liminf_{r\to 0} \frac{\mathcal H^{n+1}(B(\xi,r)\cap \Omega_i)}{r^{n+1}}>0
\end{equation}
\item \label{P3} We have the flatness estimate
\begin{equation}
\lim_{r\to 0}\beta_{\omega_1,1}(B(\xi, r))\frac{r^n}{\omega_1(B(\xi,r))}=0.
\end{equation}
\end{enumerate}
The condition (P1) holds because of the flatness of the tangents $\text{Tan}(\omega_i, \xi)$, see \cite[Theorem 1.3]{AM2}, which is known to imply the locally doubling condition (\cite[Corollary 2.7]{Pr}).

The property (P2) follows by the arguments in \cite{AMTV} together with \eqref{estim_3_4}.

To prove (P3), it suffices to argue as in the end of \cite[Section 5]{AMTV}.
\end{proof}
Now consider $m\geq 1$ such that $\omega_i(E_m)$.
\begin{lemm}
Let $\delta>0$. For $\omega_1$-almost every $x\in E_m$ there is $r_x>0$ such that, given an $a$-$P_{\tilde{\gamma},\omega_1}$-doubling ball $B(x,r)$ with $r\leq r_x$, there exits a set $G_m(x,r)\subset E_m\cap B(x,r)$ such that
\begin{equation}
\frac{\omega_1(B(z,t))}{t^n}\lesssim \frac{\omega_1(B(x,r))}{r^n} \text{ for every }z\in G_m(x,r),\, 0<t\leq 2r.
\end{equation}
In particular,
\begin{equation}\label{delta_ineq_cover}
\omega_1(B(x,r)\setminus G_m(x,r))\leq \delta \omega_1(B(x,r)).
\end{equation}
and, if we denote by $\tilde{E}_{m\delta}$ the set of points where \eqref{delta_ineq_cover} is verified, we have
\begin{equation}
\omega_1\big(E_m\setminus \tilde{E}_{m,\delta}\big)=0.
\end{equation}
\end{lemm}
This lemma can be proved arguing as in \cite[Lemma 6.2]{AMTV} and more precisely combining the locally doubling property of the elliptic measure ensured by the blowup argument together with Lemma \ref{transition_monotonicity_elliptic_measure}.

We also point out that their argument relies on the monotonicity formula of Alt, Caffarelli and Friedman. So, to prove it in the elliptic case we have to invoke Lemma \ref{monotonicity}, whose hypothesis include the assumption $A_s(x)=Id$. This, of course, is not true in general. However, one can argue via the change of variable in Lemma \ref{lemma_change_matrix_identity} to achieve this property. For a more detailed treatment of how the elliptic measure varies under that transformation we refer to \cite[Corollary 2.5]{AGMT}. We omit further details.

From now on fix $\tilde \gamma=\tilde\alpha$. The following lemma contains an estimate of the potential of $\omega_1$ which is needed to recollect the property (4) in Theorem \ref{main_thm_new}.
\begin{lemm}[cfr. {\cite[Lemma 6.3]{AMTV}}]\label{estim_T_dens}
Let $0<c\ll 1$ to be chosen small enough. For $m\geq 1$ and $\delta>0$, let $\tilde{E}_{m,\delta}$ and $r_{x_0}$ be as in the previous lemma. Consider $x_0\in \tilde{E}_{m,\delta}$ and take
\begin{equation}
0<r_0<\min\big(r_{x_0},1/m, \dist(p_1,\partial\Omega_1)\big).
\end{equation}
Assume, moreover, that $B_0=B(x_0,r)$ is an $a$-$P_{\omega_1,\tilde\alpha}$-doubling ball. Then, for all $x\in G_m(x_0,r_0)$ we have
\begin{equation}
T_*(\chi_{2B_0}\omega_1)(x)\lesssim \Theta_{\omega_1}(B_0).
\end{equation}
\end{lemm}
\begin{proof}
Suppose $A_s(x_0)=Id$. Indeed, if this is not the case, one can argue via a change of variable as mentioned before. Also, without loss of generality, we can consider only the case $r\leq r_0/4.$

Let $\varepsilon >0$. The proof relies on the estimates for the smoothened potential
\begin{equation}
\tilde{T}_\varepsilon \omega_1(z)\coloneqq \int K(z,y)\varphi_\varepsilon (z-y)d\omega_1(y), \qquad z\in\Rn1,
\end{equation}
where $\varphi\colon\Rn1\to [0,1]$ is a smooth radial function whose support is contained in $\Rn1\setminus B(0,1)$, equals $1$ on $\Rn1\setminus B(0,2)$ and $\varphi_\epsilon$ denotes the dilate $\varphi_\varepsilon(z)=\varphi(\epsilon^{-1}z).$

Now take $x\in G_m(x_0,r_0)$ and considering $r\leq r_0/4$ and define
\begin{equation}\label{v_r_ell}
v_r(z)=\mathcal E(p_1,z)-\int \mathcal E(z,y)\varphi_r(x-y)d\omega_1(y), \qquad z\in \Rn1\setminus [\supp(\varphi_r(x-\cdot)\omega_1)\cup\{p_1\}].
\end{equation}
Recall that $A_s(x_0)=Id$ and that $\Theta(\cdot;A(x_0))=\Theta(\cdot;A_s(x_0))$. On the same range of $z$ of \eqref{v_r_ell} we consider
\begin{equation}
\bar v_r(z)= \Theta(p_1-z;Id)-\int \Theta (z-y;Id)\varphi_r(x-y)d\omega_1(y).
\end{equation}
As in \cite[Lemma 6.3]{AMTV}, to prove the lemma it suffices to show the validity of the estimate
\begin{equation}
|\tilde T _r\omega_1(x)-\tilde T _{r_0/4}\omega_1(x)|\lesssim \Theta_{\omega_1}(B_0).
\end{equation}
To this purpose, observe that
\begin{equation}
\begin{split}
|\tilde T _r\omega_1(x)-\tilde T _{r_0/4}\omega_1(x)|&=|\nabla v_r(x)-\nabla v_{r_0/4}(x)|\\
&=\Big|\int \n1 \E(x,y)\big(\varphi_r(x-y)-\varphi_{r_0/4}(x-y)\big)d\omega_1(y)\Big|.
\end{split}
\end{equation}
Now, using Lemma \ref{lemm_freezing} and the H\"older continuity of $A$, it is not difficult (recall that $r_0\leq 1$) to prove that
\begin{equation}
|\n1 \E(x,y)-\n1 \Theta(x-y;Id)|\lesssim \frac{r_0^{\tilde\alpha}}{|x-y|^n}\leq \frac{1}{|x-y|^n},
\end{equation}
which in turn implies
\begin{equation}\label{ttilde1}
\begin{split}
|\tilde T _r\omega_1(x)-\tilde T _{r_0/4}\omega_1(x)|&\lesssim \Theta_{\omega_1}(B_0)+\Big|\int \n1 \Theta(x-y;Id)\big(\varphi_r(x-y)-\varphi_{r_0/4}(x-y)\big)d\omega_1(y)\Big|\\
&= |\nabla\bar{v}_r(x)-\nabla\bar v_{r_0/4}(x)|+ \Theta_{\omega_1}(B_0).
\end{split}
\end{equation}
We claim that $|\bar{v}_r(x)-\bar v_{r_0/4}(x)|\lesssim  \Theta_{\omega_1}(B_0),$ which would conclude the proof.
To show this, notice that functions $\bar v_r$ and $\bar v_{r_0/4}$ are harmonic outside $\supp(\varphi_r(x-\cdot)\omega_1)\cup \{p_1\}$, hence in particular in $B(x,r)$. Then, an application of the mean value property gives
\begin{equation}\label{ttilde2}
|\nabla\bar{v}_r(x)-\nabla\bar v_{r_0/4}(x)|\lesssim \frac{1}{r}\fint_{B(x,r)}|\bar{v}_r(z)-\bar v_{r_0/4}(z)|dz.
\end{equation}
Another application of the freezing argument together with the $C^{\tilde\alpha}$-continuity of $A$ proves
\begin{equation}
|\bar{v}_r(z)-\bar v_{r_0/4}(z)-v_r(z)-v_{r_0/4}(z)|\lesssim  r_0^{\tilde\alpha} r\, \Theta_{\omega_1}(B_0), \qquad z\in B(x,r)
\end{equation}
that, gathered with \eqref{ttilde1} and \eqref{ttilde2} gives
\begin{equation}
\begin{split}
|\tilde T _r\omega_1(x)-\tilde T _{r_0/4}\omega_1(x)|&\lesssim \Theta_{\omega_1}(B_0)+\frac{1}{r}\fint_{B(x,r)}|v_r(z)-v_{r_0/4}(z)|dz \\
&\leq \Theta_{\omega_1}(B_0)+\frac{1}{r}\fint_{B(x,r)}|v_r(z)| dz + \frac{1}{r}\fint_{B(x,r)} |v_{r_0/4}(z)|dz.
\end{split}
\end{equation}
From this point on, the proof is analogous to that in \cite{AMTV}.
\end{proof}
The proof of the Theorem \ref{theorem_two_phase_elliptic} follows the footprints of that of \cite{AMT} and \cite{AMTV}. More precisely, taking $x_0\in \tilde E_{m,\delta}$ and $r_0$ as in Lemma \ref{estim_T_dens}, we split the set $G_m(x_0,r_0)$ as a union of
\begin{equation}
G^{zd}_m(x_0,r_0)=\big\{x\in G_m(x_0,r_0):\lim_{r\to 0}\Theta_{\omega_1}\big(B(x,r)\big)=0\big\}
\end{equation}
and
\begin{equation}
G^{pd}_m(x_0,r_0)=G_m(x_0,r_0)\setminus G^{zd}_m(x_0,r_0).
\end{equation}
Then, using Lemma \ref{estim_T_dens}, the elliptic analogue of \cite[Lemma 6.5]{AMTV} and the rectifiability Theorem \ref{main_thm_new} that we proved in the present paper, it is possible to infer that
\begin{equation}
\omega_1(G^{zd}_m(x_0,r_0))=0.
\end{equation}
On the other side, \cite[Theorem 3]{PPT} ensures the existence of an $n$-rectifiable set $F(x_0,r_0)\subset G^{pd}_m(x_0,r_0)$ of mutual absolute continuity of the elliptic measure $\omega_1|_{F(x_0,r_0)}$ and the Hausdorff measure $\mathcal H^n|_{F(x_0,r_0)}$ that covers $G_m(x_0,r_0)$ up to a $\omega_1$-null set. This concludes the proof of Theorem \ref{theorem_two_phase_elliptic}.

\label{Bibliography}

\end{document}